\documentclass[a4paper]{article}

\usepackage{graphicx}
\usepackage{amssymb}
\usepackage{amsmath}
\usepackage{amsthm}
\usepackage{amsfonts}
\usepackage{mathtools}
\usepackage{hyperref}
\usepackage[shortlabels]{enumitem}
\usepackage{tikz}
\pgfdeclarelayer{background}
\pgfsetlayers{background,main}
\usetikzlibrary{shapes.geometric}
\usetikzlibrary{shapes.symbols}
\usetikzlibrary{positioning}
\usetikzlibrary{patterns.meta}
\usepackage[hypcap=false]{caption}
\usepackage{etoolbox}
\usepackage{environ}
\usepackage{MnSymbol}
\usepackage{bbm}
\usepackage{dsfont}
\usepackage{comment}
\usepackage[authoryear, round]{natbib}
\usepackage{float}
\usepackage{authblk}

\theoremstyle{plain}
\newtheorem{theorem}{Theorem}[section]
\newtheorem{lemma}[theorem]{Lemma}
\newtheorem{corollary}[theorem]{Corollary}
\newtheorem{proposition}[theorem]{Proposition}

\theoremstyle{definition}
\newtheorem{definition}[theorem]{Definition}
\newtheorem*{*definition}{Definition}

\newtheorem{remark}[theorem]{Remark}
\newtheorem{assumption}[theorem]{Assumption}

\newtheorem{example}[theorem]{Example}

\numberwithin{figure}{section}

\newcommand{\mc}[1]{\mathcal{#1}}
\newcommand{\mbb}[1]{\mathbb{#1}}

\newcommand{\proj}[1]{\textrm{proj}_{#1}}

\newcommand{\Id}[1]{\textrm{Id}_{#1}}
\newcommand{\pushforward}[2]{#1{}_{\#}#2} 

\newcommand{\Cb}[1]{C_{b}\left(#1\right)}
\newcommand{\eps}{\varepsilon}
\newcommand{\brackets}[1]{\left(#1\right)}

\newcommand{\curlybrackets}[1]{\left\{#1\right\}}
\newcommand{\probabilities}[1]{\mc{P}\brackets{#1}}
\newcommand{\pprobabilities}[2]{\mc{P}_{#1}\brackets{#2}}

\newcommand{\couplings}[2]{\Pi\brackets{#1,#2}}

\newcommand{\abs}[1]{\left|#1\right|}
\newcommand{\Pp}[2]{\mc{P}_{#1}\brackets{#2}}

\newcommand{\expectation}[2]{\mathbb{E}_{#1}\left[#2\right]}
\newcommand{\norm}[1]{\left\lVert#1\right\rVert}

\newcommand{\sqnorm}[1]{\norm{#1}_{2}}

\newcommand{\supp}{\mathrm{supp} \,}

\newcommand{\law}[1]{\mc{L}\brackets{#1}}
\newcommand{\condlaw}[2]{\mc{L}\brackets{#1 \, \middle| \, #2}}

\newcommand{\setdef}[2]{\curlybrackets{#1 \, \middle| \, #2}}

\newcommand{\hist}[1]{\mathrm{an}_{#1}}
\newcommand{\parents}[1]{\mathrm{pa}_{#1}}
\newcommand{\desc}[1]{\mathrm{dsc}_{#1}}
\newcommand{\prox}{\mathrm{prox}}
\newcommand{\dpp}{\mathrm{DPP}}
\newcommand{\depth}{\mathrm{depth}}
\newcommand{\kancestors}[2]{\mathrm{pa}^{#1}_{#2}}
\newcommand{\Gcouplings}[2]{\Pi_{G}^{c}\brackets{#1,#2}}
\newcommand{\specificGcouplings}[3]{\Pi_{#1}^{c}\brackets{#2,#3}}
\newcommand{\Gpluscouplings}[2]{\Pi_{G^+}^{c}\brackets{#1,#2}}
\newcommand{\Ganticouplings}[2]{\Pi_{G}^{ac}\brackets{#1,#2}}
\newcommand{\Gbicouplings}[2]{\Pi_{G}^{bc}\brackets{#1,#2}}
\newcommand{\bicouplings}[2]{\Pi^{bc}\brackets{#1,#2}}
\newcommand{\specificGbicouplings}[3]{\Pi_{#1}^{bc}\brackets{#2,#3}}
\newcommand{\indep}{\mathrel{\text{\scalebox{1.07}{$\perp\mkern-10mu\perp$}}}}
\newcommand{\condindep}[3]{#1 \underset{#2}{\indep} #3}
\newcommand{\Gprobabilities}[1]{\mc{P}^{G}\brackets{#1}}
\newcommand{\causalG}{G^c}
\newcommand{\cGprobabilities}[1]{\mc{P}^{\causalG}\brackets{#1}}
\newcommand{\GOT}[2]{\mathbf{P}_G(#1,#2)}
\newcommand{\specificGprobabilities}[2]{\mc{P}^{#1}\brackets{#2}}
\newcommand{\specificGpprobabilities}[3]{\mc{P}^{#1}_{#2}\brackets{#3}}
\newcommand{\Gpprobabilities}[2]{\mc{P}_{#1}^{G}\brackets{#2}}
\newcommand{\xparents}[1]{x_{\mathrm{pa}_{#1}}}
\newcommand{\yparents}[1]{y_{\mathrm{pa}_{#1}}}

\newcommand{\glue}{\mathbin{\mathop{\circ}\limits_{\raisebox{1ex}{$\scriptstyle \mathfrak{g}$}}}}
\newcommand{\otimesdot}{\dot{\otimes}}
\newcommand{\bigglue}[2]{#1 \, \otimesdot \, #2}
\newcommand{\Garrow}{G_{2}}
\newcommand{\Gsplit}{G_{\mathrm{split}}}
\newcommand{\Gmerge}{G_{\mathrm{merge}}}
\newcommand{\Gmarkov}{G_{\mathrm{chain}}}
\newcommand{\Glinear}[1]{G_{#1}}

\newcommand{\condexpectation}[2]{\mathbb{E}\left[#1 \mvert #2\right]}

\newcommand{\isGadapted}{\triangleleft_{G}}

\newcommand{\specificGmonge}[3]{\mc{T}_{#1}\brackets{#2,#3}}
\newcommand{\Gmonge}[2]{\mc{T}_{G}\brackets{#1,#2}}
\newcommand{\monge}[2]{\mc{T}\brackets{#1,#2}}

\newcommand{\newGbicouplings}[2]{\widehat{\Pi}_{G}^{bc}\brackets{#1,#2}}
\newcommand{\tildeGbicouplings}[2]{\widetilde{\Pi}_{G}^{bc}\brackets{#1,#2}}
\newcommand{\newspecificGbicouplings}[3]{\widehat{\Pi}_{#1}^{bc}\brackets{#2,#3}}
\newcommand{\tildespecificGbicouplings}[3]{\widetilde{\Pi}_{#1}^{bc}\brackets{#2,#3}}

\newcommand{\closedconvex}[1]{\overline{\operatorname{conv}}\brackets{#1}}
\newcommand{\closure}[1]{\overline{#1}}
\newcommand{\Gbar}{\closure{G}}

\newcommand{\singlebar}[1]{\underline{#1}}
\newcommand{\doublebar}[1]{\underline{\underline{#1}}}

\newcommand{\setcomplement}[1]{#1^{\mathsf{c}}}

\newcommand{\tildeyparents}[1]{\tilde{y}_{\mathrm{pa}_{#1}}}

\newcommand{\metricplaceholder}[2]{\sum_{i=1}^{n-1} \singlebar{d}_i(#1_i, #2_i)}

\newcommand{\doublehat}[1]{\Hat{\Hat{{#1}}}}
\newcommand{\specificGparents}[2]{\mathrm{pa}^{#1}_{#2}}

\title{Graph Causal Optimal Transport \\ and Wasserstein Distances} 

\author{Jan Ob{\l}{\'o}j\thanks{\href{mailto:jan.obloj@maths.ox.ac.uk}{jan.obloj@maths.ox.ac.uk}} } 
\author{Vlad Tuchilu\textcommabelow{s}\thanks{\href{mailto:vlad.tuchilus@maths.ox.ac.uk}{vlad.tuchilus@maths.ox.ac.uk}}}

\affil{Mathematical Institute, University of Oxford}

\begin{document}

\maketitle

\begin{abstract}
    We study the graph causal optimal transport problem, a generalisation of the classical optimal transport problem in which the allowed couplings satisfy causal restrictions prescribed by a directed graph. We characterise fully the directed acyclic graphs for which the associated graph causal Wasserstein discrepancy is a metric and show that the induced topology agrees with other natural adapted topologies. We characterise the gluing properties of graph causal couplings, prove denseness of Monge couplings, and obtain a dynamic programming principle which allows us to deduce when the graph causal Wasserstein and the adapted Wasserstein distances are equal. Our results link fundamental properties of graph causal optimal transport to structural properties of its underlying graph.     
    Complementing \cite{Eckstein2023CausalGraphs}, who first introduced such distances and established Lipschitz continuity for the average treatment effect in structural causal models, we obtain Lipschitz continuity of the value function in stochastic team problems.  
\end{abstract}

\section{Introduction}

This paper offers a contribution at the crossroads of optimal transport, probability and statistics. 
We define and study optimal transport (OT) problems where the marginals, and their couplings, are constrained to be compatible with causality relations encoded via a directed graph $G$.   
The associated graph causal couplings and Wasserstein distances were first introduced by \cite{Eckstein2023CausalGraphs}, who were inspired by the notion of causality \citep{Pearl2009Causality} in statistics and established Lipschitz continuity of average treatment effects with respect to these distances. Such causal restrictions on marginals we consider arise in graphical causal models, which are important objects across many fields of applications, including medicine, law, economics and beyond. However, apart from important special subclasses of problems, as discussed below, general theoretical properties of graph causal OT, including when and if the associated Wasserstein discrepancies are in fact distances, remain open. Our work addresses this key gap. We establish fundamental properties of graph causal OT problems and the induced Wasserstein distances. 

Optimal transport methods offer a natural way to measure distances between probability measures, one which generates the topology of weak convergence whilst being compatible with analytic and geometric analysis. OT theory offers rich and powerful insights into structural properties of spaces of probability measures, their geodesics and gradient flows, as well as how properties of their functionals reflect the geometry of the underlying state space, see \cite{Villani2003TopicsTransportation,ambrosio2005gradient,Villani2009OptimalNew} for  excellent exposition of these seminal contributions. However, classical OT approaches are not well suited to distances between laws of stochastic processes, where the information structure, often encoded via a filtration, plays an important role. Over the years, many adapted topologies have been introduced to address the inadequacy of the topology of weak convergence when dealing with processes. The tools of optimal transport were also brought into the fray leading to the definition of causal OT and so-called \emph{adapted Wasserstein distances}, see \cite{Pflug2009NestedDistributions,Lassalle2013Causal,Backhoff2016CausalDiscrete}. Importantly, \cite{BackhoffVeraguas2020AllEqual} proved that essentially all of the proposed adapted topologies were equal, and adapted OT has become an active field of research, with some results mirroring classical OT and others showcasing genuinely novel challenges and types of results. Our work is inspired by this recent stream of results, but offers a structurally much richer framework and, we believe, opens up new directions for future studies and applications. Indeed, 
from the point of view of the theory we advance here, classical OT corresponds to the special case of a fully connected graph and adapted OT corresponds to the linear graph in which each node has incoming edges from all of its ancestors. Naturally, there are many other graphs and varying the causality graph gives rise to a wide range of intermediate problems, allowing us to 'interpolate' between classical and causal OT, and to generalise beyond them.

Broadly speaking, we show that the key properties of graph causal OT reflect the structural properties of directed acyclic graphs (DAGs) defined in relation to absence or presence of three basic local structures: 
\begin{itemize}
    \item a \emph{chain} of three or more nodes in a linear sequence;
    \item a \emph{split}, also known as an \emph{unshielded fork};
    \item and a \emph{merge}, also known as a \emph{v-structure}, an \emph{immorality} or an \emph{unshielded collider};
\end{itemize}
see Definition \ref{def:DAGthree_structures}, and Figure \ref{fig:basic_structures} for a graphical representation.
\begin{figure}[h]
    \centering
    \begin{tikzpicture}[
    >=stealth,
    every node/.style={circle, fill=black, inner sep=2.5pt},
    blacknode/.style={circle, fill=black, inner sep=2.5pt},
    blackedge/.style={->, draw=black, thick},
    contextedge/.style={->, draw=gray!55, thick, densely dotted},
    nodelabel/.style={font=\small, fill=none, inner sep=1pt},
    paneltitle/.style={font=\large, fill=none, inner sep=0pt}
]

\begin{scope}[local bounding box=chainbox]
    \node[blacknode] (c1) at (0,0) {};
    \node[blacknode] (c2) at (1.35,0) {};
    \node[blacknode] (c3) at (2.70,0) {};
    \coordinate (cin) at (-1.20,0);
    \coordinate (cout) at (3.90,0);

\draw[contextedge] (cin) -- (c1);
\draw[contextedge] (c3) -- (cout);

    \draw[blackedge] (c1) -- (c2);
    \draw[blackedge] (c2) -- (c3);

    \node[nodelabel] at (0,0.32) {$i$};
    \node[nodelabel] at (1.35,0.32) {$j$};
    \node[nodelabel] at (2.70,0.32) {$k$};
    \node[paneltitle] at (1.35,-0.72) {Chain};
\end{scope}

\begin{scope}[xshift=5.85cm, local bounding box=splitbox]
    \coordinate (sin) at (-1.20,0);
    \node[blacknode] (s1) at (0,0) {};
    \node[blacknode] (s2) at (1.45,0.62) {};
    \node[blacknode] (s3) at (1.45,-0.62) {};
    \coordinate (sout2) at (2.65,0.62);
    \coordinate (sout3) at (2.65,-0.62);

    \draw[contextedge] (sin) -- (s1);
    \draw[blackedge] (s1) -- (s2);
    \draw[blackedge] (s1) -- (s3);
    \draw[contextedge] (s2) -- (sout2);
    \draw[contextedge] (s3) -- (sout3);

    \node[nodelabel] at (0,0.32) {$i$};
    \node[nodelabel] at (1.45,0.94) {$j$};
    \node[nodelabel] at (1.45,-0.94) {$k$};
    \node[paneltitle] at (0.72,-1.45) {Split};
\end{scope}

\begin{scope}[xshift=10.45cm, local bounding box=mergebox]
    \coordinate (min2) at (-1.20,0.62);
    \coordinate (min3) at (-1.20,-0.62);
    \node[blacknode] (m2) at (0,0.62) {};
    \node[blacknode] (m3) at (0,-0.62) {};
    \node[blacknode] (m1) at (1.45,0) {};
    \coordinate (mout) at (2.65,0);

    \draw[contextedge] (min2) -- (m2);
    \draw[contextedge] (min3) -- (m3);
    \draw[blackedge] (m2) -- (m1);
    \draw[blackedge] (m3) -- (m1);
    \draw[contextedge] (m1) -- (mout);

    \node[nodelabel] at (0,0.94) {$j$};
    \node[nodelabel] at (0,-0.94) {$k$};
    \node[nodelabel] at (1.45,0.32) {$i$};
    \node[paneltitle] at (0.72,-1.45) {Merge};
\end{scope}

\end{tikzpicture}
    \caption{Three important local structures arising in DAGs.}
    \label{fig:basic_structures}
\end{figure}
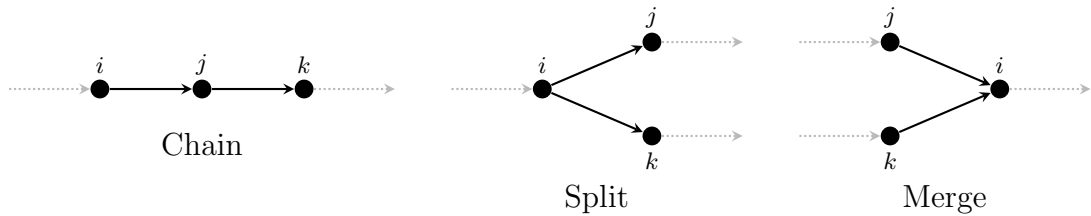
We show that, in general, the graph causal Wasserstein discrepancy is a distance for a DAG $G$ essentially if and only if $G$ does not have a chain. And for $G$ with a chain, the natural way to define the associated distance is to consider the distance generated by $\Gbar$, the transitive closure of $G$, i.e., the graph $G$ with added edges $i\to k$ for any chain $i\to j \to k$ in $G$. 
In addition, for a DAG $G$ with no chains, its set of bicausal couplings is weakly closed, as well as closed under gluing, if and only if $G$ has no splits. In this case, the sets of bicausal couplings and lifted bicausal couplings, as introduced below, are also equal. 
Finally, in a natural setting of an additive cost across the nodes, graph causal Wasserstein distances are equal to the adapted Wasserstein distance, essentially if and only if $G$ is a perfect DAG, i.e., a DAG with no merges. 

\textbf{Graph causal Wasserstein distances.}
\cite{Eckstein2023CausalGraphs} introduced 
graph causal Wasserstein discrepancies $W_{G,p}$
and used these objects but did not prove if, or when, these were distances. Indeed, they provided a numerical example to show that the triangle inequality may fail even for a simple graph made of a single chain of three nodes. In Theorem \ref{main_result_metric}, we establish that the graph causal OT problem associated with a directed acyclic graph $G$ defines a genuine distance $W_{G,p}$ on the subspace of $G$-compatible measures if and only if $G$ is transitively closed, i.e., all of a node's ancestors are its parents. 
We further show that $W_{G,p}$ are both natural for applications, e.g., stochastic team games are Lipschitz continuous, see Theorem \ref{thm:lipschitz_optimisation},  and that they generate the natural topology: the subspace topology of the adapted weak topology, see Theorem \ref{theorem:sametopology}. Finally, our subsequent results, discussed below, suggest a modified definition of the $G$-causal Wasserstein metric, $\widetilde{W}_{G,p}$, which is always a distance. We show that in fact 
$\widetilde{W}_{G,p}=W_{\Gbar,p}$, where $\Gbar$ is the transitive closure of $G$. In particular, $\widetilde{W}_{G,p}=W_{G,p}$ when $G$ is transitively closed, and $\widetilde{W}_{G,p}$ is, in some sense, the minimal object that cures the deficiency of $W_{G,p}$. 

\textbf{Graph causal OT problem and bicausal couplings.}
Our next set of results focuses on the graph causal OT problem, as well as the spaces of $G$-compatible measures $\Gprobabilities{\mc{X}}$, see Definition \ref{G-compatible}, and $G$-(bi)causal couplings. The sets $\Gprobabilities{\mc{X}}$ arise naturally in probabilistic graphical models, see \cite{LauritzenBook}, where $G$-compatibility is known as the local Markov property, and different characterisations of them are well known, see Proposition \ref{prop:Gcomp_characterisation}. We note that throughout, we work on general Polish product spaces and impose neither a joint-density nor a strict-positivity assumptions often encountered in graphical-model results. Interestingly, even some basic properties of $\Gprobabilities{\mc{X}}$ were open, such as when the set is weakly closed, and are among structural results we contribute to the field, see Proposition \ref{prop:P_G_closed}. Our first main result in this context concerns the key property of admissible (i.e., $G$-causal) couplings: in Theorem \ref{glue_characterisation} we show that the gluing of two couplings is still admissible if and only if $G$ is transitively closed and has no splits. 
This is a novel research direction that arises in graph causal OT. It is linked to the study of conditional independence structures via semigraphoid properties \citep{Studeny2005CIStructures}. In contrast, variants of the gluing lemma have typically been available in the other OT settings, see 
\cite[Lemma 7.6]{Villani2003TopicsTransportation} for classical OT, \cite[Lemma 3.1]{Eckstein2022ComputationalTransport} for causal OT, \cite[Lemma A.5]{Acciaio2024Multicausal} for multicausal OT, and imply quick and direct proofs of the triangle inequality for their associated Wasserstein distances. To our knowledge, this is the first natural OT setting in which the relevant gluing property can fail while the associated discrepancy still satisfies the triangle inequality.

Our second main result here, Theorem \ref{main_dense}, shows that, under a suitable continuity assumption (i.e., absence of atoms) on the measures, plans generated by $G$-biadapted Monge maps, our equivalent of Monge maps in OT  \citep{Monge1781MemoireRemblais}, are dense in $G$-bicausal couplings. This mirrors the classical OT results of \cite{Pratelli2007MongeKantorovich} and we were guided by the recent analogue for the adapted OT in \cite{Beiglbock2022Biadapted}. 
 Their approach involves viewing bicausal couplings as projections of deterministic couplings induced by biadapted Monge maps on an enlarged or lifted space, where the new randomness in non-deterministic couplings is made explicit. We follow the same approach and define our set of couplings  as the projections of couplings induced by $G$-biadapted Monge maps on a \emph{lifted} space.  Our key results, Theorems \ref{thm:lifted} and \ref{thm:lifted_nosplits}, characterise precisely the relations between these sets of different couplings and, in particular, show that $G$-causal OT problems defined using regular or lifted couplings are equivalent. This is a stepping stone for our proof of the metric property in Theorem \ref{main_result_metric}, where we can exploit the fact that compositions of $G$-biadapted Monge maps are $G$-biadapted when $G$ is transitively closed. These results also suggest the alternative definition of the $G$-causal Wasserstein metric, $\widetilde{W}_{G,p}$ mentioned above, which uses lifted bicausal couplings.

\textbf{Dynamic programming principle and relation to $AW_p$.}
Our third set of main results focuses on characterising directed acyclic graphs $G$ for which $W_{G,p}=AW_p$, the graph causal and adapted Wasserstein distances coincide. As we always have $W_{G,p}\geq AW_p$, this reduces to understanding when adding the extra causal relations to transform graph $G$ into a linear one adds no advantage for marginals in $\Gprobabilities{\mc{X}}$. We show that the relevant notion is one already known in the probabilistic graphical models literature and require $G$ to be a \emph{perfect} DAG, i.e., that it has no merges. For such graphs, and for costs which are $G$-separable, see \eqref{eq:G-sep cost}, the $G$-causal OT problem can be computed via an iterative procedure, i.e., it satisfies a dynamic programming principle, see Proposition \ref{dpp_proof}. This, in turn, allows one to build an isometric embedding into a certain (classical) Wasserstein space of nested probability measures. The equality $W_{G,p}=AW_p$, and hence also $W_{G,p}=\widetilde{W}_{G,p}=AW_p$,
then follows from analogous embeddings for the latter in \cite{veraguas2020fundamental,Bartl2021WassersteinSpace}, see Theorems \ref{AWisWG} and \ref{aw_equality}.
\smallskip

\textbf{Related literature and outlook.} 
Graphical causal models (GCMs) are a large and active field of study. In many fields of applications, from medicine \cite{greenland1999causal,hernan2000marginal}, through law \cite{chockler2004responsibility,halpern2005causes,halpern2015cause}, to many applications of causal inference in psychology, sociology or political science, see \cite{hernan2020whatif}, structural causal models define how joint distributions arise and these specific representations map onto a qualitative representation as a GCM. We refer to \cite{Pearl2009Causality} and the references therein for a broad overview, to \cite{morgan2007counterfactuals} for a point of view linking with classical data analysis and statistics tools, and to \cite{peters2017elements} for links with machine learning and data-driven causal discovery. All of these provide fundamental motivation to study the spaces of $G$-compatible measures, $\Gprobabilities{\mc{X}}$, for different directed acyclic graphs $G$, and to understand and quantify continuity of natural operations and optimisation problems over such spaces.

Such questions in the context of laws of stochastic processes and their optimisation problems inspired \cite{Pflug2012DistanceMultistage} and the literature on adapted Wasserstein distances. A stream of works established these as the right tools to capture continuity of  multistage stochastic optimisation \citep{PflugPichler2014Multistage}, optimal stopping \citep{BackhoffVeraguas2020AllEqual}, and utility maximisation \citep{Backhoff-Veraguas2019AdaptedFinance}. 
Similarly, the \emph{do operator} and continuity of the average treatment effect, were behind the original work of \cite{Eckstein2023CausalGraphs}, who introduced $G$-causal Wasserstein distances. And whilst we show in this paper that for this particular problem $AW_1$ is in fact sufficient, the original intuition is vindicated by our results in Section \ref{section:applications}, and $W_{G,1}$ turn out to be the right tool to study continuity of optimisation problems over $\Gprobabilities{\mc{X}}$. More broadly, we believe the results discussed so far offer synergies with existing approaches and open up natural directions for future research. We sketch now some examples.

Distributionally robust optimisation \citet{Rahimian2019DROReview} with Wasserstein distances offers a powerful and tractable tool to robustify optimisation or learning tasks, e.g., \cite{mohajerin2018data,Blanchet2019Duality,Gao22a,BLT+22Unified}, and quantify sensitivity to model uncertainty \cite{Bartl2021SensitivityAnalysis,nendel22parametric}. Such applications have recently also been extended to the context of adapted Wasserstein distances, see \cite{bartlSensitivityMultiperiodOptimization2022,fuhrmann2023wasserstein,jiang2024sensitivity}. Similar questions for structural causal models, and causal discovery, are open. With the right distances now to hand, we have a clearer path to approach such problems, even if they can still prove to be very challenging. 

Numerical approaches for OT problems, see \cite{Peyre2018ComputationalTransport}, often go via entropic regularisation. Similar ideas have been fruitful also in other contexts, such as the recent work on the Gromov-Wasserstein OT setup in \cite{Zhang2022GWDuality}, or for adapted OT in \cite{Eckstein2022ComputationalTransport}. Apart from numerical algorithms, these papers often study the relevant entropic OT problem, an interesting mathematical question in its own right, see \cite{leonard_survey}. Defining and studying entropic version of the graph causal OT problem appears as a natural open problem to us. Likewise, 
motivated by the contributions on Wasserstein barycentres going back to \cite{agueh2011barycenters}, averaging over SCMs, i.e., computing graph causal barycentres of a family of elements in $\Gprobabilities{\mc{X}}$, is a very natural problem with many potential applications. We plan to study these questions in subsequent works and note that the first step in this direction was already taken in \cite[Section 4.3]{Eckstein2023CausalGraphs} who considered graph causal interpolation between two measures in $\Gprobabilities{\mc{X}}$.  

Throughout this paper, we use `causal' to refer to dependence constraints encoded via a directed graph $G$ which is assumed fixed. However, several of our results involved comparing graph causal OT, or $W_{G,p}$, for different graphs, such as when considering the transitive closure of $G$, or when showing $W_{G,p}=AW_{p}$. This relates more broadly to the theory of causal abstraction \citep{rubenstein2017causal} which considers relations between two SCMs representing the same system at different levels of granularity. These questions are important, in particular, across sub-fields of AI and relate to surrogate models, aggregation of information and transfer learning. \cite{pmlr-v236-felekis24a} used a multi-marginal OT approach to learn causal abstraction maps from data. We believe that our results could offer novel tools towards such applications.
\smallskip

\textbf{Organisation.} 
The rest of the paper is organised as follows. We begin by establishing notation. In Section \ref{section_main_results} we present our main results on graph causal Wasserstein distances. We next consider the graph causal OT problem more generally in Section \ref{section_structure}. This includes different characterisations of $G$-compatibility and $G$-bicausality, analysis of the gluing property for couplings, denseness of Monge couplings, the key relaxation of lifted graph bicausality, as well as a short discussion of duality. Section \ref{section_dpp} is devoted to establishing the dynamic programming principle and characterising when graph causal Wasserstein and adapted Wasserstein distances are equal. Section \ref{proof_section} then contains proofs. Our decision to assemble the proofs separately from the earlier presentation of the main results is motivated by the logic of arguments. Sections \ref{section_main_results}--\ref{section_dpp} present our main results in what we feel is a natural logical flow to discover them, but it is not one that can be adopted for the proofs where, for example, lifted bicausal couplings and Theorem \ref{thm:lifted} are key in showing the distance property in Theorem \ref{main_result_metric}. Importantly, we start the proofs with Section \ref{section_counterexamples}, where we present several important examples and counterexamples. Further examples are provided throughout the proofs -- most of our results are characterisations and for each of them we build explicit examples to show that the statements fail when one of the assumptions is dropped. Appendix \ref{appendixA} contains some technical, measure-theoretic results,  Appendix \ref{section_closedness} some technical results needed in the proof of Proposition \ref{prop:P_G_closed} on weak closure properties of graph compatible measures, and finally Appendix \ref{section_nonHP} further technical lemmata related to the generalised graph causal Wasserstein distances and the proof of Proposition \ref{prop:nonHP_metric}.

\section{Graph Compatible Measures and Couplings}

\subsection{Notation}

For random variables $A, B, C$ defined on the same probability space, $\condindep{A}{B}{C}$ means that $A$ and $C$ are conditionally independent given $B$. 
We will study probability measures which satisfy conditional independence relations induced by directed graphs of the form $G = (V, E)$, consisting of a finite set of vertices $V$ and a set of directed edges $E \subseteq \curlybrackets{(i, j) \in V \times V \mid i \neq j}$. Vertex labels play no role and, unless stated otherwise, we will  assume $V = \curlybrackets{1, 2, \ldots, n}$ for some $n \geq 1$.
For $i \in V$ we denote the set of \emph{parents} of node $i$ by $\parents{i} \coloneqq \curlybrackets{j \in V \mid (j,i) \in E}.$
We define inductively on $k \geq 1$ the set of \emph{$k$-ancestors} of node $i$ by $\kancestors{1}{i} \coloneqq \parents{i}$ and for all $k \geq 1$, $\kancestors{k+1}{i} \coloneqq \bigcup_{j \in \kancestors{k}{i}} \parents{j}.$ Finally, we define the set of \emph{ancestors} of node $i$ by
$\hist{i} \coloneqq \bigcup_{k \geq 1} \kancestors{k}{i}$ and the set of \emph{descendants} of $i$ by $\desc{i}\coloneqq \{j\in V: i\in \hist{j}\}$. For $A \subseteq V$, we write $\parents{A} \coloneq \bigcup_{i \in A} \parents{i}$. Graph inclusion $G_1\subseteq G_2$ simply means that $V_1\subseteq V_2$ and $E_1\subseteq E_2$, where $G_i=(V_i,E_i)$ $i=1,2$. In contrast, we say that $G_1$ is a \emph{proper subgraph} of $G_2$ when $V_1\subseteq V_2$ and $E_1 = E_2 \cap V_1\times V_1$.

A directed graph $G$ is called \emph{cyclic} if there exists $i \in V$ such that $i \in \hist{i}$. If no such vertex exists, $G$ is a \emph{directed acyclic graph (DAG)}. We say that a DAG is \emph{sorted}, or \emph{topologically ordered}, if $i < j$ whenever $i \in \parents{j}$.  
We say a directed graph $G' = (V', E')$ is a relabelling of $G = (V, E)$ if there exists a bijection $\lambda: V \to V'$ such that $(i, j) \in E$ if and only if $(\lambda(i), \lambda(j)) \in E'$. Relabelling is an equivalence relation on directed graphs and we often choose a suitable relabelling $G'$ instead of $G$. In particular, a DAG can always be relabelled so that it is sorted and hence, without loss of generality, we will mostly work with sorted DAGs. 

\begin{definition}
\label{def:DAGthree_structures}
    Let $G = (V, E)$ be a directed graph and $i, j, k \in V$.
    \begin{itemize}
        \item If $j \in \parents{k}$ and $i \in \parents{j}$ but $i \notin \parents{k}$ we call the triplet $(i,j,k)$ a \emph{chain}. If no such triplet exists, we say $G$ is \emph{transitively closed}. 
        \item If $i \in \parents{j} \cap \parents{k}$ but $j \notin \parents{k}$ and $k \notin \parents{j}$, we call the triplet $(i,j,k)$ a \emph{split}. If no such triplet exists, we say $G$ contains \emph{no splits}.
        \item If $j \in \parents{i}$ and $k\in \parents{i}$ but $j \notin \parents{k}$ and $k \notin \parents{j}$, we call the triplet $(i,j,k)$ a \emph{merge}. If no such triplet exists, we say $G$ is \emph{perfect}.
    \end{itemize}
\end{definition}
\begin{remark}\label{rk:three_structures}
In probabilistic terms, a chain represents the information flow of a Markov process. Note that $G$ is \emph{transitively closed} if whenever $i,j,k \in V$ are such that $i \in \parents{j}$ and $j \in \parents{k}$ then $i \in \parents{k}$. Given a DAG, we can consider its transitive closure, see Definition \ref{def:transitive_closure}. \\
A split is also known as a \emph{unshielded fork} and a merge as a \emph{v-structure}, an \emph{immorality} or an \emph{unshielded collider}, see Sections \ref{section_dpp} and \ref{proof_dpp} for related discussion and equivalent characterisations. 
\end{remark}

Given a directed graph $G = (V, E)$, we will associate to each $v \in V$ Polish spaces $\mc{X}_v$, $\mc{Y}_v$, $\mc{Z}_v$. We write $\mc{X} \coloneqq \prod_{v \in V}\mc{X}_v$, and for $A \subseteq V$, $\mc{X}_A \coloneqq \prod_{v \in A}\mc{X}_v$, with analogous notions for $\mc{Y}, \mc{Z}$. 
Throughout, when $G$ is given we tacitly assume that $\mc{X,Y,Z}$ are given as well. For $A \subseteq V$ and random variables $X$ with state space $\mc{X}$ we write $X_A$ for their projection on $\mc{X}_A$ and for $1 \leq a \leq b \leq n$, we use the shorthand $a : b$ to mean the set $[a,b] \cap \mbb{N}$. As noted, we mostly take $V = \curlybrackets{1, 2, \ldots, n}= 1 : n$ and then $\mc{X} = \mc{X}_1 \times \ldots \times \mc{X}_n$.

For a Polish space $\Omega$, $d_{\Omega}$ denotes its metric, and $\probabilities{\Omega}$ denotes the set of Borel probability measures on $\Omega$. For $p \geq 1$, $\Pp{p}{\Omega}$ is the subset of $\mu \in \probabilities{\Omega}$ such that $\int d_{\Omega}(x,x_0)^p \, d\mu(x) < \infty$ for some $x_0 \in \Omega$. 
For two Polish spaces $\Omega_1, \Omega_2$, we denote by $\proj{\Omega_1} : \Omega_1 \times \Omega_2 \to \Omega_1$ the projection map $(x, y) \mapsto x$. For $\mu \in \probabilities{\Omega_1}$ and a measurable map $T: \Omega_1 \to \Omega_2$, $\pushforward{T}{\mu} \in \probabilities{\Omega_2}$ denotes the pushforward of $\mu$ through $T$. For $\mu_i \in \probabilities{\Omega_i}$, $i=1,2$, we will write 
$\couplings{\mu_1}{\mu_2}= \setdef{\pi \in \probabilities{\Omega_1 \times \Omega_2}}{\pushforward{\proj{\Omega_i}}{\pi}=\mu_i, i=1,2}$ 
for the set of couplings of $\mu_1$ and $\mu_2$. 

\subsection{Bayesian Networks and Their Couplings}

Consider the joint distribution $\mu$ of an $n$-tuple of random variables $X=(X_1,\ldots, X_n)$ which respects the causality relations encoded by a DAG $G$: each $X_i$ is conditionally independent of its non-descendants given its parents. Such objects naturally arise in the causal inference literature as the measures induced by \emph{structural causal models} (SCM) \citep{PearlReview}, and are also called \emph{Bayesian networks}, although that name is more often associated with the couple $(\mu,G)$. More generally, such objects are also known as \emph{probabilistic graphical models}, and different definitions are used in the literature, often assuming joint density exists, see Section \ref{section_structure} for further discussion and equivalence results, and \cite{LauritzenBook} for a broader perspective. Our focus is on $\mu,G$ given, as opposed to designed, and we simply speak of their \emph{compatibility}.

\begin{definition}[G-compatible measures]
\label{G-compatible}
    Let $G$ be a DAG. A measure $\mu \in \probabilities{\mc{X}}$ is called \emph{$G$-compatible} if when $X \sim \mu$, the following conditional independence relations hold:
        $$\condindep{X_i}{X_{\parents{i}}}{X_{V\setminus (\desc{i}\cup \{i\})}},\qquad i\in V.$$
    We denote by $\Gprobabilities{\mc{X}}$ the set of all such $G$-compatible measures, and for $p \geq 1$, $\Gpprobabilities{p}{\mc{X}} \coloneq \Gprobabilities{\mc{X}} \cap \pprobabilities{p}{\mc{X}}$. 
\end{definition}
The above compatibility notion is also known as the \emph{local Markov property}. It is further equivalent to $G$-separation implying conditional independence, also known as the \emph{global Markov property}. Both definitions are invariant under relabelling of vertices. When we take a topologically ordered labelling, i.e., when $G$ is a sorted DAG, they are further equivalent to  
$$
\condindep{X_i}{X_{\parents{i}}}{X_{1:i-1}},\qquad 2\leq i \leq n,
$$
see \cite[Theorem~4.32]{LauritzenBook} and Proposition \ref{prop:Gcomp_characterisation} below which summarises these relations. The notion of $G$-compatibility and its equivalent reformulations are well studied in the causal statistics literature and remain an active field of research, e.g., with extensions to models with cycles and latent variables as in \cite{Bongers2021Cycles}. Nevertheless, to the best of our knowledge, some basic properties of $\Gprobabilities{\mc{X}}$ remain open and our work provides some novel contributions. In particular, we provide a full characterisation of DAGs for which $\Gprobabilities{\mc{X}}$ is weakly closed, see Proposition \ref{prop:P_G_closed}, complementing the partial results in \cite[Proposition 3.3]{Eckstein2023CausalGraphs}.

To define a transport distance between two $G$-compatible measures, we need to consider their couplings and understand what are the causal restrictions they need to satisfy. The central notion of \emph{graph (bi)causal} couplings was first introduced by \cite{Eckstein2023CausalGraphs} and we undertake their comprehensive study, necessary to characterise and understand their induced Wasserstein distances. We start with a new definition here, aligned in spirit with ideas from Bayesian networks and causal statistics.
\begin{definition}[Graph (bi)causal couplings for DAGs]
\label{G_causal_definition}
    Suppose $G$ is a DAG, $\mu \in \Gprobabilities{\mc{X}}$ and $\nu \in \Gprobabilities{\mc{Y}}$, the set of their \emph{$G$-causal} couplings is given as
$$
\Gcouplings{\mu}{\nu} := \couplings{\mu}{\nu}\cap \cGprobabilities{\mc{X}\times\mc{Y}},
$$
where $\causalG = (V^c, E^c)$ is the graph made of $G$ and its copy $G'$ with additional edges linking $i\in V$ (associated to $\mc{X}_i$) and its parents to its copy $i'$ in $V'$ (associated to $\mc{Y}_i$): 
    $$
    V^c = \{i,i': i\in V\},\quad E^c = \{(i,j), (i,j'), (i', j'): (i,j)\in E\} \cup \{(i,i'): i\in V\}.
    $$
Moreover, we define the set of \emph{$G$-anticausal} couplings $\Ganticouplings{\mu}{\nu}$ to be the set of measures in $\couplings{\mu}{\nu}$ whose pushforward under the map $(x,y) \mapsto (y,x)$ is in $\Gcouplings{\nu}{\mu}$, and the set of \emph{$G$-bicausal couplings} $\Gbicouplings{\mu}{\nu} \coloneqq \Gcouplings{\mu}{\nu} \cap \Ganticouplings{\mu}{\nu}$.
\end{definition}

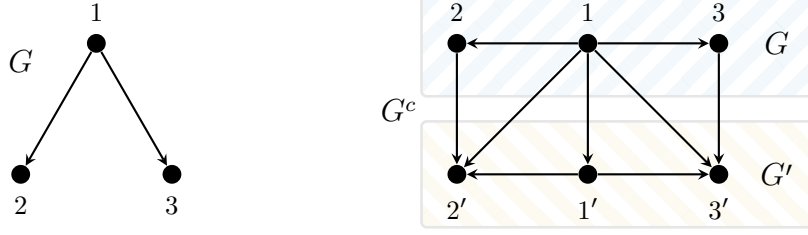
\begin{figure}[H]
    \centering
    \definecolor{cbblue}{RGB}{0,114,178}      
\definecolor{cborange}{RGB}{230,159,0}      
\definecolor{cbvermillion}{RGB}{213,94,0} 

\begin{tikzpicture}[
    >=stealth,
    every node/.style={circle, fill=black, inner sep=2.5pt},
    blacknode/.style={circle, fill=black, inner sep=2.5pt},
    bluenode/.style={circle, fill=cbblue, inner sep=2.5pt},
    orangenode/.style={circle, fill=cborange, inner sep=2.5pt},
    rednode/.style={circle, fill=cbvermillion, inner sep=2.5pt},
    grayedge/.style={->, draw=gray!70, thick},
    blackedge/.style={->, draw=black, thick},
    orangeedge/.style={->, draw=cborange, very thick, densely dashed},
    graphlabel/.style={font=\large, fill=none, inner sep=0pt},
    dotedge/.style={-, draw=gray!30, thick, dotted}
]

\begin{scope}[local bounding box=Gbox]

    \node[blacknode] (1)  at (0, 1.732) [label=above:$1$] {};
    \node[blacknode] (2)  at (-1,0) [label=below:$2$] {};
    \node[blacknode] (3)  at ( 1,0) [label=below:$3$] {};
    
    \draw[blackedge] (1)   -- (2);
    \draw[blackedge] (1)   -- (3);

    \node[graphlabel] at (-1,1.5) {$G$};
\end{scope}

\begin{scope}[xshift=6.5cm, local bounding box=Pbox]
    \node[blacknode] (1)  at (0, 1.732) [label=above:$1$] {};
    \node[blacknode] (2)  at (-1.732, 1.732) [label=above:$2$] {};
    \node[blacknode] (3)  at ( 1.732, 1.732) [label=above:$3$] {};
    \node[blacknode] (1a)  at (0,0) [label=below:$1'$] {};
    \node[blacknode] (2a)  at (-1.732,0) [label=below:$2'$] {};
    \node[blacknode] (3a)  at ( 1.732,0) [label=below:$3'$] {};

    \draw[blackedge] (1)   -- (2);
    \draw[blackedge] (1)   -- (3);
    \draw[blackedge] (1a)   -- (2a);
    \draw[blackedge] (1a)   -- (3a);
    \draw[blackedge] (1)   -- (1a);
    \draw[blackedge] (2)   -- (2a);
    \draw[blackedge] (3)   -- (3a);
    \draw[blackedge] (1)   -- (2a);
    \draw[blackedge] (1)   -- (3a);

    \node[graphlabel] at (-2.5, 0.8516) {$G^c$};
    \node[graphlabel] at (2.5,0) {$G'$};
    \node[graphlabel] at (2.5,1.732) {$G$};

    \begin{pgfonlayer}{background}
        \draw[
        draw=gray!20,
        line width=1pt,
        fill=gray!20,
        fill opacity=0.20,
        pattern={
                Lines[
                    angle=-45,
                    distance=10pt,
                    line width=5pt
                ]
            },
            pattern color=cborange!25.062017,
            rounded corners=2pt
        ] (-2.2,-0.7) rectangle (3,0.7);

        \draw[
        draw=gray!20,
        line width=1pt,
        fill=gray!20,
        fill opacity=0.20,
        pattern={
                Lines[
                    angle=45,
                    distance=10pt,
                    line width=5pt
                ]
            },
            pattern color=cbblue!25.062017,
            rounded corners=2pt
        ] (-2.2,-0.7 + 1.732) rectangle (3,0.7 + 1.732);
    \end{pgfonlayer}
    
\end{scope}
\end{tikzpicture}
    \caption{An example of the construction of $G^c$ for the DAG $G$ with vertex set $\curlybrackets{1,2,3}$ and edge set $\curlybrackets{(1,2), (1,3)}$, shown on the left. $G^c$ is shown on the right, and the two copies $G$ and $G'$ of the original DAG in $G^c$ are highlighted. The spaces $\mc{X}_1, \mc{X}_2, \mc{X}_3$ are associated to the nodes $1$, $2$, $3$, while the spaces $\mc{Y}_1, \mc{Y}_2, \mc{Y}_3$ are associated to the nodes $1'$, $2'$, $3'$.}
    \label{fig:larger_graph}
\end{figure}

We show that $\Gbicouplings{\mu}{\nu}$ is non-empty and discuss a number of equivalent characterisations of graph bicausal couplings in Section \ref{section_structure}, as well as its extensions to directed (potentially cyclical) graphs, see Definition \ref{G_causal_definition_2}. 
In particular, in Proposition \ref{equivalent_characterisations}, we show that when $G$ is a sorted DAG then $\pi \in \couplings{\mu}{\nu}$ is \emph{$G$-causal} if and only if when $(X,Y) \sim \pi$, then the following conditional independence relations hold:
        $$\condindep{Y_i}{X_i, X_{\parents{i}}, Y_{\parents{i}}}{(X, Y_{1:i-1})},\qquad 1\leq i\leq n.$$
We typically work with this characterisation and can do so without any loss of generality since, given a DAG, we may always choose a relabelling of its vertices such that it becomes a sorted DAG with vertex set of the form $\curlybrackets{1,2, \ldots, n}$ and, as noted above, $G$\emph{-compatibility}, and hence also \emph{$G$-(bi)causality}, is invariant under relabelling.

\section{Graph Causal Wasserstein Distances}
\label{section_main_results}

\subsection{Metric Property}

Graph bicausal couplings naturally induce the following variant of the Wasserstein distance:

\begin{definition}[Graph causal Wasserstein distances]\label{def:G-Wass}
    Let $G$ be a directed graph and $p \geq 1$. The $p$\textsuperscript{th} order \emph{$G$-causal Wasserstein distance} between $\mu, \nu \in \Gpprobabilities{p}{\mc{X}}$ is given by
    \begin{equation}\label{eq:W_Gdef}
        W_{G,p}(\mu, \nu)^p \coloneq \inf_{\pi \in \Gbicouplings{\mu}{\nu}} \int d_{\mc{X}}(x,y)^p \, d\pi(x,y).
    \end{equation}
\end{definition}

This object was originally proposed in \cite{Eckstein2023CausalGraphs}. It is immediate that $W_{G,p}$ satisfies two of the three properties of a metric: symmetry ($W_{G,p}(\mu, \nu) = W_{G,p}(\nu, \mu)$) and positivity ($W_{G,p}(\mu, \nu) \geq 0$, with equality if and only if $\mu = \nu$). 
The former follows from the purposefully symmetrical definition of $\Gbicouplings{\mu}{\nu}$, and the second from $d_{\mc{X}}$ being a distance and the fact that the identity coupling is always graph bicausal. In contrast, the third property, the triangle inequality, was left open by \cite{Eckstein2023CausalGraphs} who  only provide a numerical counterexample showing that the triangle inequality does not generally hold for the single chain graph $\Gmarkov$, with vertex set $\curlybrackets{1,2,3}$ and edge set $\curlybrackets{(1,2), (2,3)}$. 

Our first main result proves that the triangle inequality does hold for a large class of DAGs and hence $W_{G,p}$ is a metric. 

\begin{theorem}
    \label{main_result_metric}
    Let $G$ be a transitively closed DAG. Then $W_{G,p}$ is a metric on $\Gpprobabilities{p}{\mc{X}}$. Conversely, if $G$ is a DAG which is not transitively closed, then there is a choice of $(\mc{X},d_{\mc{X}})$ for which $W_{G,p}$ does not satisfy the triangle inequality on $\Gpprobabilities{p}{\mc{X}}$.
\end{theorem}
The first part of the above theorem is a positive statement. Together with the result presented in the next sections, these offer a cornerstone for a new field of graph causal optimal transport and, we believe, open up exciting avenues for future research, some of which we outlined already above. 

The second part of the statement shows the limitations of Definition \ref{def:G-Wass}. Example \ref{example_counter_Markov} gives an explicit counterexample where the triangle inequality fails for $\Gmarkov$,   complementing the numerical results of \cite{Eckstein2023CausalGraphs}. The converse statement then follows exploiting the fact that if $G$ is not transitively closed, then it has $\Gmarkov$ as a proper subgraph. The proof of the first part of Theorem \ref{main_result_metric} requires us to introduce $G$-bicausal couplings on a lifted space. Whilst their definition may appear technical at first, it lead to the modified Definition \ref{def:lifted_G-Wass}, of lifted $G$-causal Wasserstein distances $\widetilde{W}_{G,p}$, which are always a metric and which agree with $W_{G,p}$ for transitively closed DAGs. In fact, they correspond to $W_{\Gbar, p}$, where $\Gbar$ is the (transitive) closure of $G$, see Proposition \ref{prop:nonHP_metric} and the discussion which precedes it.

Important classes of graphs $G$ allow us to recover known transport distances. 
First, when $G$ corresponds to the linear information structure, i.e., each node has all earlier nodes as its parents, $G=G_n := ({1:n}, \setdef{(i, j)}{1\leq i<j\leq n})$, $G_n$-causality is equivalent to $\condindep{Y_i}{X_{1:i}, Y_{1:i-1}}{(X, Y_{1:i-1})}$ for all $i\in V$, which by the chain rule of conditional independence \cite[Theorem 8.12]{Kallenberg2021FoundationsThird}, is equivalent to $\condindep{Y_{1:i}}{X_{1:i}}{X}$ for all $i\in V$. Informally speaking: the past of $Y$ only depends on $X$ through its past. This is precisely the notion of a causal coupling from causal OT \citep{Lassalle2013Causal,Backhoff2016CausalDiscrete}, all marginals are $G_n$-compatible, $\specificGpprobabilities{G_n}{p}{\mc{X}}=\pprobabilities{p}{\mc{X}}$, and we recover the adapted, or bicausal, Wasserstein distance $W_{G_n,p}=AW_p$. 
Second, when $G$ is a fully connected graph, $\parents{i}=V\setminus\{i\}$ for all $i\in V$, extending current definitions to allow for cyclic graphs, see Definition \ref{G_causal_definition_2}, $W_{G,p}$ reduces to the classical Wasserstein distance.
Third, a graph $G$ with no edges results in a distance analogous to the Gromov-Wasserstein distance, see Section \ref{sec:duality}. 
We refer to \cite[Section 2.4]{Eckstein2023CausalGraphs} for a detailed discussion. In these classical examples, it is relatively easy to establish the triangle inequality, and hence the distance property, using the so-called \emph{gluing lemma}: constructing a $\pi\in \couplings{\mu}{\eta}$ from $\pi_1\in \couplings{\mu}{\nu}, \pi_2\in \couplings{\nu}{\eta}$ by disintegrating and joining along the common marginal. This natural strategy does not generalise to general DAGs: the gluing property only holds for transitively closed DAGs with no \emph{splits}, see Theorem \ref{glue_characterisation} and the discussion in Section \ref{main_gluing}. 

The proof of Theorem \ref{main_result_metric} is presented in Section \ref{proof_metric}. Since we can not simply rely on a gluing lemma, its  strategy is to consider a lifted space, which allows for additional independent randomisations at each node. We show that $G$-bicausal couplings are projections of Monge couplings (i.e., couplings supported on the graph of a function) on the lifted space. Furthermore, such projections are essentially convex combinations of elements in $\Gbicouplings{\mu}{\nu}$. Since Monge couplings are easy to glue (by composing their associated functions) these results allow us to establish the desired triangle inequality. These, and other, structural properties of $G$-bicausal couplings are discussed in Section \ref{section_structure}. 

\subsection{First Applications}
\label{section:applications}

In the context of stochastic optimisation problems, $G$-compatible measures arise as the state distributions in stochastic team problems with information sets which do not necessarily increase \citep{Yuksel2024TeamsBook}. We establish Lipschitz continuity of their value function with respect to $W_{G,1}$, which provides a foundational motivation for these distances. We are inspired by the results of \citet[Theorem~6.4, Corollary~6.7]{PflugPichler2014Multistage} on continuity of multistage stochastic optimisation problems with respect to the adapted Wasserstein distance. 

\subsubsection*{Stochastic Team Problems}

Consider a sorted and transitively closed DAG $G$ with $V = \curlybrackets{1, \ldots, n}$, and a system with stochastic state described by the random variable $X= (X_1, \ldots, X_n) \coloneqq \Id{\mc{X}}$ such that $\mu \coloneqq \law{X} \in \Gprobabilities{\mc{X}_{1:n}}$. A collection of $n$ decision makers (DMs) must choose actions $\alpha = (\alpha_1, \ldots, \alpha_n)$ in a closed cylinder set $A \coloneqq \prod_{k=1}^n A_k \subseteq \mbb{R}^n$ in such a way that the action of the $k$\textsuperscript{th} decision maker, $\alpha_k \in A_k$ only uses the information available by observing $X_k, X_{\parents{k}}$ i.e., $\alpha_k$ must be $\sigma(X_k, X_{\parents{k}})$-measurable. We write $\alpha \isGadapted X$ for this set of so-called $G$-adapted actions. Note that for $\alpha \isGadapted X$, there exist measurable functions $H_k : \mc{X}_k \times \mc{X}_{\parents{k}} \to \mbb{R}$ such that $\alpha_k = H_k(X_k, X_{\parents{k}})$.
The (shared) goal of the DMs is to minimise the expectation of a given loss function $Q : \mc{X} \times A \to [0, \infty)$. We thus want to solve
\begin{equation}\label{eq:StochTeamPb}
v(\mu) \coloneqq \inf_{\alpha \in A, \alpha \isGadapted X} \expectation{\mu}{Q(X, \alpha)},    
\end{equation}
where we have written $\mbb{E}_{\mu}$ to emphasize the fact that $\law{X}=\mu$.

\begin{remark}
    In the case of a linear information structure $G=G_n$, $\alpha \isGadapted X$ reduces to the requirement that $\alpha$ is adapted to the canonical filtration generated by $X$ i.e., $\alpha \triangleleft X$. In this case, the DMs have non-decreasing information sets and can be viewed as the series of consecutive decisions by a central actor who accumulates the information. We thus recover the classical multistage stochastic optimisation problem, and the setting of \cite{PflugPichler2014Multistage}. 
\end{remark}

\begin{remark}
    Since $G$ is transitively closed, the relation $i \in \parents{j}$ induces a partial order on $V$, which we denote by $\leq_G$. For a series of $m \geq 1$ decision makers with labels $k_1 \leq_G \ldots \leq_G k_m$, their information sets $\mc{F}_{k_i} \coloneqq \sigma(X_{k_i}, X_{\parents{k_i}})$ then form a filtration $\mc{F}_{k_1} \subseteq \ldots \subseteq \mc{F}_{k_m}$. We are thus in the setting of a quasi-classical team game \cite{Yuksel2024TeamsBook}, i.e., one with partially nested information sets.
\end{remark}

The following result, whose proof is presented in Section \ref{proof_lipschitz_optimisation}, establishes the desired Lipschitz continuity with respect to $W_{G,1}$. It is complemented by Example \ref{example:notLipAW} which shows that, in general, Lipschitz continuity does not hold in adapted Wasserstein distance. 

\begin{theorem}
\label{thm:lipschitz_optimisation}
    Suppose $A \subseteq \mbb{R}^n$ is compact and convex, that the map $\alpha \mapsto Q(x, \alpha)$ is convex for all $x \in \mc{X}$, and that $\abs{Q(x, \alpha) - Q(y, \alpha)} \leq Ld_{\mc{X}}(x,y)$ for all $\alpha \in A$ for some $L > 0$. Then 
    $$\abs{v(\mu) - v(\nu)} \leq L \cdot W_{G,1}(\mu, \nu), \, \forall \mu, \nu \in \Gpprobabilities{1}{\mc{X}}.$$
\end{theorem}

\subsubsection*{Average Treatment Effect}

As a second application, we recall the result from \cite{Eckstein2023CausalGraphs}, which showed that the \emph{average treatment effect} in causal statistics, see \cite{Pearl2009Causality}, is $W_{G,1}$-Lipschitz. Assume that for vertices $j < k$ in a sorted DAG $G = (1:n,E)$, $X_j \in \mc{X}_j = \curlybrackets{0,1}$ encodes whether a treatment is applied or not (or, in non-medical terms, whether an intervention is done or not), and $X_k \in \mc{X}_k$ encodes the outcome for some target variable. We assume $\mc{X}_k$ is a compact subset of $\mbb{R}$, that $d_{\mc{X}_i}(x_i, y_i) = \abs{x_i - y_i}$ for $i = j, k$, and that $d_{\mc{X}} = \sum_{i=1}^n d_{\mc{X}_i}(x_i, y_i)$. 

If the true data-generating distribution is given by some causal model $(\mu, G)$ with $\mu \in \Gprobabilities{\mc{X}}$, then the average effect of applying the treatment, expressed in Pearl's {\emph{do}}-notation \citep{Pearl2009Causality}, can be identified from $\mu$ via the back-door adjustment \cite[Theorem 3.3.2]{Pearl2009Causality} as

$$\int_{\mc{X}_k} x_k \, \mu(dx_k \,\vert\, \mathrm{do}(x_j=1)) = \int_{\mc{X}_{\parents{j}}} \int_{\mc{X}_k} x_k \, \mu(dx_k \,\vert\,x_j=1, x_{\parents{j}}) \, \mu(dx_{\parents{j}}).$$

Replacing $1$ with $0$ above yields the formula for the average effect of not applying the treatment. The \emph{average treatment effect (ATE)} is then defined as
\begin{equation}\label{eq:AvTrEff}
\psi^{\mu, G} \coloneq \abs{\int_{\mc{X}_k} x_k \, \mu(dx_k \,\vert\, \mathrm{do}(x_j=1)) - \int_{\mc{X}_k} x_k \, \mu(dx_k \,\vert\, \mathrm{do}(x_j=0))}.    
\end{equation}

Consider the situation that we were right about the causal graph $G$ but might have misestimated our model $\mu$. We restrict our class of possible models to those satisfying $\delta \leq \mu(x_j =1 \, \vert \, x_{\parents{j}}) \leq 1 - \delta$ for $\mu$-almost all $\xparents{j}$ for some fixed $\delta > 0$,  the set of which is denoted by $\specificGprobabilities{G, \delta}{\mc{X}}$. Then the following result controls the impact of our misestimation on the ATE via the $W_{G,1}$ distance:

\begin{proposition}
[\cite{Eckstein2023CausalGraphs}, Proposition 4.4]
\label{prop:continuity_ate}
    Under the assumptions above, there exists a constant $L > 0$ such that
    $$\abs{\psi^{\mu, G} - \psi^{\nu, G}} \leq L \cdot W_{G,1}(\mu, \nu) ,\, \forall \, \mu, \nu \in \specificGprobabilities{G, \delta}{\mc{X}}.$$
\end{proposition}

Whilst the above result motivated \cite{Eckstein2023CausalGraphs} to introduce $W_{G,1}$, the nature of the average treatment effect is such that adding additional causal relations does not necessarily change the quantity in \eqref{eq:AvTrEff}. In consequence, and in contrast to the setting of Theorem \ref{thm:lipschitz_optimisation} above, the Lipschitz continuity also holds with respect to the adapted Wasserstein distance $AW_1$. More generally, we have the following result. 

\begin{corollary}
\label{corollary:ate}
    Let the assumptions above hold, and let $G' = (1:n, E')$ be another sorted DAG such that $E \subseteq E'$. Then, there exists a constant $L > 0$ such that
    $$\abs{\psi^{\mu, G} - \psi^{\nu, G}} \leq L \cdot W_{G',1}(\mu, \nu) ,\, \forall \, \mu, \nu \in \specificGprobabilities{G, \delta}{\mc{X}}.$$
    In particular, there exists a constant $L > 0$ such that
    $$\abs{\psi^{\mu, G} - \psi^{\nu, G}} \leq L \cdot AW_{1}(\mu, \nu) ,\, \forall \, \mu, \nu \in \specificGprobabilities{G, \delta}{\mc{X}}.$$
\end{corollary}
\begin{proof}
    The result follows from Proposition \ref{prop:continuity_ate} if we show that $\specificGprobabilities{G, \delta}{\mc{X}} \subseteq \specificGprobabilities{G', \delta}{\mc{X}}$ and that $\psi^{\mu, G} = \psi^{\mu, G'}$ for all $\mu \in \specificGprobabilities{G, \delta}{\mc{X}}$. This is shown in Lemma \ref{lemma:ate} and is essentially a consequence of the back-door adjustment \cite[Theorem 3.3.2]{Pearl2009Causality}. The final conclusion follows by taking $G'$ to be the graph $G_n = ({1:n}, \setdef{(i, j)}{1\leq i<j\leq n})$, for which by definition $W_{G_n,1} = AW_1$.
\end{proof}
In a recent work, \cite{visentin2025robust} considered $W_{G,1}$ in the context of continuity of causal optimisation problems. Whilst we believe that analogous remarks as above should apply also to this work, i.e., one could use $AW_1$ instead of $W_{G,1}$, we note that their proposed graph-causal data augmentation techniques are of independent interest.

\subsection{Topological Properties}
\label{sec:topology}

We now investigate the topology induced by graph causal Wasserstein distances. Fix a transitively closed DAG $G$ and $p \geq 1$. In this section, we also fix on our product $\mc{X} = \prod_{i=1}^n \mc{X}_i$ of Polish spaces $(\mc{X}_i, d_i)$ the metric $d$ given by $d(x,y)^p = \sum_{i=1}^n d_i(x_i, y_i)^p$. We recall that the adapted weak topology on 
$\pprobabilities{p}{\mc{X}}$ is induced by the adapted Wasserstein distance $AW_p = W_{G_n,p}$. 

\begin{theorem}
    \label{theorem:sametopology}
     The topology induced by $W_{G,p}$ on $\Gpprobabilities{p}{\mc{X}}$ is the subspace topology on $\Gpprobabilities{p}{\mc{X}}$ of the adapted weak topology on $\pprobabilities{p}{\mc{X}}$. 
\end{theorem}

The proof is deferred to Section \ref{proof:sametopology}. We recall that the adapted Wasserstein topology was studied extensively, e.g., in \cite{BackhoffVeraguas2020AllEqual}, \cite{Bartl2021WassersteinSpace}, \cite{Pammer2022Adapted}. In particular, in the first of these papers, the authors obtained the seminal result that this, and many other natural adapted topologies, including \cite{Aldous1981ExtendedWeak}'s extended weak topology, or \cite{Hellwig1996InformationTopology}'s information topology, 
were all equal. The above theorem thus asserts the same principle: there is just one natural adapted topology and it is recovered by graph causal Wasserstein distances. The intuition behind this potentially surprising result is that at infinitesimal scales, couplings constructed from myopic one-step choices are nearly optimal, see Remark \ref{rk:myopiccouplings}.
Note however that while $AW_p$ and $W_{G,p}$ are equivalent on $\Gpprobabilities{p}{\mc{X}}$, they are not uniformly equivalent, see Example \ref{example:notequivalent}, and stochastic team games covered in Theorem \ref{thm:lipschitz_optimisation} are not Lipschitz continuous with respect to the adapted Wasserstein distance, see Example \ref{example:notLipAW}.

\begin{remark}\label{rk:myopiccouplings}
In order to compute $AW_1(\mu, \nu)$ for $\mu, \nu \in \probabilities{\mc{X}}$, one chooses for each step $1 \leq i \leq n$ couplings of disintegrations $\pi_i^{x_{1:i-1}, y_{1:i-1}} \in \couplings{\mu_i^{x_{1:i-1}}}{\nu_i^{y_{1:i-1}}}$ in such a way that if $\pi(dx, dy) = \bigotimes_{i=1}^n \pi_i^{x_{1:i-1}, y_{1:i-1}}(d{x_i}, d{y_i})$, then $\int \sum_{i=1}^n d_i(x_i, y_i) \,d\pi(x,y)$ is as small as possible.

A naive attempt to achieve this would be to build the coupling in a myopic fashion: choose each $\pi_i^{x_{1:i-1}, y_{1:i-1}}$ such that $\int d_i(x_i, y_i) \, d\pi_i^{x_{1:i-1}, y_{1:i-1}}(x_i, y_i) = W_1(\mu_i^{x_{1:i-1}}, \nu_i^{y_{1:i-1}})$. This choice would result in a bicausal coupling, but one that is unlikely to be optimal. Indeed, the optimal choice is to pick $\pi_i^{x_{1:i-1}, y_{1:i-1}}$ such that it minimises
$$\int d_i(x_i, y_i) + AW_1(\mu_{i+1:n}^{x_{1:i}}, \nu_{i+1:n}^{y_{1:i}}) \, d\pi_i^{x_{1:i-1}, y_{1:i-1}}(x_i, y_i),$$
thus taking into account the effect of our choice at time $i$ on the optimisation problems for times $j > i$, in a dynamic programming fashion. For myopic couplings, the errors can compound exponentially in the number of time steps. 

However, at infinitesimal scales, when $AW_1(\mu^{(m)}, \nu) \to 0$ as $m \to \infty$, the myopic couplings are not  far from optimal, as the extra dynamic programming errors converge to zero (because of the finite number of time steps). Indeed, an equivalent topology, Hellwig's (reduced) information topology \citep{Hellwig1996InformationTopology}, is the initial topology of the maps
$$\mu \in \probabilities{\mc{X}} \mapsto \mc{L}_{X \sim \mu}\brackets{X_1, \ldots, X_{i-1}, \condlaw{X_i}{X_{1}, \ldots, X_{i-1}}} \in \probabilities{\mc{X}_{1:i-1} \times \probabilities{\mc{X}_i}},$$
for $1 \leq i \leq n$. Thus,  we can focus only on the next step conditional distributions. In our case, when $\mu, \nu \in \Gprobabilities{\mc{X}}$, the set of couplings $\couplings{\mu_i^{x_{1:i-1}}}{\nu_i^{y_{1:i-1}}}$ is in fact equal to $\couplings{\mu_i^{\xparents{i}}}{\nu_i^{\yparents{i}}}$, and thus we can build myopic $G$-bicausal couplings by choosing $\pi_i^{\xparents{i}, \yparents{i}} \in \couplings{\mu_i^{\xparents{i}}}{\nu_i^{\yparents{i}}}$ that keep $\int d_i(x_i, y_i) \,d\pi_i^{\xparents{i}, \yparents{i}}(x_i, y_i)$ minimal. The proof of Theorem \ref{theorem:sametopology} shows that at small scales such myopic constructions converge to the optimum. A more general discussion of dynamic programming results for $W_{G,p}$ is presented in Section \ref{section_dpp}.
\end{remark}

\begin{example}[Not uniformly equivalent]
    \label{example:notequivalent}
    While $AW_p \leq W_{G,p}$, there does not in general exist a constant $C > 0$ such that $AW_p \geq C \cdot W_{G,p}$, as shown below.

    For $\Gmerge$, the graph with vertex set $\curlybrackets{1,2,3}$ and edge set $\curlybrackets{(1,3), (2,3)}$, let for $\eps >0$
    $$\mu^{(\eps)} \coloneq \frac{1}{4} \brackets{\delta_{(\eps, \eps, 1)} + \delta_{(\eps, -\eps, -1)} + \delta_{(-\eps, \eps, -1)} + \delta_{(-\eps, -\eps, 1)}} \in \specificGprobabilities{\Gmerge}{\mbb{R}^3},$$
    $$\nu^{(\eps)} \coloneq \frac{1}{2} \brackets{\delta_{(0, \eps, 1)} + \delta_{(0,-\eps,-1)}} \in \specificGprobabilities{\Gmerge}{\mbb{R}^3}.$$

    Then, taking the transference plan induced by the Monge map $$((-1)^j\eps, (-1)^k\eps, (-1)^{j+k}) \mapsto (0, (-1)^{j+k}\eps, (-1)^{j+k}), \,\forall \, j,k \in \curlybrackets{0,1},$$
    one sees that $AW_p(\mu^{(\eps)}, \nu^{(\eps)})^p \leq \eps^p + \frac{1}{2}(2\eps)^p \searrow 0$ as $\eps \searrow 0$. 

    However, any $\Gmerge$-bicausal coupling $\pi^{(\eps)}$ of $\mu^{(\eps)}, \nu^{(\eps)}$ will have $\pi^{(\eps)}(X_3 \neq Y_3) = \frac{1}{2}$ and thus $W_{G,p}(\mu^{(\eps)}, \nu^{(\eps)})^p \geq \frac{1}{2} 2^p$, so no such $C>0$ can exist. 
\end{example}

\begin{example}[Cauchy sequences]
     As $AW_p \leq W_{G,p}$ a Cauchy sequence in $(\Gpprobabilities{p}{\mc{X}}, W_{G,p})$ will be Cauchy for $AW_p$. The converse is false. Indeed, in the example above, it is easy to see that both $(\mu^{\brackets{\frac{1}{m}}})_{m \geq 1}$ and $(\nu^{\brackets{\frac{1}{m}}})_{m \geq 1}$ are Cauchy in both $AW_p$ and $W_{G,p}$ (by always transporting $\pm \eps_1$ to $\pm \eps_2$ coordinatewise). However, the example above shows that the sequence formed by alternating elements of $(\mu^{\brackets{\frac{1}{m}}})_{m \geq 1}$ and $(\nu^{\brackets{\frac{1}{m}}})_{m \geq 1}$ is Cauchy in $AW_p$, but not in $W_{G,p}$. 

     The question of the completion of $(\Gpprobabilities{p}{\mc{X}}, W_{G,p})$ alongside compactness properties are reserved for future study. In the case of adapted Wasserstein distances, completion and compactness were studied in \cite{Bartl2021WassersteinSpace}, \cite{Eder2019Compactness}.
\end{example}

\begin{example}[Non-separable spaces]
    Note that the assumption that our metric spaces $(\mc{X}_i, d_i)$ are separable is key for Theorem \ref{theorem:sametopology} to hold.

    Indeed consider $\Gmerge$ again, and on $\mbb{R}^3$ the metric given by $d(x,y) \coloneq \abs{x_1-y_1} + \abs{x_2-y_2} + \mathbbm{1}_{x_3 \neq y_3}$. 

    Let $\lambda$ be the uniform measure on $[0,1]$, and for $\eps > 0$ define $\mu^{(\eps)} \in \specificGprobabilities{\Gmerge}{\mbb{R}^3}$ as the pushforward of the measure $\brackets{\frac{1}{2} \delta_{\eps} + \frac{1}{2} \delta_{-\eps}} \otimes \lambda$ through the map $(x_1,x_2) \mapsto (x_1, x_2, x_1 + x_2)$. It is easy to see that $\mu^{(\eps)}$ converges in $AW_1$ as $\eps \searrow 0$ to the measure $\nu \in \specificGprobabilities{\Gmerge}{\mbb{R}^3}$ defined as the pushforward of the measure $\lambda$ through the map $y_2 \mapsto (0, y_2, y_2)$. 

    However, for any $\pi^{(\eps)} \in \Gbicouplings{\mu^{(\eps)}}{\nu}$, we have
    \begin{align*}
        \int d(x,y) \,d\pi(x,y) &= \eps + \int \abs{x_2 - y_2}\, d\pi_2^{(\eps)}(x_2, y_2) \\
        &+ \frac{1}{2} \int \mathbbm{1}_{(x_2+\eps) \neq y_2}\, d\pi_2^{(\eps)}(x_2, y_2) + \frac{1}{2} \int \mathbbm{1}_{(x_2-\eps) \neq y_2}\, d\pi_2^{(\eps)}(x_2, y_2) \geq \frac{1}{2},
    \end{align*}
    and so $W_{G,1}(\mu^{(\eps)},\nu) \not\to 0$ as $\eps \searrow 0$.
\end{example}

\begin{example}[Stochastic team games]\label{example:notLipAW}
    Continuing Example \ref{example:notequivalent}, consider the stochastic team game for three DMs with information structure given by $\Gmerge$. 

    We choose $A = [-1,1]^3$ and objective function $Q: \mbb{R}^3 \times A$, given by $Q(x, a) \coloneq \abs{x_3 - a_2}$, which satisfies the assumptions of Theorem \ref{thm:lipschitz_optimisation}. Then $$v(\mu) = \inf \setdef{\expectation{\mu}{\abs{X_3 - \alpha_2(X_2)}}}{\alpha_2 : \mbb{R} \to A \, \mathrm{measurable}}.$$

    Then it is easy to see that 
    $$v(\mu^{(\eps)}) = \inf_{a_{\pm} \in [-1,1]} {\frac{1}{4} \brackets{\abs{1-a_{+}} + \abs{1-a_{-}} + \abs{-1-a_{+}} +\abs{-1-a_{-}}}} = 1,$$
    $$v(\nu^{(\eps)}) = \inf_{a_{\pm} \in [-1,1]} {\frac{1}{2} \brackets{\abs{1-a_{+}} + \abs{-1-a_{-}}}} = 0.$$
    As $AW_1(\mu^{(\eps)}, \nu^{(\eps)}) \to 0$ as $\eps \searrow 0$, we see that $v$ is not Lipschitz continuous (or in fact uniformly continuous) with respect to $AW_1$. In contrast, Theorem \ref{thm:lipschitz_optimisation} shows $v$ is Lipschitz continuous in $W_{\Gmerge,1}$ and thus graph causal Wasserstein distances provide a better metric to use for estimating errors of model misspecification for stochastic team games with a quasi-classical information structure.
\end{example}

\section{Graph Causal Optimal Transport}
\label{section_structure}

We move now from the discussion of graph causal Wasserstein distances to the more general graph causal OT problem, i.e., the equivalent of the classical OT problem but with additional $G$-bicausality requirements on the couplings. 
Given a directed graph $G$, $\mu\in \Gprobabilities{\mc{X}}$, $\nu\in \Gprobabilities{\mc{Y}}$ and a cost function $\xi: \mc{X}\times \mc{Y}\to \mbb{R}$, we define \emph{the graph causal OT problem} as
\begin{equation}
\label{eq:GOT}
    \GOT{\mu}{\nu}:= \inf_{\pi\in \Gbicouplings{\mu}{\nu}} \int_{\mc{X}\times \mc{Y}} \xi(x,y) \, d \pi(x,y).
\end{equation}

Whilst the restriction on couplings reflects causal relations between different variables and may be natural in many applications, it is far from innocent and imposes intricate structural difficulties. The set $\Gbicouplings{\mu}{\nu}$ is, in general, neither convex nor weakly closed. In addition, as noted before and explored in detail below, we may not be able to glue couplings together and preserve $G$-bicausality. For these reasons, our first focus is on characterising fine structural properties of the sets $\Gprobabilities{\mc{X}}$ and $\Gbicouplings{\mu}{\nu}$, as well as the lifted graph bicausal couplings, which will form a key element of the proof of Theorems \ref{main_result_metric} and \ref{main_dense}, and lead to another characterisation of graph causal Wasserstein distances. 

\subsection{Equivalent Definitions of $G$-compatibility and (Bi)causality}

We start with equivalent characterisations of the notion of $G$-compatibility which we introduced in Definition \ref{G-compatible}. These are well-known in the probabilistic graphical models literature. We refer the reader to \cite[Theorem 3.2.1]{Bongers2021Cycles} for a complete statement and proof containing yet further equivalent characterisations, among which we mention the \emph{$d$-separation} criterion \citep{Pearl1988DSeparation}, and a discussion for measures admitting a density. Partial versions appear in, e.g., \cite[Theorem 4.38]{LauritzenBook} and \cite[Definition 3.1]{Eckstein2023CausalGraphs}.

\begin{proposition}\label{prop:Gcomp_characterisation}
    Let $G$ be a DAG and $\mu \in \probabilities{\mc{X}}$. The following are equivalent:
    \begin{enumerate}
        \item[1)] $\mu \in \Gprobabilities{\mc{X}}$;
        \item[2)] There exists a random variable $X \sim \mu$, independent $\mbb{R}$-valued random variables $(U_i)_{i\in V}$, and measurable functions $f_i \colon \mc{X}_{\parents{i}} \times \mbb{R} \to \mc{X}_i, i\in V$ such that $X_i = f_i(X_{\parents{i}}, U_i), \, \forall \, i \in V$.
    \end{enumerate}
    If $G$ is a sorted DAG, the above are further equivalent to the following:
    \begin{enumerate}
        \item[3)] $\mu$ admits a disintegration $\mu(dx) = \bigotimes_{i=1}^n \mu_i(dx_i \vert \xparents{i})$;
        \item[4)] If $X\sim \mu$ then 
        $$ \condindep{X_i}{X_{\parents{i}}}{X_{1:i-1}},\qquad 2\leq i \leq n.$$ 
    \end{enumerate}
\end{proposition}

\begin{remark}
    The first two characterisations make it clear that $G$-compatibility is invariant under relabelling. In particular, the latter two characterisations hold for one relabelling which turns $G$ into a sorted DAG if and only if they work for all such relabellings. 
\end{remark}

In Definition \ref{G_causal_definition} we defined graph (bi)causal couplings between measures which are compatible with a DAG, which is the natural context for the main contributions in this paper. However, building on the characterisation 2) in Proposition \ref{prop:Gcomp_characterisation}, \cite[Definition 2.1]{Eckstein2023CausalGraphs} proposed a more general definition for possibly cyclical directed graphs and possibly non-compatible measures, which we now recall. We then show, in Proposition \ref{prop:equivalent_Gcausal} below,  that Definitions \ref{G_causal_definition} and \ref{G_causal_definition_2} agree when $G$ is a DAG and $\mu \in \Gprobabilities{\mc{X}}$, $\nu \in \Gprobabilities{\mc{Y}}$, and in this context we can, and will, use both definitions interchangeably.

\begin{definition}[Graph (bi)causal couplings for directed graphs]
\label{G_causal_definition_2}
    Let $G$ be a directed graph, $\mu \in \probabilities{\mc{X}}$, $\nu \in \probabilities{\mc{Y}}$. We call a coupling $\pi \in \couplings{\mu}{\nu}$ $G$-causal if there exists a pair of random variables $(X,Y) \sim \pi$, measurable functions
    $$f_i \colon \mc{X}_i \times \mc{X}_{\parents{i}} \times \mc{Y}_{\parents{i}} \times \mbb{R} \to \mc{Y}_i, \quad i\in V,$$
    and $\mbb{R}$-valued random variables $(U_i)_{i\in V}$ such that $\sigma(X), (\sigma(U_i))_{i \in V}$ are independent and
    $$Y_i = f_i(X_i, X_{\parents{i}}, Y_{\parents{i}}, U_i), \quad \forall \, i \in V.$$
    We denote by $\Gcouplings{\mu}{\nu}$ the set of all such $G$-causal couplings. Moreover, we define the set of $G$-anticausal couplings $\Ganticouplings{\mu}{\nu}$ to be the set of measures in $\couplings{\mu}{\nu}$ whose pushforward under the map $(x,y) \mapsto (y,x)$ is in $\Gcouplings{\nu}{\mu}$, and the set of $G$-bicausal couplings $\Gbicouplings{\mu}{\nu} \coloneqq \Gcouplings{\mu}{\nu} \cap \Ganticouplings{\mu}{\nu}$.
\end{definition} 

\begin{remark}
    \label{relabelling_remark}
    One easily sees from the above definition that if one relabels the vertices of $G$, which includes keeping track of the initial associated spaces $\mc{X}_v$, $\mc{Y}_v$, this will not affect the set of induced $G$-causal couplings. Thus, the definition above is invariant under relabelling.
\end{remark}

\begin{remark}
    Definition \ref{G_causal_definition_2} and the equivalence of 1) and 2) in Proposition \ref{prop:Gcomp_characterisation} show that $\mu \in \Gprobabilities{\mc{X}}$ if and only if for some (and hence all) $y \in \mc{Y}$ the set $\Gcouplings{\delta_y}{\mu} \subseteq \curlybrackets{\delta_y \otimes \mu}$ is non-empty. Thus, $G$-compatibility is a natural condition if we want graph (bi)causal couplings to exist, and it is indeed sufficient: see the first part of Proposition \ref{equivalent_characterisations}.
\end{remark}

Similarly to graph compatibility, the notions of graph (bi)causality from Definition \ref{G_causal_definition} have characterisations in terms of conditional independence relations and disintegrations. These were shown in \cite{Eckstein2023CausalGraphs}, and we recall the ones necessary for our proofs in Propositions \ref{prop:equivalent_Gcausal} and \ref{equivalent_characterisations} below.

\begin{proposition}
    \label{prop:equivalent_Gcausal}
    Suppose $G$ is a DAG, $\mu \in \Gprobabilities{\mc{X}}$, $\nu \in \Gprobabilities{\mc{Y}}$, and $\pi \in \couplings{\mu}{\nu}$. Then the following are equivalent:

    \begin{enumerate}
        \item[1)] $\pi$ is $G$-causal in the sense of Definition \ref{G_causal_definition};
        \item[2)] $\pi$ is $G$-causal in the sense of Definition \ref{G_causal_definition_2}.
    \end{enumerate}
    If $G$ is a sorted $DAG$, the above are further equivalent to the following:
    \begin{enumerate}
        \item[3)] When $(X,Y) \sim \pi$ the following conditional independence relations hold for $1 \leq i \leq n$:
            $$\condindep{Y_i}{X_i, X_{\parents{i}}, Y_{\parents{i}}}{(X, Y_{1:i-1})}.$$
    \end{enumerate}
\end{proposition}
\begin{proof}
    Since the first two conditions are invariant under relabelling, we may assume without loss of generality that $G$ is sorted. That 2) is equivalent to 3) was then shown in \cite[Theorem 3.4]{Eckstein2023CausalGraphs}. 

    That 1) is equivalent to 3) is a consequence of the equivalence of 1) and 4) in Proposition \ref{prop:Gcomp_characterisation}. Indeed, using the notation of Definition \ref{G_causal_definition}, we see $1 \leq \ldots \leq n \leq 1' \leq \ldots \leq n'$ is a topological ordering for $G^c$, and so $\pi \in \specificGprobabilities{G^c}{\mc{X} \times \mc{Y}}$ encodes precisely the conditional independence relations $\condindep{X_i}{X_{\parents{i}}}{X_{1:i-1}}$ for $2 \leq i \leq n$ (as the parents of $i$ in $G^c$ are the same as in $G$, namely $\parents{i}$) and $\condindep{Y_i}{X_i, X_{\parents{i}}, Y_{\parents{i}}}{(X_{1:n}, Y_{1:i-1})}$ for $1 \leq i \leq n$ (the parents of $i'$ in $G^c$ are precisely $\curlybrackets{i} \cup \parents{i} \cup \setdef{j'}{j \in \parents{i}}$, corresponding to the random variables $X_i, X_{\parents{i}}, Y_{\parents{i}}$). But, by Proposition \ref{prop:Gcomp_characterisation} again, the first set of conditional independence relations always holds as $\mu \in \Gprobabilities{\mc{X}}$.
\end{proof}

\begin{proposition}
    \label{equivalent_characterisations}
    Suppose $G$ is a sorted DAG, $\mu \in \Gprobabilities{\mc{X}}$, $\nu \in \Gprobabilities{\mc{Y}}$, and $\pi \in \couplings{\mu}{\nu}$. Then $\mu \otimes \nu \in \Gbicouplings{\mu}{\nu} \subseteq \Gcouplings{\mu}{\nu}$.
    Moreover, the following are equivalent:
    
    \begin{enumerate}
        \item[1)] $\pi \in \Gbicouplings{\mu}{\nu}$;
        \item[2)] When $(X,Y) \sim \pi$ the following conditional independence relations hold for $2 \leq i \leq n$:
        $$\condindep{(X_i, Y_i)}{X_{\parents{i}}, Y_{\parents{i}}}{(X_{1:i-1}, Y_{1:i-1})}, \quad \condindep{X_i}{X_{\parents{i}}}{Y_{\parents{i}}}, \quad \condindep{Y_i}{Y_{\parents{i}}}{X_{\parents{i}}};$$
        
        \item[3)] $\pi$ may be disintegrated as $\pi(dx, dy) = \bigotimes_{i=1}^n \pi_{i}\brackets{dx_i, dy_i \mid x_{\parents{i}}, y_{\parents{i}}},$
        with $$\pi_{i}\brackets{dx_i, dy_i \mid x_{\parents{i}}, y_{\parents{i}}} \in \couplings{\mu_{i}\brackets{dx_i \mid  x_{\parents{i}}}}{\nu_{i}\brackets{dy_i \mid  y_{\parents{i}}}}$$ $\pi$-a.s. for all $1 \leq i \leq n$.
    \end{enumerate}
\end{proposition}
\begin{proof}
    See \cite[Corollary 3.5, Proposition 3.6]{Eckstein2023CausalGraphs}. 
\end{proof}

\subsection{Gluing Properties}
\label{main_gluing}

The notion of gluing is a classical tool in the theory of optimal transport, often used to show that the triangle inequality holds for different variants of Wasserstein distances. 

\begin{definition}[Gluing]
    Let $\mc{A}, \mc{B}, \mc{C}$ be Polish spaces and let $\pi_1 \in \probabilities{\mc{A} \times \mc{B}}$, $\pi_2 \in \probabilities{\mc{B} \times \mc{C}}$ such that $\pushforward{\proj{\mc{B}}}{\pi_1} = \pushforward{\proj{\mc{B}}}{\pi_2} \eqqcolon \nu \in \probabilities{\mc{B}}$. Let $\pi_2$ be disintegrated as $\pi_2(db, dc) = \nu(db) \otimes \kappa(dc \mid b)$ for a stochastic kernel $\kappa$.

    Then $\bigglue{\pi_1}{\pi_2}$ is defined as the measure 
    $$\bigglue{\pi_1}{\pi_2}(da, db, dc) \coloneqq \pi_1(da, db) \otimes \kappa(dc \mid b) \in \probabilities{\mc{A} \times \mc{B} \times \mc{C}},$$ 
    and is called the \emph{conditional independent product} of $\pi_1$, $\pi_2$. Its $\mc{A}\times\mc{C}$ marginal, $\pi_1 \glue \pi_2 \coloneqq \pushforward{\proj{\mc{A}\times\mc{C}}}{\bigglue{\pi_1}{\pi_2}}$ is called the \emph{gluing} of $\pi_1$, $\pi_2$.
\end{definition}

The classical gluing lemma \cite[Lemma 7.6]{Villani2003TopicsTransportation} states that if $\pi \in \couplings{\mu}{\nu}$ and $\tilde{\pi} \in \couplings{\nu}{\eta}$ for marginals $\mu, \nu, \eta$, then $\pi \glue \tilde{\pi} \in \couplings{\mu}{\eta}$, thus providing a way of combining two transport plans, one from $\mu$ to $\nu$, and one from $\nu$ to $\eta$ into a transport plan from $\mu$ to $\eta$. We are interested in whether graph causal restrictions are preserved under the gluing operation.

\begin{definition}[(Bi)causal gluing property]
    We say that the directed graph $G$ has the \emph{(bi)causal gluing property} if whenever $\mu \in \Gprobabilities{\mc{X}}$, $\nu \in \Gprobabilities{\mc{Y}}$, $\eta \in \Gprobabilities{\mc{Z}}$ for some products of Polish spaces $\mc{X}, \mc{Y}, \mc{Z}$, and $\pi \in \Gcouplings{\mu}{\nu}$, $\tilde{\pi} \in \Gcouplings{\nu}{\eta}$ (respectively $\pi \in \Gbicouplings{\mu}{\nu}$, $\tilde{\pi} \in \Gbicouplings{\nu}{\eta}$), we have that $\pi \glue \tilde{\pi} \in \Gcouplings{\mu}{\eta}$ (respectively $\pi \glue \tilde{\pi} \in \Gbicouplings{\mu}{\eta}$).
\end{definition}

To characterise fully causal gluing properties for DAGs, we need to define the notion of a \emph{split}, also known as an \emph{unshielded fork}.

\begin{definition}[Splits]
\label{split_def}
    Let $G = (V, E)$ be a directed graph. If $i, j, k \in V$ are such that $i \in \parents{j} \cap \parents{k}$ but $j \notin \parents{k}$ and $k \notin \parents{j}$, we call the triplet $(i,j,k)$ a \emph{split}. If no such triplet exists, we say $G$ contains \emph{no splits}.
\end{definition}

\begin{theorem}
\label{glue_characterisation}
    For a DAG $G$, the following are equivalent:
    \begin{enumerate}
        \item $G$ has the causal gluing property;
        \item $G$ has the bicausal gluing property;
        \item $G$ is transitively closed and has no splits.
    \end{enumerate}
\end{theorem}

\begin{remark}
\label{causal_implies_bicausal_gluing}
    If $\pi'$ is the pushforward of the measure $\pi$ under the map $(x,y) \mapsto (y,x)$, $\tilde{\pi}'$ is the pushforward of the measure $\tilde{\pi}$ under the map $(y,z) \mapsto (z,y)$, then the pushforward of $\pi \glue \tilde{\pi}$ under the map $(x,z) \mapsto (z,x)$ is $ \tilde{\pi}' \glue \pi'$. Thus, one easily sees that if $G$ has the causal gluing property, it also has the bicausal gluing property.
\end{remark}

The proof of Theorem \ref{glue_characterisation} is given in Section \ref{proof_gluing}. 
An interesting and novel feature of graph causal optimal transport, as highlighted by Theorems \ref{main_result_metric} and \ref{glue_characterisation}, is that while the bicausal gluing property is sufficient for the triangle inequality to hold (by a trivial modification of the proof of \cite[Theorem 7.3]{Villani2003TopicsTransportation}), it is not necessary. Indeed, the graph $\Gsplit$, with vertex set $\curlybrackets{1,2,3}$ and edge set $\curlybrackets{(1,2), (1,3)}$, is transitively closed but contains a split. Hence, $W_{\Gsplit,p}$ is a proper distance, even if $\Gsplit$ does not have the (bicausal) gluing property. Working to develop tools that will enable us to prove this, and Theorem \ref{main_result_metric} more generally, we will now look at denseness of Monge couplings and subsequently introduce lifted graph bicausal couplings. 

The inductive construction technique used in the proof of Theorem \ref{glue_characterisation} can also be used to understand weak closure properties of $\Gprobabilities{\mc{X}}$, the set of graph compatible measures. By \cite[Proposition 3.3]{Eckstein2023CausalGraphs}, we know that $\Gprobabilities{\mc{X}}$ is closed in total variation, and that if $\mc{X}$ is a discrete Polish space, $\Gprobabilities{\mc{X}}$ is also weakly closed. However, in general $\Gprobabilities{\mc{X}}$ is not necessarily weakly closed and \cite[Proposition 3.3]{Eckstein2023CausalGraphs} provides an example for the case $G = \Gmarkov$. We are able to provide a full characterisation of the graphs $G$ for which $\Gprobabilities{\mc{X}}$ is weakly closed for a general Polish space $\mc{X}$.
\begin{proposition}\label{prop:P_G_closed}
    If $G$ is transitively closed and contains no splits, then $\Gprobabilities{\mc{X}}$ is weakly closed. Otherwise, $\Gprobabilities{\mbb{R}^n}$ is not weakly closed.
\end{proposition}
\begin{proof}
The proof is similar to the proof of Theorem \ref{glue_characterisation}. If $G$ does not have the transitive closure or no splits properties, Example \ref{closedness_counterexamples} shows $\Gprobabilities{\mbb{R}^n}$ is not closed. Otherwise, combining Lemmas \ref{closedness_independent}, \ref{closedness_add}, and Lemma \ref{make_graph_two_operations}, we see that $\Gprobabilities{\mc{X}}$ is weakly closed.  
\end{proof}

\subsection{Denseness of $G$-biadapted Monge Couplings}

Monge couplings between marginals $\mu \in \probabilities{\mc{X}}$ and $\nu \in \probabilities{\mc{Y}}$ are transport plans of the form $\pushforward{(\mathrm{Id},T)}{\mu}$ which are supported on the graph of some measurable map $T : \mc{X} \to \mc{Y}$ such that $\pushforward{T}{\mu} = \nu$. We denote by $\monge{\mu}{\nu} \subseteq \couplings{\mu}{\nu}$ the set of all such Monge couplings. 

They are named after Gaspard Monge, who originally introduced the optimal transport problem in \cite{Monge1781MemoireRemblais} as a minimisation of an expected cost over $\monge{\mu}{\nu}$. In modern OT theory, following \cite{Kantorovich1942}, the optimisation is done over couplings $\couplings{\mu}{\nu}$ which is a convex set, resulting in a well posed problem with a natural dual \citep{Rubinstein}, see also \citep{Kellerer1984duality} and the discussion therein. Nevertheless, Monge couplings remain an important object. In some cases the optimal coupling is a Monge one,  most notably for the squared distance cost on $\mbb{R}^d$, where the seminal result of \cite{Brenier1991} showed it is supported on a gradient of a convex function. More generally, provided $\mu$ contains no atoms, $\monge{\mu}{\nu}$ is weakly dense in $\couplings{\mu}{\nu}$, and thus the Monge and Kantorovich formulations of the optimal transport problem agree (see e.g. \cite{Pratelli2007MongeKantorovich})
$$\inf_{\pi \in \monge{\mu}{\nu}} \int c(x,y) \,d\pi(x,y) = \inf_{\pi \in \couplings{\mu}{\nu}} \int c(x,y) \,d\pi(x,y),$$
for any continuous and bounded $c: \mc{X} \to \mc{Y} \to \mbb{R}$. A similar denseness result holds if we restrict to Monge couplings supported on graphs of bijective maps, provided now both $\mu$ and $\nu$ contain no atoms \cite[Theorem 2.6]{Beiglbock2022Biadapted}. We generalise this denseness result to graph bicausal couplings. Our Monge maps are restricted to respect the graph structure in the sense of the following definition:

\begin{definition}[$G$-adapted maps]
    We say a function $T : \prod_{i=1}^n\mc{X}_i \to \prod_{i=1}^n\mc{Y}_i$ is \emph{$G$-adapted} if it can be written as
    $$T(x) = (T_i(x_i, x_{\parents{i}}))_{1 \leq i \leq n}$$
    for some functions $T_i : \mc{X}_i \times \mc{X}_{\parents{i}} \to \mc{Y}_i$. We call $T$ \emph{$G$-biadapted} if it is $G$-adapted, bijective, and its inverse is itself $G$-adapted.\\
    Given $\mu \in \probabilities{\mc{X}}$, $\nu \in \probabilities{\mc{Y}}$, we denote the set of all Monge couplings of $\mu$ and $\nu$ induced by $G$-biadapted maps as
    $$\Gmonge{\mu}{\nu} \coloneqq \setdef{\pushforward{(\mathrm{Id},T)}{\mu}}{T \, \, G\mathrm{-biadapted}, \pushforward{T}{\mu}=\nu}.$$
\end{definition}
Note that $\Gmonge{\mu}{\nu}\subseteq \Gbicouplings{\mu}{\nu}$ for $\mu\in \Gprobabilities{\mc{X}}$, $\nu\in \Gprobabilities{\mc{Y}}$. This follows directly by verifying the conditional independence relations using the structure of $G$-biadapted maps.
\begin{remark}
     With $G=\Glinear{n}$, the definition above recovers the notions of adapted and biadapted maps from \cite{Beiglbock2022Biadapted}. Theorem \ref{main_dense} below can then be seen as a generalisation of \cite[Theorem 3.11]{Beiglbock2022Biadapted}.
\end{remark}

We record, for future use, the following self-evident property of $G$-adapted maps:
\begin{lemma}
    \label{closed_composition}
    Suppose $G$ is a transitively closed DAG and $\mc{Y}=\mc{X}$. Then the set of $G$-adapted maps, and the set of $G$-biadapted maps, are closed under composition.
\end{lemma}

As is the case of classical optimal transport, in order to prove a version of denseness of Monge couplings, we need to assume the measures are sufficiently regular. 

\begin{assumption}
    \label{continuity_assumption}
    We say a measure $\mu \in \Gprobabilities{\mc{X}}$ satisfies the \emph{continuity assumption} if it has a disintegration
    \begin{equation}
        \label{disintegration_equation}
    \mu(dx) = \bigotimes_{i=1}^n \mu_i(dx_i \vert \xparents{i})
    \end{equation}
    for some stochastic kernels $\mu_i(dx_i \vert \xparents{i})$ which contain no atoms for all $\xparents{i} \in \mc{X}_{\parents{i}}$, under some sorted relabelling of $G$.
\end{assumption}
We can now state the key denseness result. Its proof is deferred to Section \ref{proof_dense}.
\begin{theorem}
    \label{main_dense}
    Suppose $G$ is a transitively closed DAG. Let $\mu \in \Gprobabilities{\mc{X}}$, $\nu \in \Gprobabilities{\mc{Y}}$ satisfy Assumption \ref{continuity_assumption}. Then the set $\Gmonge{\mu}{\nu}$ is weakly dense in $\Gbicouplings{\mu}{\nu}$. 
    
    Let $p\geq1$. If furthermore $\mu$, $\nu$ have finite $p$\textsuperscript{th} moments, then we have $W_p$-denseness.
\end{theorem}

\begin{remark}
    As shown in Proposition \ref{prop:Gcomp_characterisation}, $\mu \in \Gprobabilities{\mc{X}}$ if and only if it has a disintegration of the form (\ref{disintegration_equation}) under any sorted relabelling. The actual restriction in Assumption \ref{continuity_assumption} is therefore on the absence of atoms in the disintegration kernels which is the natural extension from the bicausal case in  \cite[Assumption 3.7]{Beiglbock2022Biadapted}.
\end{remark}

\subsection{Lifted Graph Bicausality}
\label{lifted_section}

As noted in Section \ref{main_gluing} there are graphs $G$ for which the graph causal Wasserstein distance is a true metric, while a version of the gluing lemma does not hold because the gluing of two $G$-(bi)causal couplings is not necessarily $G$-(bi)causal. To avoid this difficulty in gluing, we introduce a slightly relaxed notion of $G$-bicausality, which we call lifted $G$-bicausality, that can be easily seen to satisfy a version of the gluing lemma when $G$ is transitively closed. They are projections of $G$-biadapted Monge couplings between measures on a lifted space, allowing for independent randomisation at every node. 

To make this more precise, we first set up some notation for the lifted space. Let $\Hat{\mc{X}}_i \coloneqq \mc{X}_i \times [0,1]$, $\Hat{\mc{X}} \coloneqq \prod_{i=1}^n \Hat{\mc{X}}_i$, with similar notation for $\Hat{\mc{Y}}_i$, $\Hat{\mc{Y}}$. We denote by $(x_i, u_i)$ a generic element of $\Hat{\mc{X}}_i$, $(y_i, v_i)$ a generic element of $\Hat{\mc{Y}}_i$ and, slightly abusing the ordering, $(x_{1:n}, u_{1:n})$ and $(y_{1:n}, v_{1:n})$ are generic elements of $\Hat{\mc{X}}$ and $\Hat{\mc{Y}}$ respectively. For $k \geq 1$ we write $\lambda^k$ for the uniform measure on $[0,1]^k$. For a measure $\mu \in \probabilities{\mc{X}}$, we denote by $\Hat{\mu} \in \probabilities{\Hat{\mc{X}}}$ the pushforward of the measure $\mu \otimes \lambda^n$ under the reordering $(x, u) \mapsto (x_1, u_1, \ldots, x_n, u_n)$. Similarly, up to the obvious reordering, for $\nu \in \probabilities{\mc{Y}}$, $\Hat{\nu} \coloneqq \nu \otimes \lambda^n \in \probabilities{\Hat{\mc{Y}}}$. Note that $\mu\in \Gprobabilities{\mc{X}}$ implies $\Hat{\mu}\in \Gprobabilities{\Hat{\mc{X}}}$.

\begin{definition}[Lifted $G$-bicausal coupling]
    \label{new_definiton}
    For $\mu \in \probabilities{\mc{X}}$, $\nu \in \probabilities{\mc{Y}}$, we say a coupling $\pi \in \couplings{\mu}{\nu}$ is \emph{lifted $G$-bicausal} if there exists $\Hat{\pi} \in \Gmonge{\Hat{\mu}}{\Hat{\nu}}$ such that $\pi = \pushforward{\proj{\mc{X}\times\mc{Y}}}{\Hat{\pi}}$. We write $\newGbicouplings{\mu}{\nu}$ for the set of lifted $G$-bicausal couplings between $\mu$ and $\nu$.\\
    More generally, the set of \emph{generalised lifted $G$-bicausal} couplings is given by
    $$\tildeGbicouplings{\mu}{\nu} := \left\{\pushforward{\proj{\mc{X}\times\mc{Y}}}{\Hat{\pi}} : \Hat{\pi} \in \Gbicouplings{\Hat{\mu}}{\Hat{\nu}}\right\}.
    $$    
\end{definition}
The definition of lifted couplings was inspired by \cite{Beiglbock2022Biadapted}, in which bicausal couplings are written as projections of biadapted Monge couplings on a lifted space in order to obtain a denseness result. 
We show below that for transitively closed graphs, lifted graph bicausality is indeed a relaxation of graph bicausality. The more general definition of generalised lifted graph bicausal couplings is necessary to deal with non-transitively closed DAGs, but in the case of transitively closed DAGs it agrees with the notion of lifted graph bicausal couplings. We note that for non-transitively closed DAGs the two are generally different, see Example \ref{example:generalised_different}.

\begin{theorem}\label{thm:lifted}
    Suppose $G$ is a transitively closed DAG and $\mu \in \Gprobabilities{\mc{X}}$, $\nu \in \Gprobabilities{\mc{Y}}$. Then
$$\Gbicouplings{\mu}{\nu} \subseteq \newGbicouplings{\mu}{\nu}.$$
    Moreover, if $\Hat{\pi} \in \newGbicouplings{\mu}{\nu}$, then there exists a probability measure $\rho \in \probabilities{\probabilities{\mc{X} \times \mc{Y}}}$ such that $\Hat{\pi} = \int \pi \,d\rho(\pi)$ and $\rho$ is supported on $\Gbicouplings{\mu}{\nu}$.\\ In particular, the following hold:    
    \begin{equation}\label{eq:liftedinclusions}
    \Gbicouplings{\mu}{\nu} \subseteq \newGbicouplings{\mu}{\nu} = \tildeGbicouplings{\mu}{\nu} \subseteq \closedconvex{\Gbicouplings{\mu}{\nu}},
    \end{equation}
    and all inclusions above may be strict.
\end{theorem}
The important question of when equality holds in the first inclusion in (\ref{eq:liftedinclusions}) is addressed next. 
\begin{theorem}\label{thm:lifted_nosplits}
Suppose $G$ is a transitively closed DAG and $\mu \in \Gprobabilities{\mc{X}}$, $\nu \in \Gprobabilities{\mc{Y}}$. If $G$ has no splits then 
    $$\Gbicouplings{\mu}{\nu} = \newGbicouplings{\mu}{\nu}$$
    and is weakly closed. Otherwise, if $G$ has a split, the inclusion $\Gbicouplings{\mu}{\nu} \subseteq \newGbicouplings{\mu}{\nu}$ may be strict, and $\Gbicouplings{\mu}{\nu}$ may not be weakly closed.
\end{theorem}

\begin{remark}\label{rk:monge_on_extended_space}
The first inclusion in Theorem \ref{thm:lifted} can be reformulated by saying that 
if $\pi \in \Gbicouplings{\mu}{\nu}$, then there exists a $G$-biadapted map $T : \Hat{\mc{X}} \to \Hat{\mc{Y}}$ such that
$\Hat{\pi} \coloneqq \pushforward{(\Id{}, T)}{\Hat{\mu}} \in \Gbicouplings{\Hat{\mu}}{\Hat{\nu}}$ (i.e., $\pushforward{T}{\Hat{\mu}} = \Hat{\nu}$) and $\pushforward{\proj{\mc{X} \times \mc{Y}}}{\Hat{\pi}} = \pi.$
When $G=G_n$ is the linear graph, this recovers the forward direction of \cite[Theorem 3.5]{Beiglbock2022Biadapted}. And, as a special case when $n=1$, we see that any coupling can be seen as a projection of a Monge coupling induced by a bijective map on the lifted space, a well-known result in classical optimal transport. As $G_n$ has no splits, Theorem \ref{thm:lifted_nosplits} recovers the converse direction in \cite[Theorem 3.5]{Beiglbock2022Biadapted}, namely $\specificGbicouplings{\Glinear{n}}{\mu}{\nu} = \newspecificGbicouplings{\Glinear{n}}{\mu}{\nu}$. For general transitively closed DAGs however this equality no longer holds true, see the discussion at the start of Section \ref{proof:lifted_nosplits}.
\end{remark}

\begin{remark}
    We note that as a consequence of Theorem \ref{thm:lifted_nosplits} above and the counterexample in \cite[Proposition 3.6(iv)]{Eckstein2023CausalGraphs}, which extends to any non-transitively closed DAG as observed in Example \ref{example:split_not_closed} below, we see that $\Gbicouplings{\mu}{\nu}$ is weakly closed in general if and only if $G$ is transitively closed and contains no splits, thus complementing the results of \cite[Proposition 3.6]{Eckstein2023CausalGraphs}. 
\end{remark}

The proofs of Theorems \ref{thm:lifted} and \ref{thm:lifted_nosplits} are given, respectively, in Sections \ref{proof_lifted} and \ref{proof:lifted_nosplits}.
We note that from the definition, the following inclusions are immediate:
\begin{equation}\label{eq:obviousliftedinclusions}
    \Gbicouplings{\mu}{\nu} \subseteq \tildeGbicouplings{\mu}{\nu}\qquad \textrm{and}\qquad \newGbicouplings{\mu}{\nu} \subseteq \tildeGbicouplings{\mu}{\nu}.
\end{equation}

We now present a first important application of the results above.
\begin{corollary}
    \label{coro:GOT}
    Suppose $G$ is a transitively closed DAG, $\mu\in \Gprobabilities{\mc{X}},\nu\in \Gprobabilities{\mc{Y}}$ and $\xi: \mc{X}\times \mc{Y}\to \mbb{R}$ is measurable with $\xi(x,y)\geq a(x)+b(y)$ for two upper semicontinuous functions $a\in L^1(\mu),b\in L^1(\nu)$. Then
    $$ \GOT{\mu}{\nu}= \inf_{\pi\in \newGbicouplings{\mu}{\nu}} \int_{\mc{X}\times \mc{Y}} \xi(x,y)d \pi(x,y) = \inf_{\pi\in \tildeGbicouplings{\mu}{\nu}} \int_{\mc{X}\times \mc{Y}} \xi(x,y)d \pi(x,y).
    $$
    If $G$ has no splits and $\xi$ is lower semicontinuous, then $ \GOT{\mu}{\nu}$ is attained by some $\pi \in \Gbicouplings{\mu}{\nu}$.
\end{corollary}
\begin{proof}
    Integrability assumptions on $a,b$ allow us to replace $\xi$ by $\xi-a-b\geq 0$ and use Fubini. 
    Theorem \ref{thm:lifted} then yields the first assertion, allowing us to replace $\Gbicouplings{\mu}{\nu}$ in \eqref{eq:GOT} by $\newGbicouplings{\mu}{\nu} = \tildeGbicouplings{\mu}{\nu}$.
    The second part follows from classical OT proofs, see \cite[Thm.~4.1]{Villani2009OptimalNew}, and the fact that $\Gbicouplings{\mu}{\nu}$ is closed by Theorem \ref{thm:lifted_nosplits}.
\end{proof}

Taking $\mc{X} = \mc{Y}$ and $c = d_{\mc{X}}^p$ we obtain 
  \begin{equation}
    \label{alternative_Wasserstein}
  W_{G,p}(\mu, \nu)^p = \inf_{\Hat{\pi} \in \newGbicouplings{\mu}{\nu}} \int d_{\mc{X}}(x,y)^p \, d\Hat{\pi}(x,y).
  \end{equation}
We use this representation to prove Theorem \ref{main_result_metric}. 
The key advantage of lifted $G$-bicausal couplings is that they have nice gluing properties provided $G$ is transitively closed: if $\pi \in \newGbicouplings{\mu}{\nu}$ and $\pi' \in \newGbicouplings{\nu}{\eta}$ for marginals $\mu, \nu, \eta$, then $\pi \glue \pi' \in \newGbicouplings{\mu}{\eta}$, as shown in Lemma \ref{bicausal_gluing_lemma}, and using the key observations noted in Lemma \ref{closed_composition} above. 
With this lifted graph bicausal version of the gluing lemma in place, one can then replicate the classical proof of the triangle inequality for Wasserstein distances to obtain Theorem \ref{main_result_metric} (see Section \ref{proof_metric}).

To close this section, we consider DAGs which are not transitively closed. We know that $W_{G,p}$ may no longer be a metric in this case and Corollary \ref{coro:GOT} actually suggests a possible remedy and an improved definition of $G$-causal Wasserstein distances.
\begin{definition}[Lifted graph causal Wasserstein distances]
\label{def:lifted_G-Wass}
    Let $G$ be a directed graph and $p \geq 1$. The $p$\textsuperscript{th} order \emph{lifted $G$-causal Wasserstein distance} between $\mu, \nu \in \Gpprobabilities{p}{\mc{X}}$ is given by
    $$\widetilde{W}_{G,p}(\mu, \nu)^p \coloneq \inf_{\pi \in \tildeGbicouplings{\mu}{\nu}} \int d_{\mc{X}}(x,y)^p \, d\pi(x,y).$$
\end{definition}
It turns out this object is always a metric for DAGs and we can understand precisely why. To this end, for a general DAG, consider its transitive closure, which is obtained explicitly by adding in all ancestors of each vertex to its parent set.
\begin{definition}[Transitive closure]\label{def:transitive_closure}
    For a directed graph $G = (V, E)$, we define its \emph{transitive closure}, denoted by $\Gbar$, as the directed graph with vertex set $V$ and edge set $\closure{E} = \setdef{(i,j)}{j \in V, i \in \hist{j} \setminus \curlybrackets{j}}$.
\end{definition}
It is easy to see that $\Gbar$ is the DAG with the least number of edges on the same vertex set which is transitively closed and has $G$ as a subgraph. When $G$ is a sorted DAG with $n$ vertices, then $G\subseteq \Gbar\subseteq G_n$. 
\begin{proposition}
\label{prop:nonHP_metric}
    Let $G$ be a DAG, $p \geq 1$, $\mu, \nu \in \Gpprobabilities{p}{\mc{X}}$. Then $\widetilde{W}_{G,p}(\mu, \nu) = W_{\Gbar, p}(\mu, \nu)$ and is a metric on $\Gpprobabilities{p}{\mc{X}}$. 
\end{proposition}
\begin{proof}
    The key observation is that $\specificGbicouplings{\Gbar}{\mu}{\nu} \subseteq \tildeGbicouplings{\mu}{\nu}$, see Lemma \ref{lem:nonHP_inclusion}. This shows that $W_{\Gbar,p}\geq \widetilde{W}_{G,p}$. We know from Corollary \ref{coro:GOT} that $W_{\Gbar,p}= \widetilde{W}_{\Gbar,p}$. Finally, since $\Gbar$ has more edges than $G$, we have $\tildeGbicouplings{\mu}{\nu}\subseteq \tildespecificGbicouplings{\Gbar}{\mu}{\nu}$ and hence $\widetilde{W}_{\Gbar,p}\leq \widetilde{W}_{G,p}$, giving the desired equality by sandwiching. $W_{\Gbar,p}$ is a metric by Theorem \ref{main_result_metric}.
\end{proof}
Note that $\widetilde{W}_{\Gbar,p}$ is, in a certain way, the minimal object which cures the non-metric issue of $W_{G,p}$ in general. Specifically, if we look for $G'$, with $G\subseteq G'$, for which $W_{G',p}$ is a metric, then any $G' \subsetneq \Gbar$ is not transitively closed, so it contains $\Gmarkov$ as a proper subgraph and hence, by the arguments in the proof of Theorem \ref{main_result_metric}, there is a choice of $(\mc{X},d_{\mc{X}})$ for which $W_{G',p}$ does not satisfy the triangle inequality on $\Gprobabilities{p}{\mc{X}}$. Hence, $G'=\Gbar$ is the minimal choice.

\subsection{Complexity and Prospects for Duality for Graph Causal OT}
\label{sec:duality}

Duality in the classical optimal transport theory reformulates the minimisation problem over couplings as a maximisation problem over pairs of potentials, see \cite{RamachandranRuschendorf1995} and the references therein. It has played a major role, both for the theoretical aspects, such as in characterising optimal couplings via certain geometric properties of their support, and for its numerical methods, such as for the \citep{Cuturi2013Sinkhorn} Sinkhorn algorithm. More recently, for causal optimal transport, duality was first established in \cite{Backhoff2016CausalDiscrete}. 

A natural question is what is the dual formulation to the general graph causal OT problem (\ref{eq:GOT}), outside of the special cases of classical OT (corresponding to $G = G_1$) and causal/adapted OT (corresponding to $G = G_n$). While an abstract theoretical formulation can be described (by looking at the closed convex hull of the constraint set $\closedconvex{\Gbicouplings{\mu}{\nu}}$), we argue that a tractable dual formulation is, in general, most likely out of reach.

Indeed, for the case that $G$ is the graph with two independent nodes $(\curlybrackets{1,2}, \varnothing)$ and marginals $\mu = \mu_1 \otimes \mu_2,  \nu = \nu_1 \otimes \nu_2 \in \Gprobabilities{\mbb{R}^2}$, the set of permissible couplings for the graph causal OT problem (\ref{eq:GOT}) is $\setdef{\pi_1 \otimes \pi_2}{\pi_i \in \couplings{\mu_i}{\nu_i}}$ (by Proposition \ref{equivalent_characterisations}). When $\mu_i, \nu_i$ are supported uniformly on $m \geq 1$ points each, the problem thus reduces to the \emph{bilinear assignment problem (BAP)} of \cite{altman1968bilinear}: 
\begin{align*}
\min_{\pi^{(1)}_{ij}, \pi^{(2)}_{kl}} \quad & \sum_{1 \leq i,j,k,l \leq m}
  \xi_{ijkl} \pi^{(1)}_{ij} \pi^{(2)}_{kl} \\
\text{s.t.} \quad
  & \sum_{j'=1}^{m} \pi^{(1)}_{ij'} = 1, \, \sum_{i'=1}^{m} \pi^{(1)}_{i'j} = 1,\qquad 1 \leq i, j \leq m, \\
  & \sum_{l'=1}^{m} \pi^{(2)}_{kl'} = 1, \, \sum_{k'=1}^{m} \pi^{(2)}_{k'l} = 1,\qquad 1 \leq k, l \leq m, \\
  & \pi^{(1)}_{ij}, \pi^{(2)}_{kl} \in \{0,1\}, \qquad \quad \,\,\,\, 1 \leq i,j,k,l \leq m.
\end{align*}
If we removed the causal constraint we would have obtained a simple LP problem, solvable in polynomial time. In contrast, 
the bilinear assignment problem is well-known to be NP-hard \citep{Custic2017BAP}. Finding a polynomial algorithm to solve it or its dual would imply (P)=(NP) and hence is unlikely to be simple. All DAGs, apart from those with a linear information structure $G_n$, have the graph with two independent nodes as a proper subgraph, and hence have its graph causal OT problem as a subproblem of their associated graph causal OT problem. The lack of a tractable duality can be seen to propagate from the BAP, to the graph causal OT for the DAG on two independent nodes, and to the graph causal OT problem for all DAGs except those with a linear information structure $G_n$.

The difficulty in describing the dual of the graph causal OT problem lies in the fact that the conditional independence constraints in Proposition \ref{equivalent_characterisations} involve random variables from both marginals. The causal constraints given by $G_n$ can however be rewritten as $\condindep{X}{X_{1:i}}{Y_{1:i}}$, $\condindep{Y}{X_{1:i}}{X_{1:i}}$ for $1 \leq i \leq n$, thus allowing the causal OT problem to be written as an LP program when its marginals are finitely supported, as shown in \cite[Lemma 3.11]{Eckstein2022ComputationalTransport}.

We note that when $\mu, \nu$ are finitely supported, but not necessarily uniformly on the same number of points, the graph causal OT problem on two independent nodes above coincides with the CO-OT problem \cite{Redko2020COOT}. Another related problem is the computation of the Gromov-Wasserstein distance \citep{Memoli2011GromovWassersteinMatching}, a popular OT distance between measures on different domains, whose elements cannot be easily compared. The marginals in this problem take the form $\mu = \mu_1 \otimes \mu_1$, $\nu = \nu_1 \otimes \nu_1$, while the set of permissible couplings is now $\setdef{\pi_1 \otimes \pi_1}{\pi_1 \in \couplings{\mu_1}{\nu_1}}$. When $\mu_i, \nu_i$ are supported uniformly on $m \geq 1$ points each, one recovers a \emph{quadratic assignment problem (QAP)} \citep{Cela1998QAP}, corresponding to adding the constraint that $\pi^{(1)}_{ij} = \pi^{(2)}_{ij}$ to the BAP. The QAP itself is also NP-hard and by the same token as above, a tractable dual formulation for Gromov-Wasserstein OT is not accessible. At the same time, dual formulations are possible and different variants could be envisaged to obtain specific types of results. \cite{Zhang2022GWDuality} provide a recent illustration: whilst their dual formulation is not necessarily numerically tractable itself, it lends itself to the study of entropic relaxations and statistical properties such as sample complexity. Extending such results to graph causal OT would be an interesting direction for future research.

\section{Dynamic Programming Principle for Graph Causal Distances}
\label{section_dpp}

We recall that the graph causal Wasserstein distance corresponding to the linear graph $G_n$ is called the adapted Wasserstein distance. In this section we establish that this distance is in fact shared between all the subrgaphs of $G_n$ which do not have merges, i.e., are perfect.
This can be achieved by a dynamic programming principle (DPP) approach. We present the key results here and defer further discussion, technical results and proofs to Section \ref{proof_dpp}. 

We assume (via a possible relabelling) that $G$ is a sorted DAG. In order to establish a convenient equivalent characterisation of perfect DAGs, we introduce the following notation:

\begin{definition}[Proximal parent]
    We define the \emph{proximal parent map} $\prox : V \to V$ by
    \begin{equation*}
        \prox (i) \coloneqq 
        \begin{cases}
            \max{\parents{i}} & \text{if } \parents{i} \neq \varnothing, \\
            i & \text{if } \parents{i} = \varnothing. 
        \end{cases}
    \end{equation*}
    If $\parents{i} \neq \varnothing$, we call $\prox(i)$ the \emph{proximal parent} of i.
\end{definition}

\begin{proposition}
\label{prop:perfection}
    Let $G$ be a sorted DAG. Then the following are equivalent:
    \begin{enumerate}[label=\arabic*)]
        \item $G$ is perfect, in the sense of Definition \ref{def:DAGthree_structures};
        \item for every $i \in V$, $\parents{i} \subseteq \parents{\prox(i)} \cup \curlybrackets{\prox(i)}$.
    \end{enumerate}
\end{proposition}
We note that the notion of perfection is well studied in the graphical models literature in the more general context of chain graphs, see \cite[Section 2.1]{LauritzenBook}. We prove the above Proposition, provide further discussion, comparison with \cite{LauritzenBook}, and yet another equivalent characterisation in Section \ref{proof_dpp}. 

The assumption that $G$ is perfect thus ensures that either $i \in V$ has no parents, or there exists a unique $j \in \parents{i}$ such that no other $k \in \parents{i}$ is a child of $j$. In this latter case, $j = \prox(i)$ and it is the unique node with the property that the longest possible directed path from it to node $i$ has length $1$. All other elements of $\parents{i}$ have a directed path of length at least $2$ from themselves to $i$. Interpreting proximity between nodes in this sense, $\prox(i)$ is the closest node to $i$ among $\parents{i}$, justifying the name of proximal parent.

We note that this characterisation of proximal parents is label-independent, and thus choosing a different relabelling of $G$ into a sorted DAG would not change whether $G$ is perfect. This was also clear from Definition \ref{def:DAGthree_structures}. Perfect DAGs look like forests of directed trees to which edges have been added according to the following rule: if an edge is added from a node $i$ to a node $j$, an edge must be added from $k$ to $j$ for every node $k$ on the shortest directed path in the tree from $i$ to $j$. Proposition \ref{condition_explained} makes this intuitive description precise.

Assume $G$ is perfect. Suppose our cost function $c : \mc{X} \times \mc{Y} \to [0, \infty]$ is \emph{$G$-separable} in the sense that it can be written as
\begin{equation}
    \label{eq:G-sep cost}
c(x,y) = \sum_{i=1}^n c_i(x_i, x_{\parents{i}}, y_i, y_{\parents{i}})
\end{equation}
for some $c_i : \mc{X}_{\curlybrackets{i} \cup \parents{i}} \times \mc{Y}_{\curlybrackets{i} \cup \parents{i}} \to  [0, \infty]$. 

Recalling the third characterisation of graph bicausality from Proposition \ref{equivalent_characterisations} and using the shorthand $k^b$ for a kernel $k(da \mid b)$, the $G$-bicausal transport problem may be written as
\begin{align*}
    &V^{G,c}(\mu, \nu) = \inf_{\pi \in \Gbicouplings{\mu}{\nu}} \int c(x,y) \, d\pi(x,y) \\
    &= \inf_{\pi \in \Gbicouplings{\mu}{\nu}} \int c_1 + \brackets{\int c_2 + \brackets{\ldots + \brackets{\int c_n \, d\pi_n^{x_{\parents{n}}, y_{\parents{n}}}} \ldots} \, d\pi_2^{x_{\parents{2}}, y_{\parents{2}}}} \, d\pi_1.
\end{align*}

This inspires a DPP approach similar to the one in \cite[Section 5]{Backhoff2016CausalDiscrete}. We define the following intermediate problems
$$V_n^{G,c}(\xparents{n}, \yparents{n}) \coloneqq \inf_{\pi \in \couplings{\mu_n^{\xparents{n}}}{\nu_n^{\yparents{n}}}} \int c_n(x_n, \xparents{n}, y_n, \yparents{n}) \, d\pi(x_n,y_n)$$
and recursively for $n - 1 \geq k \geq 1$, $V_k^{G,c}(\xparents{k}, \yparents{k}) \coloneqq$
\begin{align*}
    \coloneqq \inf_{\pi \in \couplings{\mu_k^{\xparents{k}}}{\nu_k^{\yparents{k}}}} \int c_k(x_k, \xparents{k}, y_k, \yparents{k}) + \sum_{i \in A_k} V_i^{G,c}(\xparents{i}, \yparents{i}) \, d\pi(x_k,y_k),
\end{align*}
where $A_k \coloneqq \curlybrackets{i \in V \setminus \curlybrackets{k} \vert \, \prox(i) = k}$.

Note that $A_k \subseteq \curlybrackets{k + 1, \ldots, n}$, so the recursion is well-defined provided the integrals exist, which we show in Proposition \ref{dpp_proof}. Moreover, as $G$ is perfect, we see that for each $i \in A_k$, $\parents{i} \subseteq \parents{k} \cup \curlybrackets{k}$, so indeed $V_k^{G,c}$ only depends on $\xparents{k}, \yparents{k}$. When $\parents{k} = \varnothing$,  $V_k^{G,c}$ is a constant.

Finally, we define $A_0 \coloneqq \setdef{k \in V}{\parents{k} = \varnothing} = \setdef{k \in V}{\prox(k) = k}$ and $$\dpp^{G,c}(\mu, \nu) \coloneqq \sum_{k \in A_0} V_k^{G,c}.$$

We can now state the DPP for $G$-bicausal transport.

\begin{proposition}
\label{dpp_proof}
    If $G$ is a sorted perfect DAG, $\mu \in \Gprobabilities{\mc{X}}$, $\nu \in \Gprobabilities{\mc{Y}}$, $c : \mc{X} \times \mc{Y} \to [0, \infty]$ is $G$-separable, and $c_k$ is lower semi-analytic for all $1 \leq k \leq n$, then $\dpp^{G,c}(\mu, \nu)$ is well-defined and 
    $$V^{G,c}(\mu, \nu) = \dpp^{G,c}(\mu, \nu).$$
\end{proposition}

As a consequence of the DPP, we obtain the following:

\begin{theorem}
    \label{AWisWG}
     Let $G$ be a sorted perfect DAG and $p \geq 1$. Assume the underlying metric on $\mc{X}$ takes the form $d_{\mc{X}}(x,y)^p = \sum_{i \in V} d_{\mc{X}_i}(x_i, y_i)^p$ for all $x,y \in \mc{X}$. If $\mu, \nu \in \Gpprobabilities{p}{\mc{X}}$, then 
     $$W_{G,p}(\mu, \nu) = AW_p(\mu, \nu).$$
\end{theorem}
    While bicausal transport plans are more general, allowing you to use all the information of the previous nodes at each step, Theorem \ref{AWisWG} shows that when $\mu$, $\nu$ are themselves compatible with the graph structure, this extra information results in no improvement. This readily implies that graph causal Wasserstein distances generated by all the ``intermediate" graphs also coincide:
    \begin{equation}\label{eq:intermediateGeq}
        W_{G,p}(\mu, \nu)=W_{G',p} (\mu, \nu),\quad \forall G\subseteq G' \subseteq G_n; \mu,\nu\in \Gprobabilities{\mc{X}}.
    \end{equation}
In particular, taking $G'=\Gbar$ we also have equality  $W_{G,p}=\widetilde{W}_{G,p}=AW_p$ on $\Gprobabilities{\mc{X}}$.

    Assumptions in Theorem \ref{AWisWG} are sharp in the sense that when the distance is not $G$-separable, or when $G$ is not assumed perfect, $W_{G,p}\neq AW_p$ in general, as shown in Examples \ref{ex:AWDnotGW} and \ref{ex:AWDnotGW2}.

\section{Proofs}
\label{proof_section}

This section contains the proofs of our main results presented so far, alongside numerous technical side results required. 
We start with some examples and counterexamples which are helpful in building intuition, as well as serving later on to give reverse directions in some of the proofs. The proofs are then presented in their logical order, which is different to the order in which we presented the results in Sections \ref{section_main_results}--\ref{section_dpp}. Specifically, we first prove Theorem \ref{thm:lifted}. This, in particular, implies Corollary \ref{coro:GOT} holds. It enables us to then prove Theorem \ref{main_result_metric}. 
Once we have the metric property, we can study topology and prove Theorem \ref{theorem:sametopology}, along with some stepping-stone results. 
We then turn to characterising when bicausal couplings have the gluing property and establish Theorem \ref{glue_characterisation}. This includes a series of intermediate results and counterexamples. It leads on to the proof of Theorem \ref{thm:lifted_nosplits}.
Next we focus on the denseness of Monge couplings, Theorem \ref{main_dense}. The proof relies on us being able to work on the lifted space and Theorem \ref{thm:lifted} in particular. 
Finally we move to the discussion of the dynamic programming principles for perfect sorted DAGs. We prove Proposition \ref{prop:perfection} and another characterisation in Proposition \ref{condition_explained}. We then show Proposition \ref{dpp_proof} holds. Isometric embedding is established in Theorem  \ref{aw_equality} and Theorem \ref{AWisWG} is deduced from it. We finish the proofs looking at the applications and the associated Lipschitz continuity results stated in Section \ref{section:applications}.

\subsection{Examples}
\label{section_counterexamples}

The following explicit counterexample complements the numerical counterexample provided in \cite[Appendix B]{Eckstein2023CausalGraphs} showing that $\Gmarkov$-causal Wasserstein distances do not generally respect the triangle inequality.

\begin{example}\label{example_counter_Markov}
    Consider the following three $\Gmarkov$-compatible measures:
    $$\mu = \frac{1}{4}\brackets{\delta_{(1,1,1)} + \delta_{(1,1,0)}+\delta_{(0,1,1)}+\delta_{(0,1,0)}},$$
    $$\nu = \frac{1}{4}\brackets{\delta_{(1,2,1)} + \delta_{(1,2,0)}+\delta_{(0,3,1)}+\delta_{(0,3,0)}},$$
    $$\eta =  \frac{1}{4}\brackets{\delta_{(1,4,1)} + \delta_{(1,4,0)}+\delta_{(0,4,1)}+\delta_{(0,4,0)}}.$$
    
    We consider the bijections $T_1$ between the supports of $\mu$ and $\nu$, and $T_2$ between the supports of $\nu$ and $\eta$ given by the diagram below:
\begin{align*}
    &(1,1,1) \overset{T_1}{\mapsto} (1,2,1) \overset{T_2}{\mapsto} (1,4,1) \\
    &(1,1,0) \mapsto (1,2,0) \mapsto (1,4,0) \\
    &(0,1,1) \mapsto (0,3,1) \mapsto (0,4,0) \\
    &(0,1,0) \mapsto (0,3,0) \mapsto (0,4,1) 
\end{align*}

 One can check that $\pushforward{(\Id{}, T_1)}{\mu} \in \specificGbicouplings{\Gmarkov}{\mu}{\nu}$, and $\pushforward{(\Id{}, T_2)}{\nu} \in \specificGbicouplings{\Gmarkov}{\nu}{\eta}$. Crucially, note that $\pushforward{(\Id{}, T_2 \circ T_1)}{\mu} \notin \specificGbicouplings{\Gmarkov}{\mu}{\eta}$. 
 
 The bijections split the supports of the three measures into four equivalence classes, whose elements form each row of the diagram above. Denote this equivalence relation by $\approx$.

 Let $\eps > 0$, $p \geq 1$. One may define a metric $d_{\eps}$ on $\supp{\mu} \cup \supp{\nu} \cup \supp{\eta}$ via $d_{\eps}(x,y) \coloneqq (\eps \mathds{1}_{x \approx y} + \frac{1}{\eps} \mathds{1}_{x \not\approx y})\mathds{1}_{x\neq y}$. 

 Under this metric, $W_{\Gmarkov,p}(\mu,\nu) \leq \brackets{\int d_{\eps}^p(x,T_1(x))\, d\mu(x)}^{\frac{1}{p}} = \eps$, and symmetrically  $W_{\Gmarkov,p}(\nu,\eta) \leq \eps$. 

 As both $\mu$ and $\eta$ have constant second coordinates, \cite[Corollary 3.5(iii)]{Eckstein2023CausalGraphs} implies that any $\pi \in \specificGbicouplings{\Gmarkov}{\mu}{\eta}$ can be disintegrated as $\pi = \pi_1 \otimes \delta_{(1,4)} \otimes \pi_3$ for some $\pi_1, \pi_3 \in \couplings{P}{ P}$, where $P \coloneqq \frac{1}{2}(\delta_0 + \delta_1) = \mu_1 = \eta_1 = \mu_3 = \eta_3$. One can easily check that any such coupling places at most a mass of $\frac{1}{2}$ on the set $\setdef{(x,y) \in \supp \mu \times \supp \eta}{x\approx y}$. Hence, provided we choose $\eps < \frac{1}{2}$, $$
 W_{\Gmarkov,p}(\mu,\eta) \geq 
 \brackets{\frac{1}{2\eps^p}}^{\frac{1}{p}} \geq \frac{1}{2\eps} > 2\eps \geq W_{\Gmarkov,p}(\mu,\nu) + W_{\Gmarkov,p}(\nu,\eta).$$
\end{example}

The next two examples show that graph causal Wasserstein distances are indeed generally different from Wasserstein or adapted Wasserstein distances, despite coinciding for certain classes of graphs and underlying costs, as discussed in Theorem \ref{AWisWG}. This was already shown in Example \ref{example:notequivalent} and these two examples complement the discussion in Section \ref{sec:topology}.

\begin{example}
\label{ex:AWDnotGW}
    Let $G = \brackets{\curlybrackets{1,2}, \emptyset}$ be the graph with two unconnected nodes, $a, b, c, d$ be distinct real numbers, and $\mu, \nu \in \Gprobabilities{\mbb{R}^2}$ be given by $\mu = \frac{1}{4}\brackets{\delta_{(a,a)}+\delta_{(a,b)}+\delta_{(b,a)}+\delta_{(b,b)}}$, $\nu = \frac{1}{4}\brackets{\delta_{(c,c)}+\delta_{(c,d)}+\delta_{(d,c)}+\delta_{(d,d)}}$. We consider a cost function as given by the following table, which can be extended to a metric on $\mbb{R}^2$:

\begin{center}
\begin{tabular}{|c|c|c|c|c|}
\hline
 & $(c,c)$ & $(c,d)$ & $(d,c)$ & $(d,d)$ \\
\hline
$(a,a)$ & 3 & 4 & 4 & 3 \\
\hline
$(a,b)$ & 4 & 3 & 3 & 4 \\
\hline
$(b,a)$ & 3 & 4 & 4 & 3 \\
\hline
$(b,b)$ & 4 & 3 & 3 & 4 \\
\hline
\end{tabular}
\end{center}
Since any $G$-bicausal coupling must take the form $\pi_1 \otimes \pi_2$ for $\pi_1, \pi_2 \in \couplings{\frac{1}{2}\delta_a + \frac{1}{2}\delta_b}{\frac{1}{2}\delta_c + \frac{1}{2}\delta_d}$ (which are fully parametrised by $\pi_i((a,c))$), one can check all of them result in the same cost, and so $W_{G, 1}(\mu, \nu) = 3.5$. 

However, the map given by $(a,a) \mapsto (c,c)$, $(a,b) \mapsto (c,d)$, $(b,a) \mapsto (d,d)$,$(b,b) \mapsto (d,c)$ is a biadapted Monge map between $\mu$ and $\nu$ which induces the optimal cost of $3$. Hence, $W_1(\mu, \nu) = AW_1(\mu, \nu) = 3 < 3.5 = W_{G,1}(\mu, \nu)$.
\end{example}

It was essential in the example above that the cost did not separate in the form $d_1(x_1,y_1) + d_2(x_2, y_2)$, as otherwise we would be in the setting of Theorem \ref{AWisWG}. If we want our cost to separate, we need to look at a larger graph. In these settings, it usually suffices to simply add a new node which is a child of every node in the original graph, as illustrated in the example below. 

\begin{example}
\label{ex:AWDnotGW2}
    Let $G = \brackets{\curlybrackets{1,2,3}, \curlybrackets{(1,3), (2,3)}}$, $a, b, c \in \mbb{R}^2$ represent the vertices of an equilateral triangle with side length 2, and $\mu, \nu \in \Gprobabilities{\mbb{R} \times \mbb{R} \times \mbb{R}^2}$ be given by $$\mu = \frac{1}{4}\brackets{\delta_{(0,0,a)}+\delta_{(0,1,b)}+\delta_{(1,0,c)}+\delta_{(1,1,c)}},$$ $$\nu = \frac{1}{4}\brackets{\delta_{(0,0,c)}+\delta_{(0,1,c)}+\delta_{(1,0,b)}+\delta_{(1,1,a)}}.$$ We take the metric $d(x,y) = \abs{x_1 - y_1} + \abs{x_2 - y_2} + \sqnorm{x_3 - y_3}$ as our underlying cost, and note that it separates. Our cost table then looks like this:

    \begin{center}
\begin{tabular}{|c|c|c|c|c|}
\hline
 & $(0,0,c)$ & $(0,1,c)$ & $(1,0,b)$ & $(1,1,a)$ \\
\hline
$(0,0,a)$ & 2 & 3 & 3 & 2 \\
\hline
$(0,1,b)$ & 3 & 2 & 2 & 3 \\
\hline
$(1,0,c)$ & 1 & 2 & 2 & 3 \\
\hline
$(1,1,c)$ & 2 & 1 & 3 & 2 \\
\hline
\end{tabular}
\end{center}

By the characterisation of bicausality in terms of disintegrations from Proposition \ref{equivalent_characterisations}, one sees that $G$-bicausal couplings between $\mu$ and $\nu$ take the form $\pi_1 \otimes \pi_2 \otimes \pi_3^{x_{1:2}, y_{1:2}}$ for $\pi_1, \pi_2 \in \couplings{\frac{1}{2}\delta_0 + \frac{1}{2}\delta_1}{\frac{1}{2}\delta_0 + \frac{1}{2}\delta_1}$ and $\pi_3^{x_{1:2}, y_{1:2}}$ supported on a single point, consisting of the third coordinates of the vector beginning with $(x_1,x_2)$ in the support of $\mu$, respectively of the vector beginning with $(y_1,y_2)$ in the support of $\nu$.

Thus again we can describe this family with just two parameters, $\alpha_i \coloneqq \pi_i(0,0)$. Letting $A \coloneq \alpha_1 \alpha_2$, $B \coloneq \alpha_1 (\frac{1}{2} - \alpha_2)$, $C \coloneq (\frac{1}{2} - \alpha_1)\alpha_2$, $D \coloneq (\frac{1}{2} - \alpha_1) (\frac{1}{2} - \alpha_2)$, a general $G$-bicausal coupling takes the form:

    \begin{center}
\begin{tabular}{|c|c|c|c|c|}
\hline
 & $(0,0,c)$ & $(0,1,c)$ & $(1,0,b)$ & $(1,1,a)$ \\
\hline
$(0,0,a)$ & A & B & C & D\\
\hline
$(0,1,b)$ & B & A & D & C \\
\hline
$(1,0,c)$ & C & D & A & B \\
\hline
$(1,1,c)$ & D & C & B & A \\
\hline
\end{tabular}
\end{center}
All such couplings have a cost of at least $8(A+B+C+D) = 2$, and so $W_{G,1}(\mu, \nu) \geq 2$.

But the map given by $(0,0,a) \mapsto (1,1,a)$, $(0,1,b) \mapsto (1,0,b)$, $(1,0,c) \mapsto (0,0,c)$, $(1,1,c) \mapsto (0,1,c)$ is a biadapted Monge map from $\mu$ to $\nu$ achieving a cost of $1.5$. So, $W_1(\mu, \nu) \leq AW_1(\mu, \nu) \leq 1.5 < 2 \leq W_{G,1}(\mu,\nu)$.
\end{example}

\subsection{Proof of Theorem \ref{thm:lifted}}
\label{proof_lifted}

\subsubsection*{Step 1: $\Gbicouplings{\mu}{\nu} \subseteq \newGbicouplings{\mu}{\nu}$
}

We start by showing that $G$-bicausal couplings can be seen as projections of $G$-biadapted Monge couplings on a lifted space. The proof is inspired by the corresponding causal OT result in \cite[Theorem 3.5]{Beiglbock2022Biadapted}. The key difference from the original proof in \cite[Theorem 3.5]{Beiglbock2022Biadapted} is that we do the induction forward instead of backward along the causal directions, as this is more suitable for working with the graph structure. 

The following elementary lemma paves the way to a recursive characterisation of $G$-biadapted maps in Lemma \ref{recursive_Gbiadapted}. For a function $f : \mc{A} \times \mc{B} \to \mc{C}$, and any $a \in \mc{A}$, we  denote by $f^a$ the function $b \mapsto f(a,b)$.

\begin{lemma}
\label{simple_recursive_biadapted}
    A function $T : \mc{X}_1 \times \mc{X}_2 \to \mc{Y}_1 \times \mc{Y}_2$ given by $T(x_1,x_2) = (T_1(x_1), T_2(x_1, x_2))$ is biadapted if and only if 
    \begin{enumerate}
        \item $T_1 : \mc{X}_1 \to \mc{Y}_1$ is bijective;
        \item $T_2^{x_1} : \mc{X}_2 \to \mc{Y}_2$ is bijective for all $x_1 \in \mc{X}_1$.
    \end{enumerate}
    In this case, the inverse of $T$ is given by $S(y_1, y_2) = (T_1^{-1}(y_1), \brackets{T_2^{T_1^{-1}(y_1)}}^{-1}(y_2))$ i.e. $S_1 = T_1^{-1}$ and $S_2^{y_1} = \brackets{T_2^{x_1}}^{-1}$ for $x_1 = T_1^{-1}(y_1)$.
\end{lemma}
\begin{proof}
    First, suppose $T$ is biadapted. Let $S$ be its inverse. Then for all $(x_1, x_2) \in \mc{X}_1 \times \mc{X}_2$
    \begin{equation}
    \label{simple_lemma_crux}
        S(T(x_1, x_2)) = S(T_1(x_1), T_2^{x_1}(x_2)) = \brackets{S_1(T_1(x_1)), S_2^{T_1(x_1)}(T_2^{x_1}(x_2))}.
    \end{equation}
    Since, $S \circ T = \Id{\mc{X}_1 \times \mc{X}_2}$, this implies that $S_1 \circ T_1 = \Id{\mc{X}_1}$ and $S_2^{T_1(x_1)} \circ T_2^{x_1} = \Id{\mc{X}_2}$ for all $x_1 \in \mc{X}_1$. Symmetrically, from $T \circ S = \Id{\mc{Y}_1 \times \mc{Y}_2}$, we get that $T_1 \circ S_1 = \Id{\mc{Y}_1}$ and $T_2^{S_1(y_1)} \circ S_2^{y_1} = \Id{\mc{Y}_2}$ for all $y_1 \in \mc{Y}_1$. Plugging in $y_1 = T_1(x_1)$, the conclusion follows.

    For the converse, let $S_1$ be the inverse of $T_1$ and, for all $y_1 \in \mc{Y}_1$, $S_2^{y_1}$ be the inverse of $T_2^{S_1(y_1)}$. We define the map $S$ via $(y_1, y_2) \mapsto (S_1(y_1), S_2^{y_1}(y_2))$. It now suffices to show that $S \circ T = \Id{\mc{X}_1 \times \mc{X}_2}$ and $T \circ S = \Id{\mc{Y}_1 \times \mc{Y}_2}$. But these follow easily from identity (\ref{simple_lemma_crux}) and its reverse analogue.
\end{proof}

\begin{lemma}
    \label{recursive_Gbiadapted}
    Suppose $n > 1$ and that $G$ is transitively closed. A function $T : \mc{X} \to \mc{Y}$ given by $T(x) = (T_{1:n-1}(x_{1:n-1}), T_n(x_n, x_{\parents{n}}))$ is $G$-biadapted if and only if 
    \begin{enumerate}
        \item $T_{1:n-1} : \mc{X}_{1:n-1} \to \mc{Y}_{1:n-1}$ is $G_{1:n-1}$-biadapted;
        \item $T_n^{x_{\parents{n}}}: \mc{X}_n \to \mc{Y}_n$ is bijective for all $x_{\parents{n}} \in \mc{X}_{\parents{n}}$.
    \end{enumerate}
    In this case, the inverse of $T$ is given by $S(y) = (S_{1:n-1}(y_{1:n-1}), S_n^{y_{1:n-1}}(y_n))$, where $S_{1:n-1} = T_{1:n-1}^{-1}$ and $S_n^{y_{1:n-1}} = \brackets{T^{x_{\parents{n}}}_n}^{-1}$ for $x_{1:n-1} = S_{1:n-1}(y_{1:n-1})$. 
\end{lemma}
\begin{proof}
The proof builds on Lemma \ref{simple_recursive_biadapted} by viewing $T$ as a function $T: \brackets{\mc{X}_{1:n-1}} \times \mc{X}_n \to \brackets{\mc{Y}_{1:n-1}} \times \mc{Y}_n$ depending on two parameters, ($x_{1:n-1}, x_n$). In fact the final statement follows directly from the final statement of Lemma \ref{simple_recursive_biadapted}.

Suppose $T$ is $G$-biadapted. Then, we know its inverse, $S$ can be written as $S(y) = \brackets{S_i(y_i, y_{\parents{i}})}_{1 \leq i \leq n}$. By Lemma \ref{simple_recursive_biadapted}, we get that $T_{1:n-1}$ is bijective, with inverse $S_{1:n-1}$. Both $T_{1:n-1}$ and $S_{1:n-1}$ are $G_{1:n-1}$-adapted, so $T_{1:n-1}$ is $G_{1:n-1}$-biadapted. Lemma \ref{simple_recursive_biadapted} also says that $T_n^{x_{1:n-1}} = T_n^{\xparents{n}}$ is bijective for all $x_{1:n-1} \in \mc{X}_{1:n-1}$.

For the converse, as $T$ is clearly $G$-adapted, it suffices to check that its inverse $S$ is too. By Lemma \ref{simple_recursive_biadapted}, we know $S_{1:n-1} = T_{1:n-1}^{-1}$. So, as $T_{1:n-1}$ is $G_{1:n-1}$-biadapted, $S_{1:n-1}$ is $G_{1:n-1}$-adapted. Moreover, we know that $S_n^{y_{1:n-1}} = \brackets{T_n^{x_{1:n-1}}}^{-1}$ for $x_{1:n-1} = S_{1:n-1}(y_{1:n-1})$. But $T_n^{x_{1:n-1}}$ only depends on $x_{1:n-1}$ through $\xparents{n} = \brackets{S_j(y_j, y_{\parents{j}})}_{j \in \parents{n}}$. Hence, $S_n^{y_{1:n-1}}$ only depends on $y_{1:n-1}$ through $y_I$, where $I = \parents{n} \cup \brackets{\bigcup_{j \in \parents{n}} \parents{j}}$. As $G$ is transitively closed, we see that in fact $I = \parents{n}$, so indeed we can write $S_n^{y_{1:n-1}} = S_n^{\parents{n}}$, ensuring $S$ is $G$-adapted.
\end{proof}

\begin{proof}[Proof that $\Gbicouplings{\mu}{\nu} \subseteq \newGbicouplings{\mu}{\nu}$
]
    We proceed by induction on the number of nodes $n$ of our graph $G$. If $n=1$, then a map is $G$-biadapted if and only if it is bijective, and the result follows via \cite[Theorem 2.1]{Beiglbock2022Biadapted}.

    Now assume that $n > 1$ and we have shown that the result holds for all transitively closed sorted DAGs with at most $n-1$ nodes. Then, by the induction hypothesis, as $\pi_{1:n-1} \in \specificGbicouplings{G_{1:n-1}}{\mu_{1:n-1}}{\nu_{1:n-1}}$, there exists a $G_{1:n-1}$-biadapted map $\Tilde{T} : \Hat{\mc{X}}_{1:n-1} \to \Hat{\mc{Y}}_{1:n-1}$ such that $\pushforward{\Tilde{T}}{\Hat{\mu}_{1:n-1}} = \Hat{\nu}_{1:n-1}$ and $\pushforward{\proj{\mc{X}_{1:n-1} \times \mc{Y}_{n-1}}}{\pushforward{(\Id{}, \Tilde{T})}{\Hat{\mu}_{1:n-1}}} = \pi_{1:n-1}$.

    Using \cite[Theorem 2.3]{Beiglbock2022Biadapted}, there exists a measurable map $$S: \brackets{\mc{X}_{\parents{n}} \times \mc{Y}_{\parents{n}}} \times \Hat{\mc{X}}_n \to \Hat{\mc{Y}}_n$$
    such that for all $(x_{\parents{n}}, y_{\parents{n}}) \in \mc{X}_{\parents{n}} \times \mc{Y}_{\parents{n}}$, the map $S^{\xparents{n}, \yparents{n}} : \Hat{\mc{X}}_n \to \Hat{\mc{Y}}_n$ is a Borel isomorphism which satisfies $\pushforward{S^{\xparents{n}, \yparents{n}}}{\widehat{\mu_n^{\xparents{n}}}} = \widehat{\nu_n^{\yparents{n}}}$ and $\pushforward{\proj{\mc{X}_{n} \times \mc{Y}_{n}}}{\pushforward{(\Id{}, S^{\xparents{n}, \yparents{n}})}{\widehat{\mu_n^{\xparents{n}}}}} = \pi_{n}^{\xparents{n}, \yparents{n}}$.

    We now define $T:\Hat{\mc{X}} \to \Hat{\mc{Y}}$ via
    $$T(x_{1:n}, u_{1:n}) = \brackets{\Tilde{T}(x_{1:n-1}, u_{1:n-1}), S^{\xparents{n}, \proj{\mc{Y}_{\parents{n}}}\brackets{\Tilde{T}(x_{1:n-1}, u_{1:n-1})}}(x_n, u_n)}.$$
    By Lemma \ref{recursive_Gbiadapted}, $T$ is $G$-biadapted. It remains to check other desired properties:
     \begin{enumerate}
        \item[(P1)] $\pushforward{T}{\Hat{\mu}} = \Hat{\nu}$, or equivalently, $\Hat{\pi} \coloneqq \pushforward{(\Id{}, T)}{\Hat{\mu}} \in \Gbicouplings{\Hat{\mu}}{\Hat{\nu}}$, and 
        \item[(P2)] $\pushforward{\proj{\mc{X} \times \mc{Y}}}{\Hat{\pi}} = \pi.$
    \end{enumerate}

    For property (P1), let $f : \Hat{\mc{Y}} \to \mbb{R}$ be an arbitrary bounded measurable function. Then
    \begin{align*}
        &\int f(y_{1:n}, 
        v_{1:n}) \, d\pushforward{T}{\Hat{\mu}}(y_{1:n}, v_{1:n}) = \\
        & = \int f\brackets{\Tilde{T}(x_{1:n-1}, u_{1:n-1}), S^{\xparents{n}, \proj{\mc{Y}_{\parents{n}}}\brackets{\Tilde{T}(x_{1:n-1}, u_{1:n-1})}}(x_n, u_n)} \, d\Hat{\mu}(x_{1:n}, u_{1:n}) \\
        & = \iint f\brackets{\Tilde{T}(x_{1:n-1}, u_{1:n-1}), S^{\xparents{n}, \proj{\mc{Y}_{\parents{n}}}\brackets{\Tilde{T}(x_{1:n-1}, u_{1:n-1})}}(x_n, u_n)} \, d\widehat{\mu^{\xparents{n}}_n}(x_n, u_n) d\Hat{\mu}_{1:n-1} \\ 
        & = \iint f\brackets{\Tilde{T}(x_{1:n-1}, u_{1:n-1}), y_n, v_n} \, d\pushforward{S^{\xparents{n}, \proj{\mc{Y}_{\parents{n}}}\brackets{\Tilde{T}(x_{1:n-1}, u_{1:n-1})}}}{\widehat{\mu^{\xparents{n}}_n}}(y_n, v_n) d\Hat{\mu}_{1:n-1} \\
        & = \iint f\brackets{\Tilde{T}(x_{1:n-1}, u_{1:n-1}), y_n, v_n} \, d\widehat{\nu}^{\proj{\mc{Y}_{\parents{n}}}\brackets{\Tilde{T}(x_{1:n-1}, u_{1:n-1})}}_n(y_n, v_n) d\Hat{\mu}_{1:n-1}(x_{1:n-1}, u_{1:n-1}) \\
        &= \iint f\brackets{y_{1:n}, v_{1:n}} \, d\widehat{\nu^{\yparents{n}}_n}(y_n, v_n) d\pushforward{\Tilde{T}}{\Hat{\mu}_{1:n-1}}(y_{1:n-1}, v_{1:n-1}) \\
        &= \iint f\brackets{y_{1:n}, v_{1:n}} \, d\widehat{\nu^{\yparents{n}}_n}(y_n, v_n) d\Hat{\nu}_{1:n-1}(y_{1:n-1}, v_{1:n-1}) \\
        &= \int f(y_{1:n}, 
        v_{1:n}) \, d\Hat{\nu}(y_{1:n}, v_{1:n}),
    \end{align*}
    so $\pushforward{T}{\Hat{\mu}} = \Hat{\nu}$.

    For property (P2), let $f : \mc{X} \times \mc{Y} \to \mbb{R}$ be an arbitrary bounded measurable function. Then
    \begin{align*}
        \int f(&x_{1:n}, y_{1:n}) \,d\pushforward{\proj{\mc{X} \times \mc{Y}}}{\Hat{\pi}}(x_{1:n}, y_{1:n}) = \\
        = \int f(&x_{1:n}, y_{1:n}) \,d\pushforward{\proj{\mc{X} \times \mc{Y}}}{\pushforward{(\Id{}, T)}{\Hat{\mu}}}(x_{1:n}, y_{1:n})\\
        = \iint f(&\proj{\mc{X} \times \mc{Y}} \circ(\Id{}, T)(x_{1:n}, u_{1:n})) \,d\widehat{\mu_n^{\xparents{n}}}(x_n, u_n)d\Hat{\mu}_{1:n-1}(x_{1:n-1}, u_{1:n-1}) \\
        = \iint f(&\proj{\mc{X}_{1:n-1} \times \mc{Y}_{1:n-1}} \circ(\Id{}, \Tilde{T})(x_{1:n-1}, u_{1:n-1}), \\
        &\proj{\mc{X}_{n} \times \mc{Y}_{n}} \circ S^{\xparents{n}, \proj{\mc{Y}_{\parents{n}}}\brackets{\Tilde{T}(x_{1:n-1}, u_{1:n-1})}}(x_n,u_n)) \,d\widehat{\mu_n^{\xparents{n}}}d\Hat{\mu}_{1:n-1} \\
        = \iint f(&\proj{\mc{X}_{1:n-1} \times \mc{Y}_{1:n-1}} \circ(\Id{}, \Tilde{T})(x_{1:n-1}, u_{1:n-1}), \\
        &x_n, y_n) \,d\pi_n^{\xparents{n}, \proj{\mc{Y}_{\parents{n}}}\brackets{\Tilde{T}(x_{1:n-1}, u_{1:n-1})}}(x_n, y_n)d\Hat{\mu}_{1:n-1} \\
        = \iint f(&\proj{\mc{X}_{1:n-1} \times \mc{Y}_{1:n-1}} \circ(\Id{}, \Tilde{T})(x_{1:n-1}, u_{1:n-1}), \\
        &x_n, y_n) \,d\pi_n^{\proj{\mc{X}_{\parents{n}} \times \mc{Y}_{\parents{n}}} \brackets{\proj{\mc{X}_{1:n-1} \times \mc{Y}_{1:n-1}} \circ(\Id{}, \Tilde{T})(x_{1:n-1}, u_{1:n-1})}}(x_n, y_n)d\Hat{\mu}_{1:n-1} \\
        = \iint f(&x_{1:n}, 
        y_{1:n}) \,d\pi_n^{\xparents{n}, \yparents{n}}(x_n, y_n)d\pushforward{\proj{\mc{X}_{1:n-1} \times \mc{Y}_{n-1}}}{\pushforward{(\Id{}, \Tilde{T})}{\Hat{\mu}_{1:n-1}}}(x_{1:n-1}, y_{1:n-1}) \\
        = \iint f(&x_{1:n}, 
        y_{1:n}) \,d\pi_n^{\xparents{n}, \yparents{n}}(x_n, y_n)d\pi_{1:n-1}(x_{1:n-1}, y_{1:n-1}) \\
        = \int f(&x_{1:n}, y_{1:n}) \,d\pi(x_{1:n}, y_{1:n}), 
    \end{align*} 
    so (P2) holds.
\end{proof}

The fact that the inclusion $\Gbicouplings{\mu}{\nu} \subseteq \newGbicouplings{\mu}{\nu}$ may be strict is discussed in detail in Section \ref{proof:lifted_nosplits}.

\subsubsection*{Step 2: $\newGbicouplings{\mu}{\nu}$ as convex combinations of elements in $\Gbicouplings{\mu}{\nu}$}

\begin{proof}
We now prove a partial converse to first inclusion, showing that lifted $G$-bicausal couplings can be written as convex combinations of $G$-bicausal couplings. 

    Let $\pi \in \newGbicouplings{\mu}{\nu}$. Then, by definition, there exists $\Hat{\pi} \in \Gmonge{\Hat{\mu}}{\Hat{\nu}}$ such that $\pi = \pushforward{\proj{\mc{X}\times\mc{Y}}}{\Hat{\pi}}$. Let $(X,U,Y,V) \sim \Hat{\pi}$ and define for brevity $Z_i \coloneq (X_i, Y_i)$, $W_i \coloneq (U_i, V_i)$. Since $\Hat{\pi} \in \Gbicouplings{\Hat{\mu}}{\Hat{\nu}}$, then by Proposition \ref{equivalent_characterisations}, $\condindep{(Z_i, W_i)}{Z_{\parents{i}}, W_{\parents{i}}}{(Z_{1:i-1}, W_{1:i-1})}$ and so $\condindep{Z_i}{Z_{\parents{i}}, W_{\parents{i}}}{(Z_{1:i-1}, W_{1:i-1})}$ and $\condindep{W_i}{Z_i, Z_{\parents{i}}, W_{\parents{i}}}{(Z_{1:i}, W_{1:i-1})}$ for all $i\in V$.  In other words, the law of $(Z_1, W_1, \ldots, Z_n, W_n)$ is compatible to the graph $\Hat{G}$ with vertex set $\curlybrackets{z_1, w_1, \ldots z_n, w_n}$ in which node $z_i$ has parents $\curlybrackets{z_{\parents{i}}, w_{\parents{i}}}$ and node $w_i$ has parents $\curlybrackets{z_i, z_{\parents{i}}, w_{\parents{i}}}$.

    Thus, by Proposition \ref{prop:Gcomp_characterisation} applied to $\Hat{G}$, there exist (by possibly enlarging the sample space) independent real-valued random variables $P_1, \ldots, P_n, R_1, \ldots, R_n$ and measurable functions $f_i : (\hat{\mc{X}}_{\parents{i}} \times \hat{\mc{Y}}_{\parents{i}}) \times \mbb{R} \to \mc{X}_{i} \times \mc{Y}_i$, $g_i : (\mc{X}_{i} \times \mc{Y}_i) \times (\hat{\mc{X}}_{\parents{i}} \times \hat{\mc{Y}}_{\parents{i}}) \times \mbb{R} \to [0,1]^2$  such that for all $i\in V$
    \begin{align*}
        \label{convexity_proof_eq}
        Z_i = f_i(Z_{\parents{i}}, W_{\parents{i}}, P_i), \\
    W_i = g_i(Z_i, Z_{\parents{i}}, W_{\parents{i}}, R_i).
    \end{align*}

    This implies 
    \begin{equation}
        \label{convex_proof_eq}
    \condindep{Z_i}{Z_{\parents{i}}, W_{\parents{i}}}{(R_{1:n}, Z_{1:i-1}, W_{1:i-1})},
    \end{equation} and hence $\condindep{Z_i}{R_{1:n}, Z_{\parents{i}}, W_{\parents{i}}}{Z_{1:i-1}}$. But, using transitive closure of $G$ and the second equation above, we see $W_{\parents{i}}$ can be written as a function of $Z_{\parents{i}}, R_{\parents{i}}$ and thus $\sigma\brackets{R_{1:n}, Z_{\parents{i}}, W_{\parents{i}}} = \sigma\brackets{R_{1:n}, Z_{\parents{i}}}$. So, $\condindep{Z_i}{R_{1:n}, Z_{\parents{i}}}{Z_{1:i-1}}$ i.e. $\condindep{(X_i, Y_i)}{R_{1:n}, X_{\parents{i}}, Y_{\parents{i}}}{(X_{1:i-1}, Y_{1:i-1})}$, for all $i \in V$.

    But, as $\Hat{\pi} \in \Gbicouplings{\Hat{\mu}}{\Hat{\nu}}$, we also have that $\condindep{(X_i, U_i)}{X_{\parents{i}}, U_{\parents{i}}}{(Y_{\parents{i}}, V_{\parents{i}})}$, so $\condindep{X_i}{X_{\parents{i}}, U_{\parents{i}}}{(Z_{\parents{i}}, W_{\parents{i}})}$. Moreover, $\condindep{X_i}{X_{\parents{i}}}{U_{\parents{i}}}$ (by independence of $X$ and $U$) and $\condindep{X_i}{Z_{\parents{i}}, W_{\parents{i}}}{(R_{1:n}, Z_{1:i-1}, W_{1:i-1})}$ (from (\ref{convex_proof_eq})), so the chain rule of conditional expectation implies
    \begin{equation}
        \label{convex_proof_eq_2}
    \condindep{X_i}{X_{\parents{i}}}{(R_{1:n}, Z_{1:i-1}, W_{1:i-1})}.
    \end{equation}

    In particular, $\condindep{X_i}{R_{1:n}, X_{\parents{i}}}{Y_{\parents{i}}}$. Symmetrically, we get $\condindep{Y_i}{R_{1:n}, Y_{\parents{i}}}{X_{\parents{i}}}$, which together with the previously obtained $\condindep{(X_i, Y_i)}{R_{1:n}, X_{\parents{i}}, Y_{\parents{i}}}{(X_{1:i-1}, Y_{1:i-1})}$ show that $\condlaw{X, Y}{R_{1:n}}$ is $G$-bicausal by Proposition \ref{equivalent_characterisations}.

    All that remains to be done is to determine the marginals. By (\ref{convex_proof_eq_2}), $\condindep{X_i}{X_{\parents{i}}}{R_{1:n}, X_{1:i-1}}$, and so $\condlaw{X_i}{R_{1:n}, X_{1:i-1}} = \condlaw{X_i}{X_{\parents{i}}} = \mu_i^{X_{\parents{i}}}$. Thus, $\condlaw{X}{R_{1:n}} = \mu$. Symmetrically, we see  $\condlaw{Y}{R_{1:n}} = \nu$.

    Thus, $\condlaw{X, Y}{R_{1:n}} \in \Gbicouplings{\mu}{\nu}$ and so
    $$\pi = \law{X,Y} = \int \condlaw{X, Y}{R_{1:n}} \, d\law{R_{1:n}} \in \closedconvex{\Gbicouplings{\mu}{\nu}}.$$
    \end{proof}

    Note that the result above does not imply that $\newGbicouplings{\mu}{\nu}$ is convex. Indeed, if $G$ is the graph on two independent nodes, then (see Theorem \ref{thm:lifted_nosplits} and Proposition \ref{equivalent_characterisations}) $$\newGbicouplings{\mu_1 \otimes \mu_2}{\nu_1 \otimes \nu_2} = \Gbicouplings{\mu_1 \otimes \mu_2}{{\nu_1 \otimes \nu_2}} = \setdef{\pi_1 \otimes \pi_2}{\pi_i \in \couplings{\mu_i}{\nu_i}},$$ which is not convex in general, and we provide an explicit counterexample in Example \ref{must_split}. As a consequence, the inclusion $\newGbicouplings{\mu}{\nu} \subseteq \closedconvex{\Gbicouplings{\mu}{\nu}}$ may be strict.

\subsubsection*{Step 3: $\newGbicouplings{\mu}{\nu} = \tildeGbicouplings{\mu}{\nu}$
}
From the definition, it is clear that $\newGbicouplings{\mu}{\nu} \subseteq \tildeGbicouplings{\mu}{\nu}$ in general. The converse is a relatively straightforward consequence of the result in Step 1. 

Let $\pi \in \tildeGbicouplings{\mu}{\nu}$. Then, by definition, there exists $\Tilde{\pi} \in \Gbicouplings{\Hat{\mu}}{\Hat{\nu}}$ such that $\pi = \pushforward{\proj{\mc{X} \times \mc{Y}}}{\Tilde{\pi}}$. By Step 1, we know that $\Gbicouplings{\Hat{\mu}}{\Hat{\nu}} \subseteq \newGbicouplings{\Hat{\mu}}{\Hat{\nu}}$, so there exists $\Hat{\pi} \in \Gmonge{\doublehat{\mu}}{\doublehat{\nu}}$ such that $\Tilde{\pi} = \pushforward{\proj{\Hat{\mc{X}} \times \Hat{\mc{Y}}}}{\Hat{\pi}}$. Composing the projections, we see $\pi = \pushforward{\proj{\mc{X} \times \mc{Y}}}{\Hat{\pi}}$. In other words, if $(X, U, \mathfrak{U}, Y, V, \mathfrak{V}) \sim \Hat{\pi}$, then $(X,Y) \sim \pi$ and moreover the $U_1, \ldots, U_n, \mathfrak{U}_1, \ldots \mathfrak{U}_n$ are independent $\mathrm{Unif}[0,1]$, independent of $X$, the $V_1, \ldots, V_n, \mathfrak{V}_1, \ldots \mathfrak{V}_n$ are independent $\mathrm{Unif}[0,1]$, independent of $Y$, and there exists a $G$-biadapted $T : \doublehat{\mc{X}} \to \doublehat{\mc{Y}}$ such that $(Y, V, \mathfrak{V}) = T(X, U, \mathfrak{U})$. 

Applying suitable variant of the isomorphism theorem for measure spaces, such as 
\cite[Theorem 17.41]{Kechris1995DescriptiveSetTheory}, gives us a measurable bijective function  $f : [0,1]^2 \to [0,1]$ 
such that $\pushforward{f}{\lambda^2} = \lambda$, and we define $\Tilde{U}_i \coloneq f(U_i, \mathfrak{U}_i)$ and $\Tilde{V}_i \coloneq f(V_i, \mathfrak{V}_i)$. 
Moreover, define $\Phi : \doublehat{\mc{X}} \to \Hat{\mc{X}}$ and $\Psi : \doublehat{\mc{Y}} \to \Hat{\mc{Y}}$ via $\Phi(x, u, \mathfrak{u}) = (x_i, f(u_i, \mathfrak{u}_i))_{1 \leq i \leq n}$ and $\Psi(y, v, \mathfrak{v}) = (y_i, f(v_i, \mathfrak{v}_i))_{1 \leq i \leq n}$. Then $\Phi$ and $\Psi$ are $G$-biadapted with $\pushforward{\Phi}{\doublehat{\mu}} = \Hat{\mu}$ and $\pushforward{\Psi}{\doublehat{\nu}} = \Hat{\nu}$. Defining $\Tilde{T} \coloneq \Psi \circ T \circ \Phi^{-1}: \Hat{\mc{X}} \to \Hat{\mc{Y}}$, $\Tilde{T}$ is $G$-biadapted by Lemma \ref{closed_composition}. It is now trivial to check that $(Y, \Tilde{V}) = \Tilde{T}(X,\Tilde{U})$, with $(X, \Tilde{U}) \sim \Hat{\mu}$ and $(Y, \Tilde{V}) \sim \Hat{\nu}$, and hence conclude that $\pi = \law{X,Y} \in \newGbicouplings{\mu}{\nu}$.

\subsection{Proof of Theorem \ref{main_result_metric}}
\label{proof_metric}

\begin{proof}[Proof of Theorem \ref{main_result_metric}]

Recall Example \ref{example_counter_Markov} which showed the triangle inequality may fail for $\Gmarkov$. The converse statement then follows by observing that if $G$ is not transitively closed, then it has $\Gmarkov$ as a proper subgraph. Hence, we can extend our explicit counterexample for $\Gmarkov$ to one for $G$ by setting all random variables corresponding to nodes which do not belong to the $\Gmarkov$ subgraph to be constant.

We move to the positive assertion. 
    By Corollary \ref{coro:GOT}, we can work with the representation in \eqref{alternative_Wasserstein}.
Non-negativity is clear since $d_{\mc{X}}\geq 0$ and $\newGbicouplings{\mu}{\nu}\neq \emptyset$ as the independent coupling is always $G$-bicausal, see Proposition \ref{equivalent_characterisations}, and thus also lifted $G$-bicausal by Theorem \ref{thm:lifted}. The symmetry follows by the symmetry of lifted $G$-bicausal couplings: the pushforward of a coupling in $\newGbicouplings{\mu}{\nu}$ under the map $(x,y) \mapsto (y,x)$ is clearly in $\newGbicouplings{\nu}{\mu}$. Thus, the only non-trivial property is the triangle inequality, which can be proved using a classical gluing argument as in \cite[Theorem 7.3]{Villani2003TopicsTransportation} by simply replacing the classical gluing lemma \cite[Lemma 7.6]{Villani2003TopicsTransportation} with its lifted $G$-bicausal version, Lemma \ref{bicausal_gluing_lemma} below.
\end{proof}

\begin{lemma}[Lifted $G$-bicausal gluing lemma]
    \label{bicausal_gluing_lemma}
    Let $\mu, \nu, \eta \in \Gpprobabilities{p}{\mc{X}}$, and $\pi_1 \in \newGbicouplings{\mu}{\nu}$, $\pi_2 \in \newGbicouplings{\nu}{\eta}$. Then there exists $\pi \in \probabilities{\mc{X}^3}$ such that its marginal on the first two coordinates is $\pi_1$, its marginal on the last two coordinates is $\pi_2$, and $\pi_3$, its marginal on the first and third coordinates, has the property that $\pi_3 \in \newGbicouplings{\mu}{\eta}$.
\end{lemma}
\begin{proof}
    By definition, there exist $G$-biadapted maps $T_1, T_2 : \Hat{\mc{X}} \to \Hat{\mc{X}}$ such that $\pushforward{T_1}{\Hat{\mu}} = \Hat{\nu}$ and $\pushforward{T_2}{\Hat{\nu}} = \Hat{\eta}$, with their induced Monge couplings $\hat{\pi}_1 \in \Gmonge{\Hat{\mu}}{\Hat{\nu}}$, $\hat{\pi}_2 \in \Gmonge{\Hat{\nu}}{\Hat{\eta}}$ satisfying $\pushforward{\proj{\mc{X} \times \mc{X}}}{\hat{\pi}_i} = \pi_i$. Define $\hat{\pi} \coloneqq \pushforward{(\mathrm{Id}, T_1, T_2 \circ T_1)}{\Hat{\mu}}$ and $\pi \coloneqq \pushforward{\proj{\mc{X} \times \mc{X} \times \mc{X}}}{\hat{\pi}}$. Then the marginal of $\pi$ on the first two coordinates is $\pi_1$, and the marginal of $\pi$ on the last two coordinates is $\pi_2$. The marginal of $\pi$ on the first and third coordinate is $\pi_3 \coloneqq \pushforward{\proj{\mc{X} \times \mc{X}}}{\hat{\pi}_3}$, where $\hat{\pi}_3 \coloneqq \pushforward{(\mathrm{Id}, T_2 \circ T_1)}{\Hat{\mu}}$. As $G$ is transitively closed, by Lemma \ref{closed_composition}, $T_2 \circ T_1$ is $G$-biadapted. Thus $\hat{\pi}_3 \in \Gmonge{\hat{\mu}}{\hat{\eta}}$ and $\pi_3 \in \newGbicouplings{\mu}{\eta}$.
\end{proof}

\subsection{Proof of Theorem \ref{theorem:sametopology}}
\label{proof:sametopology}

It is well known that convergence in the classical $p$-Wasserstein distance is equivalent to weak convergence plus convergence of $p$\textsuperscript{th} moments (see \cite[Theorem 7.12]{Villani2003TopicsTransportation} for a precise statement). Moreover, if the underlying metric is bounded, one can drop the $p$\textsuperscript{th} moment condition, and the Wasserstein topology is simply the topology of weak convergence. Thus, if we denote by $\singlebar{W}_p$ the $p$-Wasserstein distance induced by the metric $\singlebar{d}$ on $\mc{X}$ given by $\singlebar{d}(x,y) \coloneq d(x,y) \wedge 1$, then $\singlebar{W}_p$ metrizes weak convergence, and $W_p$ convergence is equivalent to $\singlebar{W}_p$ convergence plus convergence of $p$\textsuperscript{th} moments. 

We prove a similar result for $W_{G,p}$. For each $i \in V$, define $\singlebar{d_i}(x_i,y_i) \coloneq d_i(x_i, y_i) \wedge 1$, so that $\singlebar{d_i}$ is a bounded metric on $\mc{X}_i$ inducing the same topology as $d_i$.  Let $\doublebar{d}(x,y)^p \coloneq \sum_{i=1}^n \singlebar{d_i}(x_i, y_i)^p$. Then both $\singlebar{d}$ and $\doublebar{d}$ are bounded metrics on $\mc{X}$ (inducing the same topology as $d$) and we can thus define their induced $p$-Wasserstein distances via
$$\singlebar{W}_{G,p}(\mu, \nu)^p \coloneq \inf_{\pi \in \Gbicouplings{\mu}{\nu}} \int \singlebar{d}(x,y)^p \, d\pi(x,y),$$
$$\doublebar{W}_{G,p}(\mu, \nu)^p \coloneq \inf_{\pi \in \Gbicouplings{\mu}{\nu}} \int \doublebar{d}(x,y)^p \, d\pi(x,y),$$
for all $\mu, \nu \in \Gprobabilities{\mc{X}}$. By Theorem \ref{main_result_metric}, $\singlebar{W}_{G,p}$, $\doublebar{W}_{G,p}$  are metrics on $\Gprobabilities{\mc{X}}$, and hence also on its subspace $\Gpprobabilities{p}{\mc{X}}$ (which is still defined with respect to the original metric $d$). 

\begin{proposition} 
    \label{prop:bounded_metric}
    For a sequence $(\mu^{(m)})_{m \geq 1}$ in $\Gpprobabilities{p}{\mc{X}}$, and $\nu \in \Gpprobabilities{p}{\mc{X}}$, the following are equivalent:
    \begin{enumerate}
        \item $W_{G,p}(\mu^{(m)}, \nu) \to 0$ as $m \to \infty$;
        \item $\singlebar{W}_{G,p}(\mu^{(m)}, \nu) \to 0$ and $\int d(x, \tilde{x})^p \,d\mu^{(m)}(x) \to \int d(x, \tilde{x})^p \,d\nu(x),$ as $m \to \infty$ for some (and hence all) $\tilde{x} \in \mc{X}$;
        \item $\doublebar{W}_{G,p}(\mu^{(m)}, \nu) \to 0$ and $\int d(x, \tilde{x})^p \,d\mu^{(m)}(x) \to \int d(x, \tilde{x})^p \,d\nu(x),$ as $m \to \infty$ for some (and hence all) $\tilde{x} \in \mc{X}$.
    \end{enumerate}
\end{proposition}
\begin{proof}
    Note $\singlebar{d} \leq \doublebar{d} \leq d$. Hence, that 1) implies 3) implies 2) is trivial after noting that $W_{G,p}$ convergence is stronger than $W_p$ convergence, which implies the moment condition by \cite[Theorem 7.12]{Villani2003TopicsTransportation}. The proof that 2) implies 1) is essentially the same as that of the iii) implies i) direction of \cite[Theorem 7.12]{Villani2003TopicsTransportation}, with the minor difference of having to use $\eps$-optimal couplings instead of optimal couplings for $\singlebar{W}_{G,p}$ because of possible lack of attainment of the optimal cost.
\end{proof}

\begin{lemma}
    \label{lemma:WassersteinPQIneq}
    For $1 \leq p \leq q$, we have that $W_{G, p}(\mu, \nu) \leq W_{G, q}(\mu, \nu), \forall \mu, \nu \in \Gpprobabilities{q}{\mc{X}}$. If moreover the space $(\mc{X},d)$ is bounded, then $W_{G,p}(\mu,\nu) \leq \mathrm{diam}(\mc{X})^{\frac{p-1}{p}} W_{G,1}(\mu, \nu)^{\frac{1}{p}}, \forall \mu, \nu \in \Gprobabilities{\mc{X}}$, and thus all $W_{G,p}$ metrics induce the same topology on $\Gprobabilities{\mc{X}}$.
\end{lemma}
\begin{proof}
Follows analogously to the corresponding result for classical Wasserstein distances, see e.g. \cite[Section 5.1]{Santambrogio}, together with the equivalence for all $p \geq 1$ of the metrics $(\sum_{i=1}^n d_i(x_i, y_i)^p)^{\frac{1}{p}}$ (in analogy to equivalence of $l_p$ norms on $\mbb{R}^n$).
\end{proof}

We can now prove Theorem \ref{theorem:sametopology}.

\begin{proof}[Proof of Theorem \ref{theorem:sametopology}]
    We need to show that for a sequence $(\mu^{(m)})_{m \geq 1}$ in $\Gpprobabilities{p}{\mc{X}}$, and $\nu \in \Gpprobabilities{p}{\mc{X}}$, $W_{G,p}(\mu^{(m)}, \nu) \to 0$ as $m \to \infty$ if and only if $AW_p(\mu^{(m)}, \nu) \to 0$ as $m \to \infty$.
    
    We assume without loss of generality (via a possible relabelling) that $G$ is sorted. By Proposition \ref{prop:bounded_metric} applied to both $G$ and the linear graph $G_n = \setdef{(i,j)}{1\leq i < j \leq n}$, we see it is enough to show that $\doublebar{W}_{G,p}(\mu^{(m)}, \nu) \to 0$ as $m \to \infty$ if and only if $\doublebar{AW}_p(\mu^{(m)}, \nu) \to 0$ as $m \to \infty$, where we use the notation $\doublebar{AW}_p$ for $\doublebar{W}_{G_n, p}$ (since $AW_p = W_{G_n, p}$). Moreover, since $\doublebar{d}$ is bounded, by Lemma \ref{lemma:WassersteinPQIneq} it is enough to prove this for $p=1$.

    That $\doublebar{W}_{G,1}$ convergence implies $\doublebar{AW}_1$ convergence follows trivially since $G$-bicausality is a stronger condition than $G_n$-bicausality, and thus $\doublebar{AW}_1 \leq \doublebar{W}_{G,1}$. It remains to prove that if $\doublebar{AW}_1(\mu^{(m)}, \nu) \to 0$ then $\doublebar{W}_{G,1}(\mu^{(m)}, \nu) \to 0$ as $m \to \infty$.

    The proof proceeds by induction on $n$, the number of nodes of $G$. If $n=1$, $W_{G,1} = W_1 = AW_1$ and there is nothing to prove. For $n >1$, assume the result has been shown for all transitively closed DAGs with at most $n-1$ nodes and let $\tilde{G} \coloneq G_{1:n-1}$.

    Let $\eps > 0$. By Lusin's Theorem \cite[Box 1.6]{Santambrogio} applied to the function $f$ given by $y_{1:n-1} \in \mc{X}_{1:n-1} \mapsto \nu_n^{\yparents{n}} \in \probabilities{\mc{X}_n}$, there exists a compact set $K_{\eps} \subseteq \mc{X}_{1:n-1}$ with $\nu_{1:n-1}(\setcomplement{K_\eps}) \leq \eps$ such that $f|_{K_\eps}$ is uniformly continuous. Thus, there exists a $\delta > 0$ such that whenever $y_{1:n-1}, \tilde{y}_{1:n-1} \in K_{\eps}$ are such that $\sum_{i=1}^{n-1} \singlebar{d}_i(y_i, \tilde{y}_i) \leq \delta$ then $\singlebar{W}_1(\nu_n^{\yparents{n}}, \nu_n^{\tildeyparents{n}}) \leq \eps$ (where $\singlebar{W}_1$ here is the $1$-Wasserstein distance with respect to the metric $\singlebar{d}_n$ on $\mc{X}_n$).
    
    Note that node $n$ has no children by the sorted assumption, and thus $\mu^{(m)}_{1:n-1} \in \specificGprobabilities{\tilde{G}}{\mc{X}_{1:n-1}}$, $\nu_{1:n-1} \in \specificGprobabilities{\tilde{G}}{\mc{X}_{1:n-1}}$, and whenever $\pi \in \Gbicouplings{\mu^{(m)}}{\nu}$ then $\pi_{1:n-1} \in \specificGbicouplings{\tilde{G}}{\mu^{(m)}_{1:n-1}}{\nu_{1:n-1}}$. Consequently, $\doublebar{AW}_1(\mu^{(m)}_{1:n-1}, \nu_{1:n-1}) \leq \doublebar{AW}_1(\mu^{(m)}, \nu)$, and so $\doublebar{AW}_1(\mu^{(m)}_{1:n-1}, \nu_{1:n-1}) \to 0$ as $m \to \infty$. By the induction hypothesis, this implies $\doublebar{W}_{\tilde{G}, 1}(\mu^{(m)}_{1:n-1}, \nu_{1:n-1}) \to 0$ as $m \to \infty$. So, there exists a sequence $\tilde{\pi}^{(m)} \in \specificGbicouplings{\tilde{G}}{\mu^{(m)}_{1:n-1}}{\nu_{1:n-1}}$ and $M_1 \geq 1$ such that for all $m \geq M_1$
    \begin{equation}
        \label{topologyeq1}
    \int \sum_{i=1}^{n-1} \singlebar{d}_i(x_i, \tilde{y}_i) \, d\tilde{\pi}^{(m)}(x_{1:n-1},\tilde{y}_{1:n-1}) \leq \min(\eps, \eps\frac{\delta}{2}).
    \end{equation}

    Furthermore, as $\doublebar{AW}_1(\mu^{(m)}, \nu) \to 0$, using the DPP for adapted Wasserstein distances \cite[Proposition 5.2]{Backhoff2016CausalDiscrete}, there exists a sequence $\pi^{(m)} \in \bicouplings{\mu^{(m)}_{1:n-1}}{\nu_{1:n-1}}$ and $M_2 \geq 1$ such that for all $m \geq M_2$
    \begin{equation}
        \label{topologyeq2}
    \int \sum_{i=1}^{n-1} \singlebar{d}_i(x_i, y_i) + \singlebar{W}_1(\mu_n^{(m)\xparents{n}}, \nu_n^{\yparents{n}}) \, d{\pi}^{(m)}(x_{1:n-1},y_{1:n-1}) \leq \min(\eps, \eps\frac{\delta}{2}).
    \end{equation}

    Let $M \coloneq \max(M_1, M_2)$ and fix $m \geq M$. 

    By inequalities above, the set $A_\eps \coloneq \setdef{(x_{1:n-1}, y_{1:n-1}) \in \mc{X}_{1:n-1}^2}{\metricplaceholder{x}{y} \leq \frac{\delta}{2}}$ has $\pi^{(m)}(\setcomplement{A_\eps}) \leq \eps$ and $\tilde{\pi}^{(m)}(\setcomplement{A_\eps}) \leq \eps$.

    Let $\hat{\pi}^{(m)} \in \probabilities{\mc{X}_{1:n-1}^3}$ be the gluing of $\pi^{(m)}$ and $\tilde{\pi}^{(m)}$ along their common $\mu^{(m)}_{1:n-1}$ marginal, so that if $(Y_{1:n-1},X_{1:n-1},\tilde{Y}_{1:n-1}) \sim \hat{\pi}^{(m)}$ then $(X_{1:n-1},Y_{1:n-1}) \sim \pi^{(m)}$, $(X_{1:n-1},\tilde{Y}_{1:n-1}) \sim \tilde{\pi}^{(m)}$. Define the set
    $$B_\eps \coloneq \setdef{(y_{1:n-1}, x_{1:n-1}, \tilde{y}_{1:n-1}) \in K_{\eps} \times \mc{X}_{1:n-1} \times K_\eps}{\metricplaceholder{x}{y} \leq \frac{\delta}{2}, \metricplaceholder{x}{\tilde{y}} \leq \frac{\delta}{2}}.$$
    Then $(y_{1:n-1}, x_{1:n-1}, \tilde{y}_{1:n-1}) \in \mc{X}_{1:n-1}^3 \setminus B_{\eps}$ if and only if one of $y_{1:n-1} \in \setcomplement{K_\eps}$, $\tilde{y}_{1:n-1} \in \setcomplement{K_\eps}$, $(x_{1:n-1},y_{1:n-1}) \in \setcomplement{A_{\eps}}$, or $(x_{1:n-1},\tilde{y}_{1:n-1}) \in \setcomplement{A_{\eps}}$ holds. As $\pi^{(m)}(\setcomplement{A_\eps}) \leq \eps$, $\tilde{\pi}^{(m)}(\setcomplement{A_\eps}) \leq \eps$, and $\nu_{1:n-1}(\setcomplement{K_\eps}) \leq \eps$, we see $\hat{\pi}^{(m)}(\setcomplement{B_\eps}) \leq 4\eps$. Moreover, for $(y_{1:n-1}, x_{1:n-1}, \tilde{y}_{1:n-1}) \in B_\eps$, it holds that $\metricplaceholder{y}{\tilde{y}} \leq \delta$ and $y_{1:n-1}, \tilde{y}_{1:n-1} \in K_\eps$, so $\singlebar{W}_1(\nu_n^{\yparents{n}}, \nu_n^{\tildeyparents{n}}) \leq \eps$. Therefore,
    \begin{align*}
        &\abs{\int \singlebar{W}_1(\mu_n^{(m)\xparents{n}}, \nu_n^{\yparents{n}}) \, d{\pi}^{(m)} - \int \singlebar{W}_1(\mu_n^{(m)\xparents{n}}, \nu_n^{\tildeyparents{n}}) \, d\tilde{\pi}^{(m)}} \\
        & = \abs{\int \singlebar{W}_1(\mu_n^{(m)\xparents{n}}, \nu_n^{\yparents{n}}) - \singlebar{W}_1(\mu_n^{(m)\xparents{n}}, \nu_n^{\tildeyparents{n}}) \, d\hat{\pi}^{(m)}(y_{1:n-1}, x_{1:n-1}, \tilde{y}_{1:n-1})} \\
        &\leq \int \abs{\singlebar{W}_1(\mu_n^{(m)\xparents{n}}, \nu_n^{\yparents{n}}) - \singlebar{W}_1(\mu_n^{(m)\xparents{n}}, \nu_n^{\tildeyparents{n}})} \, d\hat{\pi}^{(m)}(y_{1:n-1}, x_{1:n-1}, \tilde{y}_{1:n-1}) \\
        &\leq \int \singlebar{W}_1(\nu_n^{\yparents{n}}, \nu_n^{\tildeyparents{n}}) \, d\hat{\pi}^{(m)} (y_{1:n-1}, x_{1:n-1}, \tilde{y}_{1:n-1}) \leq \hat\pi^{(m)}(\setcomplement{B_\eps}) + \eps \leq 5\eps.
    \end{align*}
    Combining this with (\ref{topologyeq1}) and (\ref{topologyeq2}), we see
    \begin{align*}
    \doublebar{W}_{G,1}&(\mu^{(m)}, \nu)\leq \int \sum_{i=1}^{n-1} \singlebar{d}_i(x_i, \tilde{y}_i) + \singlebar{W}_1(\mu_n^{(m)\xparents{n}}, \nu_n^{\tildeyparents{n}}) \, d{\tilde{\pi}}^{(m)}(x_{1:n-1},\tilde{y}_{1:n-1}) \\
    &\leq \int \sum_{i=1}^{n-1} \singlebar{d}_i(x_i, \tilde{y}_i) \, d\tilde{\pi}^{(m)}(x_{1:n-1},\tilde{y}_{1:n-1}) + \int \singlebar{W}_1(\mu_n^{(m)\xparents{n}}, \nu_n^{\yparents{n}}) \, d{\pi}^{(m)} + 5\eps \leq 7\eps,
    \end{align*}
    where the first inequality holds by considering the coupling which combines $\tilde{\pi}^{(m)}$ with measurably selected optimal couplings \cite[Corollary 5.22]{Villani2009OptimalNew} in $\couplings{\mu_n^{(m)\xparents{n}}}{\nu_n^{\tildeyparents{n}}}$. As $\eps >0$ was fixed arbitrarily, the result follows.
\end{proof}
\subsection{Proof of Theorem \ref{glue_characterisation}}
\label{proof_gluing}

We define a set of operations on graphs which we will use in the reminder of the text. 

\begin{definition}
\label{graph_operations}
    Let $G = (V, E)$ be a directed graph, and $A \subseteq V$. We write $G_A$ for the subgraph of $G$ given by $(A, E \cap (A\times A))$.
    
    Moreover, the graph $G^+$ is defined as the directed graph with vertex set $V \cup \curlybrackets{*}$, where $*$ is some vertex label that is not already in $V$, and edge set
    $$E \cup \setdef{(i, *)}{i \in V}.$$

    When $V = \curlybrackets{1, \ldots, n}$, we canonically pick the label $*$ to be $n+1$.

    Let $G_1 = (V_1, E_1)$, $G_2 = (V_2, E_2)$ be directed graphs. Through a relabelling, we may assume without loss of generality that $V_1$ and $V_2 $ are disjoint. We define their \emph{disjoint union graph} $G_1 \oplus G_2$ as the directed graph with vertex set $V_1 \cup V_2$ and edge set
    $E_1 \cup E_2$. 
\end{definition}

We begin with some examples of graphs without the causal gluing property. In our examples, we choose $\mc{X}_i = \mc{Y}_i = \mc{Z}_i = \mbb{R}$ for all $i = 1,2,3$, but they can easily be adapted to any choice of state spaces, as long as each of them consists of at least 2 distinct elements.

\begin{example}
\label{must_history}
    The graph $\Gmarkov \coloneqq (\curlybrackets{1,2,3}, \curlybrackets{(1,2), (2,3)})$ does not have the bicausal gluing property:

    Let $\mu = \eta = \frac{1}{4} \sum_{i=1}^4 \delta_{x^{(i)}}$, where $x^{(1)} = (0,0,0)$, $x^{(2)} = (0,0,1)$, $x^{(3)} = (1,0,0)$, $x^{(4)} = (1,0,1)$. Let $\nu = \frac{1}{4} \sum_{i=1}^4 \delta_{y^{(i)}}$, where $y^{(1)} = (0,1,0)$, $y^{(2)} = (0,1,1)$, $y^{(3)} = x^{(3)}$, $y^{(4)} = x^{(4)}$. Then $\mu, \nu, \eta$ are $\Gmarkov$-compatible.

    Let $X \sim \mu$ and define $Y = (X_1, 1 - X_1, X_3)$, $Z = (Y_1, 0, \mathbbm{1}_{Y_2 = Y_3})$. Moreover, $X = (Y_1, 0, Y_3)$, and $Y = (Z_1, 1 - Z_1, Z_3 \mathbbm{1}_{Y_2=1} + (1-Z_3) \mathbbm{1}_{Y_2=0})$. Then $Y \sim \nu$, $Z \sim \eta$, and clearly $\law{X, Y}$, $\law{Y,Z}$ are $\Gmarkov$-bicausal. However, as $X_2 = Z_2 = 0$, $X_1 \indep X_3$ and $Z_3 = X_3 \mathbbm{1}_{X_1=0} + (1 - X_3) \mathbbm{1}_{X_1=1}$, we do not have that $\condindep{Z_3}{X_3, X_2, Z_2}{(X, Z_{1:2})}$. Informally, going through $Y$ allows us to store information about $X_1$ in $Y_2$, which we then use in $Z_3$. Thus, $\law{X, Z} = \law{X, Y} \glue \law{Y,Z}$ is not $\Gmarkov$-bicausal.

\end{example}

\begin{example}
\label{must_split}
    The graph $\Gsplit \coloneqq (\curlybrackets{1,2,3}, \curlybrackets{(1,2), (1,3)})$ does not have the bicausal gluing property:

    Let $P \coloneqq \frac{1}{2} (\delta_0 + \delta_1)$. We define measures via their disintegrations $\mu = \eta \coloneqq \delta_0 \otimes P \otimes P$, $\nu \coloneqq P \otimes P \otimes P$. They are all $\Gsplit$-compatible.

    For $q \in [0,0.5]$, we define $\pi_q \coloneqq q (\delta_{(0,0)} + \delta_{(1,1)}) + (0.5 - q)(\delta_{(0,1)} + \delta_{(1,0)})$. Note that $\pi_{0.5} = \pushforward{(\textrm{Id},\textrm{Id})}{P}$. Clearly, $\couplings{P}{P} = \setdef{\pi_q}{q \in [0, 0.5]}$.  
    
    Fix $q_1 = q_2 = 0.25, q_3 = q_4 = 0.2$. By Proposition \ref{equivalent_characterisations}, the coupling $$\mbb{P} \coloneqq \frac{1}{2} \delta_{(0,0)} \otimes \pi_{q_1} \otimes \pi_{q_2} + \frac{1}{2} \delta_{(0,1)} \otimes \pi_{q_3} \otimes \pi_{q_4} \in \probabilities{\prod_{i=1}^3(\mc{X}_i \times \mc{Y}_i)},$$
    which via the appropriate reordering we view as a measure on $\probabilities{\mc{X} \times \mc{Y}}$, is in $\specificGbicouplings{\Gsplit}{\mu}{\nu}$. Similarly, 
    $$\mbb{Q} \coloneqq \frac{1}{2} \delta_{(0,0)} \otimes \pi_{0.5} \otimes \pi_{0.5} + \frac{1}{2} \delta_{(1,0)} \otimes \pi_{0.5} \otimes \pi_{0.5} \in \probabilities{\prod_{i=1}^3(\mc{Y}_i \times \mc{Z}_i)}$$
    is in $\specificGbicouplings{\Gsplit}{\nu}{\eta}$. However,
    $$\mbb{P} \glue \mbb{Q} = \delta_{(0,0)} \otimes \brackets{\frac{1}{2} \pi_{q_1} \otimes \pi_{q_2} + \frac{1}{2} \pi_{q_3} \otimes \pi_{q_4}} \in \probabilities{\prod_{i=1}^3(\mc{X}_i \times \mc{Z}_i)}$$
    is not in $\specificGbicouplings{\Gsplit}{\mu}{\eta}$, as all such $\Gsplit$-bicausal couplings must, by Proposition \ref{equivalent_characterisations}, be of the form
    $\delta_{(0,0)} \otimes \pi_{q_5} \otimes \pi_{q_6},$
    for some $q_5, q_6 \in [0,0.5]$, and one can check that there are no $q_5, q_6$ such that $\frac{1}{2} \pi_{q_1} \otimes \pi_{q_2} + \frac{1}{2} \pi_{q_3} \otimes \pi_{q_4} = \pi_{q_5} \otimes \pi_{q_6}$. Thus, what fails is the fact that the set $\setdef{\alpha \otimes \beta}{\alpha, \beta \in \couplings{P}{P}}$ is not convex.
\end{example}

The examples above can be extended via the following result to any $G$ which has $\Gmarkov$ or $\Gsplit$ as a substructure, i.e., which either is not transitively closed or contains a split.

\begin{proposition}
\label{substructures}
    Suppose that for some non-empty subset $A \subseteq V$, the graph $G_A$ does not satisfy the bicausal gluing property. Then $G$ does not satisfy the bicausal gluing property. 
\end{proposition}
\begin{proof}
    This is a simple consequence of the fact that if $X = (X_v)_{v\in V}, Y = (Y_v)_{v\in V}$ are random variables such that $X_v$, $Y_v$ are constant for all $v \in V \setminus A$, then $\law{X, Y}$ is $G$-causal if and only if $\law{X_A, Y_A}$ is $G_A$-causal. This is easy to check from Definition \ref{G_causal_definition_2}.
\end{proof}

Hence, we now focus on transitively closed DAGs without splits.

\begin{lemma}
\label{make_graph_two_operations}
    Suppose a transitively closed DAG $G$ contains no splits. Then it can be constructed from elementary graphs with just one node using the operations $(\cdot)^{+}$ and $\oplus$ given in Definition \ref{graph_operations}.
\end{lemma}
\begin{proof}
    We proceed by induction on the number of nodes of $G$. If $G$ is itself a graph with one node then the statement is trivial.

    Suppose now the statement holds for all DAGs with at most $n \geq 1$ nodes, and assume $G$ has $n+1$ nodes. Then, as no cycles are allowed, $G$ must have a node $x$ with no children. 

    If $i, j \in V \setminus \curlybrackets{x}$ are such that $i$ is a parent of $j$, then we must have that either $i, j \in \parents{x}$ or $i, j \in \parents{x}^\mathsf{c}$. Indeed, if $j \in \parents{x}$ then by transitive closure $i \in \parents{x}$. Otherwise, if $j \in \parents{x}^\mathsf{c}$ and $i \in \parents{x}$, then $i, j, x$ form a split (c.f., Definition \ref{split_def}), contradicting our assumption.

    Hence, $G$ can be constructed as $(G_{\parents{x}}^+) \oplus G_{V \setminus (\parents{x} \cup \curlybrackets{x})}$, where both $G_{\parents{x}}$ and $G_{V \setminus (\parents{x} \cup \curlybrackets{x})}$ have at most $n$ nodes, are transitively closed and have no splits.
\end{proof}

Clearly, if $G$ is a graph with just one node, $G$ has the causal gluing property since $\Gcouplings{\mu}{\nu} = \couplings{\mu}{\nu}$ in that case. We now show that we can construct graphs with the causal gluing property from other graphs with this property via the two operations in Lemma \ref{make_graph_two_operations}.

\begin{proposition}
\label{plus_good}
    Suppose the DAG $G$ has the causal gluing property. Then so does $G^+$.
\end{proposition}
\begin{proof}
We assume without loss of generality, via a relabelling, that $G$ is a sorted DAG with $V = \curlybrackets{1, \ldots, n}$. Then, by our convention in Definition \ref{graph_operations}, $G^+$ has vertex set $V \cup \curlybrackets{n+1}$ and edge set $E \cup (V \times \curlybrackets{n+1})$.

Let $\mu^+ \in \specificGprobabilities{G^+}{\mc{X}_{1:n+1}}$, $\nu^+ \in \specificGprobabilities{G^+}{\mc{Y}_{1:n+1}}$, $\eta^+ \in \specificGprobabilities{G^+}{\mc{Z}_{1:n+1}}$, $\pi^+ \in \Gpluscouplings{\mu^+}{\nu^+}$, $\tilde{\pi}^+ \in \Gpluscouplings{\nu^+}{\eta^+}$. We define $\mu \coloneqq \pushforward{\proj{\mc{X}_{1:n}}}{\mu^+}$, $\pi \coloneqq \pushforward{\proj{\mc{X}_{1:n} \times \mc{Y}_{1:n}}}{\pi^+}$, with similar definitions for $\nu, \eta, \tilde{\pi}$. Note that $\pi \in \Gcouplings{\mu}{\nu}$, $\tilde{\pi} \in \Gcouplings{\nu}{\eta}$ by Proposition \ref{marginal_of_coupling}.

Let $(X, Y, Z) \sim \bigglue{\pi^+}{\tilde{\pi}^+}$. Then $(X,Z) \sim {\pi^+} \glue {\tilde{\pi}^+}$, and we want to show $\condindep{Z_i}{Z_{\parents{i}}, X_i, X_{\parents{i}}}{(Z_{1:i-1}, X)}$ for all $1 \leq i \leq n + 1$. The case $i = n+1$ is trivial since $\parents{n+1} = \curlybrackets{1, \ldots, n}$.

By Proposition \ref{marginal_of_gluing}, $(X_{1:n}, Y_{1:n}, Z_{1:n}) \sim \bigglue{\pi}{\tilde{\pi}}$. In particular, $(X_{1:n}, Z_{1:n}) \sim \pi \glue \tilde{\pi}$, which is in $\Gcouplings{\mu}{\eta}$ since $G$ has the causal gluing property. Thus, $\condindep{Z_i}{X_i, X_{\parents{i}}, Z_{\parents{i}}}{(Z_{1:i-1}, X_{1:n})}$ for $1 \leq i \leq n$. By the chain rule of conditional independence \cite[Theorem 8.12]{Kallenberg2021FoundationsThird} it is enough to show $\condindep{Z_i}{X_{1:n}, Z_{1:i-1}}{X_{n+1}}$ for all $1 \leq i \leq n$. Applying the chain rule again, this is equivalent to proving $\condindep{Z_{1:n}}{X_{1:n}}{X_{n+1}}$.

As $(X, Y) \sim \pi^+ \in \Gpluscouplings{\mu^+}{\nu^+}$, we know $\condindep{Y_i}{X_i, X_{\parents{i}}, Y_{\parents{i}}}{(Y_{1:i-1}, X)}$ for all $1 \leq i \leq n$. Hence, $\condindep{Y_i}{X_{1:n}, Y_{1:i-1}}{X_{n+1}}$ for all $1 \leq i \leq n$, which is equivalent to $\condindep{Y_{1:n}}{X_{1:n}}{X_{n+1}}$, where both statements follow by the chain rule. Similarly, we obtain $\condindep{Z_{1:n}}{Y_{1:n}}{Y_{n+1}}$.

By properties of $\otimesdot$, $\condindep{X}{Y}{Z}$. In particular, $\condindep{X}{Y}{Z_{1:n}}$. Combining this with $\condindep{Z_{1:n}}{Y_{1:n}}{Y_{n+1}}$, we see that $\condindep{(X, Y_{n+1})}{Y_{1:n}}{Z_{1:n}}$, so $\condindep{X_{n+1}}{X_{1:n}, Y_{1:n}}{Z_{1:n}}$, both  using the chain rule
. This, together with the fact that $\condindep{Y_{1:n}}{X_{1:n}}{X_{n+1}}$ gives us by one more application of the chain rule that $\condindep{X_{n+1}}{X_{1:n}}{(Y_{1:n}, Z_{1:n})}$, so in particular $\condindep{X_{n+1}}{X_{1:n}}{Z_{1:n}}$, which is what we needed to show.
\end{proof}

\begin{proposition}
\label{oplus_good}
    Suppose the DAGs $G_1, G_2$ have the causal gluing property. Then so does $G_1 \oplus G_2$.
\end{proposition}
\begin{proof}
    We may assume without loss of generality via a relabelling that $G_1$, $G_2$ are sorted and have vertex sets $A \coloneqq \curlybrackets{1, \ldots, n_1}$, and $B \coloneqq \curlybrackets{n_1 + 1, \ldots, n_1 + n_2}$ respectively. Let $E_i$ denote the edge set of graph $G_i$. Then, we may write $G_1 \oplus G_2 = (A \cup B, E_1 \cup E_2)$.
    
    Let $\mu \in \specificGprobabilities{G_1 \oplus G_2}{\mc{X}_{A\cup B}}$, $\nu \in \specificGprobabilities{G_1 \oplus G_2}{\mc{Y}_{A\cup B}}$, $\eta \in \specificGprobabilities{G_1 \oplus G_2}{\mc{Z}_{A\cup B}}$, $\pi \in \specificGcouplings{G_1 \oplus G_2}{\mu}{\nu}$, $\tilde{\pi} \in \specificGcouplings{G_1 \oplus G_2}{\nu}{\eta}$. We define $\mu_1 \coloneqq \pushforward{\proj{\mc{X}_{A}}}{\mu}$, $\mu_2 \coloneqq \pushforward{\proj{\mc{X}_{B}}}{\mu}$, $\pi_1 \coloneqq \pushforward{\proj{\mc{X}_{A} \times \mc{Y}_{A}}}{\pi}$, $ \pi_2 \coloneqq \pushforward{\proj{\mc{X}_{B} \times \mc{Y}_{B}}}{\pi}$, with corresponding definitions for $\nu_i, \eta_i, \tilde{\pi}_i$. Note that $\pi_i \in \specificGcouplings{G_i}{\mu_i}{\nu_i}$, $\tilde{\pi}_i \in \specificGcouplings{G_i}{\nu_i}{\eta_i}$ by Proposition \ref{marginal_of_coupling}.

    Let $(X,Y,Z) \sim \bigglue{\pi}{\tilde{\pi}}$. Then $(X,Z) \sim \pi \glue \tilde{\pi}$, and we want to show $\condindep{Z_i}{Z_{\parents{i}}, X_i, X_{\parents{i}}}{(Z_{1:i-1}, X)}$ for all $i \in A \cup B$.

    Note that by construction of $G_1 \oplus G_2$, and the fact that $\mu, \nu, \eta$ are $G_1 \oplus G_2$-compatible and $\pi, \tilde{\pi}$ are $G_1 \oplus G_2$-causal, we see that
    \begin{equation}
    \label{indep_equation}
        {(X_A, Y_A, Z_A)} \indep {(X_B, Y_B, Z_B)}.
    \end{equation}
    
    Let $i \in A$. By Proposition \ref{marginal_of_gluing}, $(X_A, Y_A, Z_A) \sim \bigglue{\pi_1}{\tilde{\pi}_1}$. As $G_1$ has the causal gluing property, this means $(X_A, Z_A) \sim \pi_1 \glue \tilde{\pi}_1 \in \specificGcouplings{G_1}{\mu_1}{\eta_1}$. Thus, we know $\condindep{Z_i}{Z_{\parents{i}}, X_i, X_{\parents{i}}}{(Z_{1:i-1}, X_A)}$. By (\ref{indep_equation}), $ {(X_A, Z_{1:i})} \indep {X_B}$, which in turn implies that also $\condindep{Z_i}{Z_{1:i-1}, X_A}{X_B}$. Combining these via the chain rule of conditional independence \cite[Theorem 8.12]{Kallenberg2021FoundationsThird} gives us that $\condindep{Z_i}{Z_{\parents{i}}, X_i, X_{\parents{i}}}{(Z_{1:i-1}, X_A, X_B)}$, which is what we wanted. Importantly, we used the fact that $\parents{i} \subseteq \curlybrackets{1, \ldots, i-1} \subseteq A$.

    Now fix $i \in B$. As above, by Proposition \ref{marginal_of_gluing} we obtain $\condindep{Z_i}{Z_{\parents{i}}, X_i, X_{\parents{i}}}{(Z_{n_1 + 1:i-1}, X_B)}$. By (\ref{indep_equation}), ${(X_A, Z_{A})} \indep {X_B, Z_{n_1 + 1:i}}$, which in particular implies that also $\condindep{Z_i}{Z_{n_1 + 1:i-1}, X_B}{(X_A, Z_{A})}$. Finally, applying the chain rule to these two independence relations, we obtain the desired result: $\condindep{Z_i}{Z_{\parents{i}}, X_i, X_{\parents{i}}}{(Z_A, Z_{n_1+1:i-1}, X_A, X_B)} = (Z_{1:i-1}, X)$. Again, the key to this application of the chain rule is the fact that $\parents{i} \subseteq \curlybrackets{n_1 + 1, \ldots, i-1} \subseteq B$.
\end{proof}

We now have all the tools to prove Theorem \ref{glue_characterisation}.

\begin{proof}[Proof of Theorem \ref{glue_characterisation}]
    As in Remark \ref{causal_implies_bicausal_gluing}, we already know that $1) \Rightarrow 2)$. Examples \ref{must_history}, \ref{must_split}, show, via Proposition \ref{substructures}, that $2) \Rightarrow 3)$. Finally, as the graph with one node clearly has the causal gluing property, Lemma \ref{make_graph_two_operations} and Propositions \ref{plus_good} and \ref{oplus_good} imply that $3) \Rightarrow 1)$.
\end{proof}

\subsection{Proof of Theorem \ref{thm:lifted_nosplits}}

\label{proof:lifted_nosplits}
\begin{proof}[Proof of Theorem \ref{thm:lifted_nosplits}]
That the inclusion $\Gbicouplings{\mu}{\nu} \subseteq \newGbicouplings{\mu}{\nu}$ may be strict when $G$ has a split follows by the example below.

\begin{example}
    Let $U_1, U_2, U_3\sim \mathrm{Unif}[0,1]$ and $X_2, X_3 \sim \frac{1}{2}(\delta_0 + \delta_1)$ be independent random variables. Let $Y_i \coloneq \mathbbm{1}_{U_1 \geq 0.5}X_i + \mathbbm{1}_{U_1 < 0.5}(1-X_i)$ for $i \in \curlybrackets{2,3}$. Then $\pi \coloneq \law{(0,X_2, X_3), (0,Y_2,Y_3)}$ is not in $\specificGbicouplings{\Gsplit}{\mu}{\mu}$ for $\mu \coloneq \law{0,X_2,X_3} = \law{0,Y_2,Y_3} \in \specificGprobabilities{\Gsplit}{\mbb{R}^3}$, as $(X_3, Y_3)$ is not independent of $(X_2, Y_2)$.

   Let $T : (\mbb{R} \times [0,1])^3 \to (\mbb{R} \times [0,1])^3$ be given by $T((x_i, u_i))_{1 \leq i \leq 3} = ((x_1, u_1), (\mathbbm{1}_{u_1 \geq 0.5} x_2 + \mathbbm{1}_{u_1 < 0.5}(1-x_2), u_2), (\mathbbm{1}_{u_1 \geq 0.5} x_3 + \mathbbm{1}_{u_1 < 0.5}(1-x_3), u_3))$. $T$ is $\Gsplit$-biadapted and it easy to see that $T((0,U_1), (X_2, U_2), (X_3, U_3)) = ((0,U_1), (Y_2, U_2), (Y_3, U_3))$. Noting that $U_{1:3}$ is independent of $Y_2, Y_3$, we see that $\pushforward{T}{\hat{\mu}} = \hat{\mu}$ and thus $\pi \in \newspecificGbicouplings{\Gsplit}{\mu}{\mu}$.
\end{example}

Example \ref{example:split_not_closed} below shows that when $G$ contains a split $\Gbicouplings{\mu}{\nu}$ may not be weakly closed. 

\begin{example}
    \label{example:split_not_closed}
    Consider the graph $\Gsplit$, with vertices $\curlybrackets{1,2,3}$ and edges $\curlybrackets{(1,2), (1,3)}$. Let $P \coloneq \frac{1}{2} \brackets{\delta_{-1} + \delta_1}$, $\lambda$ denote the uniform measure on $[0,1]$, and $\mu \coloneq \lambda \otimes P \otimes P \in \specificGprobabilities{\Gsplit}{\mbb{R}^3}$.

    Let $(X_1, X_2, X_3, Z) \sim \mu \otimes P$. For $k \geq 3$ define $Y_1^{(k)} \coloneq \mathrm{frac}\brackets{X_1 + \frac{1}{k}Z}$ and $Y_i \coloneq ZX_i$ for $i = 2,3$, where by $\mathrm{frac}(x)$ we mean the fractional part of $x$, namely $x - \lfloor x \rfloor$. Note that $\law{Y_1^{(k)}, Y_2, Y_3} = \mu$ and $\sigma\brackets{X_1, Y_{1}^{(k)}} = \sigma\brackets{X_1, Z}$. Thus, it is easy to see that $$\law{(X_1, X_2, X_3), (Y_1^{(k)}, Y_2, Y_3)} \in \specificGbicouplings{\Gsplit}{\mu}{\mu}.$$

    However, as $k \to \infty$, the law above converges weakly to $\law{(X_1, X_2, X_3), (X_1, Y_2, Y_3)}$, which is not $\Gsplit$-bicausal, as $(X_2, Y_2)$ is not conditionally independent from $(X_3, Y_3)$ given $(X_1, X_1)$. Thus, $\specificGbicouplings{\Gsplit}{\mu}{\mu}$ is not weakly closed. 

    This example may be extended to any DAG which has a split, i.e., has $\Gsplit$ as a proper subgraph, by adding in constant random variables for all other nodes in the graph. In a similar way, the counterexample in \cite[Proposition 3.6(iv)]{Eckstein2023CausalGraphs} for $\Gmarkov$ can be extended to any DAG which is not transitively closed, thus showing $\Gbicouplings{\mu}{\nu}$ is not weakly closed in general for such DAGs.
\end{example}

We now show that 
$\Gbicouplings{\mu}{\nu} = \newGbicouplings{\mu}{\nu}$ and is weakly closed if $G$ contains no splits (and is transitively closed). We may assume without loss of generality that $G$ is sorted. 

To show that $\Gbicouplings{\mu}{\nu}$ is weakly closed we use the characterisation of graph bicausality from point 2) of Proposition \ref{equivalent_characterisations}. Let $(\law{X^{(k)}, Y^{(k)}})_{k \geq 1}$ be a sequence in $\Gbicouplings{\mu}{\nu}$ converging weakly to some $\law{X,Y} \in \probabilities{\mc{X} \times \mc{Y}}$. We first note that $\couplings{\mu}{\nu}$ is weakly closed and thus $\law{X,Y} \in \couplings{\mu}{\nu}$. By Proposition \ref{equivalent_characterisations} 2), we know that for $2 \leq i \leq n$:
$$\condindep{(X^{(k)}_i, Y^{(k)}_i)}{X^{(k)}_{\parents{i}}, Y^{(k)}_{\parents{i}}}{(X^{(k)}_{1:i-1}, Y^{(k)}_{1:i-1})}, \quad \condindep{X^{(k)}_i}{X^{(k)}_{\parents{i}}}{Y^{(k)}_{\parents{i}}}, \quad \condindep{Y^{(k)}_i}{Y^{(k)}_{\parents{i}}}{X^{(k)}_{\parents{i}}}.$$

The first set of conditional independence relations above imply by Proposition \ref{prop:Gcomp_characterisation} that $\law{(X^{(k)}_i, Y^{(k)}_i)_{i \in V}}$ is $G$-compatible (here we associate to each node $i$ the random variable $(X^{(k)}_i, Y^{(k)}_i)$). Since $G$ is transitively closed and contains no splits, by Proposition \ref{prop:P_G_closed} the set $\Gprobabilities{(\mc{X}_i, \mc{Y}_i)_{i \in V}}$ is weakly closed, and thus contains $\law{(X_i, Y_i)_{i \in V}}$, implying that $\condindep{(X_i, Y_i)}{X_{\parents{i}}, Y_{\parents{i}}}{(X_{1:i-1}, Y_{1:i-1})}$ for $2 \leq i \leq n$. 

%$(Z_1, Z_2), (W_1, W_2)$ bicausal iff $\condindep{Z_1}{W_1}{W_2}$ and $\condindep{W_1}{Z_1}{Z_2}$. Substitute $Z_1 = X_{\parents{i}}, Z_2 = X_i$,  $W_1 = Y_{\parents{i}}$, $W_2 = Y_i$. 

The second and third sets of conditional independence relations above are equivalent to saying that $\law{(X^{(k)}_{\parents{i}}, X^{(k)}_{i}), (Y^{(k)}_{\parents{i}}, Y^{(k)}_{i})}$ are ($G_2$)-bicausal couplings of $\law{(X^{(k)}_{\parents{i}}, X^{(k)}_{i})}$ and $\law{(Y^{(k)}_{\parents{i}}, Y^{(k)}_{i})}$ viewed as 2-step stochastic processes for all $2 \leq i \leq n$. However, the sets $\specificGbicouplings{G_2}{\mu_{\curlybrackets{i} \cup \parents{i}}}{\nu_{\curlybrackets{i} \cup \parents{i}}}$ are weakly closed by \cite[Proposition 3.4]{Beiglbock2022Biadapted}, see also \cite{beiglbock2018denseness}. Thus, we obtain that the same conditional independence relations hold in the limit for $2 \leq i \leq n$: $\condindep{X_i}{X_{\parents{i}}}{Y_{\parents{i}}}, \condindep{Y_i}{Y_{\parents{i}}}{X_{\parents{i}}}$. 

We conclude via Proposition \ref{equivalent_characterisations} 2) that $\law{X,Y} \in \Gbicouplings{\mu}{\nu}$, and so $\Gbicouplings{\mu}{\nu}$ is weakly closed.

The proof that $\Gbicouplings{\mu}{\nu} = \newGbicouplings{\mu}{\nu}$ now proceeds similarly to the proof of Theorem \ref{glue_characterisation}. Clearly, the result is true if $G$ is the graph with just one node. We can now use induction via Lemma \ref{make_graph_two_operations} by noting the following two facts:

First, it is easy to see, for any two DAGs $G_1, G_2$ and $\mu_i, \nu_i$ $G_i$-compatible measures, that $\specificGbicouplings{G_1 \oplus G_2}{\mu_1 \otimes \mu_2}{\nu_1 \otimes \nu_2}$ is equal to the set $\specificGbicouplings{G_1}{\mu_1}{\nu_1} \otimes \specificGbicouplings{G_2}{\mu_2}{\nu_2}$, and the set $\newspecificGbicouplings{G_1 \oplus G_2}{\mu_1 \otimes \mu_2}{\nu_1 \otimes \nu_2}$ is equal to the set $\newspecificGbicouplings{G_1}{\mu_1}{\nu_1} \otimes \newspecificGbicouplings{G_2}{\mu_2}{\nu_2}$ (modulo a reordering of the variables). So, if $\newspecificGbicouplings{G_i}{\mu_i}{\nu_i} = \specificGbicouplings{G_i}{\mu_i}{\nu_i}$, then $\newspecificGbicouplings{G_1 \oplus G_2}{\mu_1 \otimes \mu_2}{\nu_1 \otimes \nu_2} = \specificGbicouplings{G_1 \oplus G_2}{\mu_1 \otimes \mu_2}{\nu_1 \otimes \nu_2}$.

Secondly, any $\pi \in \newspecificGbicouplings{G^+}{\mu+}{\nu+}$ for $\mu^+ \in \specificGprobabilities{G^+}{\mc{X}_{1:n+1}}, \nu^+ \in \specificGprobabilities{G^+}{\mc{Y}_{1:n+1}}$ can be disintegrated as $\pi(dx_{1:n+1}, dy_{1:n+1}) = \pi_{1:n}(dx_{1:n}, dy_{1:n}) \otimes \pi_{n+1}^{x_{1:n}, y_{1:n}}(dx_{n+1}, dy_{n+1})$ with $\pi_{1:n} \in \newGbicouplings{\mu^+_{1:n}}{\nu^+_{1:n}}$ and $\pi_{n+1}^{x_{1:n}, y_{1:n}}(dx_{n+1}, dy_{n+1}) \in \couplings{\mu^{+, x_{1:n}}_{n+1}}{\nu^{+, y_{1:n}}_{n+1}}$ $\pi$-a.s. To see this, let $\Hat{\pi} \in \specificGmonge{G^+}{\Hat{\mu}^+}{\Hat{\nu}^+}$ be such that $\pi$ is the marginal of $\Hat{\pi}$ on the $\mc{X}_{1:n+1} \times \mc{Y}_{1:n+1}$ space. When $(X_{1:n+1}, U_{1:n+1}, Y_{1:n+1}, V_{1:n+1}) \sim \Hat{\pi}$, then by the definition of $G^+$-biadapted maps, $$\law{X_{1:n}, U_{1:n}, Y_{1:n}, V_{1:n}} \in \Gmonge{\Hat{\mu}^+_{1:n}}{\Hat{\nu}^+_{1:n}},$$ and so $\law{X_{1:n}, Y_{1:n}} \in \newGbicouplings{\mu^+_{1:n}}{\nu^+_{1:n}}$. Moreover, $\condlaw{X_{n+1}, Y_{n+1}}{X_{1:n}, Y_{1:n}}$ is equal to
$$\condexpectation{\condlaw{X_{n+1}, Y_{n+1}}{X_{1:n}, U_{1:n}, Y_{1:n}, V_{1:n}}}{X_{1:n}, Y_{1:n}} \in \couplings{\mu_{n+1}^{+, X_{1:n}}}{\nu_{n+1}^{+, Y_{1:n}}},$$
since $\condlaw{X_{n+1}, Y_{n+1}}{X_{1:n}, U_{1:n}, Y_{1:n}, V_{1:n}} \in \couplings{\mu_{n+1}^{+, X_{1:n}}}{\nu_{n+1}^{+, Y_{1:n}}}$ (by independence of $U$ from $X$, and $V$ from $Y$), which is weakly closed and convex. Thus, by Proposition \ref{equivalent_characterisations}, we see that if $\newGbicouplings{\mu^+_{1:n}}{\nu^+_{1:n}} = \Gbicouplings{\mu^+_{1:n}}{\nu^+_{1:n}}$, then $\newspecificGbicouplings{G^+}{\mu^+}{\nu^+} = \specificGbicouplings{G^+}{\mu^+}{\nu^+}$.
\end{proof}

\subsection{Proof of Theorem \ref{main_dense}}
\label{proof_dense}

This section closely follows \cite{Beiglbock2022Biadapted}. We adapt their proof that couplings generated by biadapted Monge maps are dense in the set of bicausal couplings into a proof of the denseness of couplings generated by $G$-biadapted Monge maps in the set of lifted $G$-bicausal couplings. The extra difficulty added by the graph structure is handled by a more technical argument in Proposition \ref{partitions} and by inducting by adding leaf nodes, back to front, as opposed to front to back in time as in \cite{Beiglbock2022Biadapted}.

Since the sets $\Gmonge{\mu}{\nu}$ and $\Gbicouplings{\mu}{\nu}$ are invariant under relabelling, without any loss of generality, throughout this section, we assume $G$ is sorted and when $\mu$ satisfies Assumption \ref{continuity_assumption} it does so with respect to this ordering. In addition, we consider the metric on $\mc{X}$ given by $d(x,y) \coloneqq \sum_{i=1}^n d_i(x_i,y_i)$ for metrics $d_i$ on $\mc{X}_i$ metrising the topologies of $\mc{X}_i$ and which we assume to be bounded above by $1$. This metric $d$ then metrises the product topology of $\mc{X}$. 
Under this assumption, the following result about perturbations of a probability measure into its lifted version holds.

\begin{proposition}
    \label{partitions}
   Suppose $\mu$ satisfies Assumption \ref{continuity_assumption} and $G$ is transitively closed. Then for any $\eps > 0$ there exists a $G$-biadapted map $T : \mc{X} \to \Hat{\mc{X}}$ such that $\pushforward{T}{\mu} = \widehat{\mu}$ and 
   $$\int d(x, \proj{\mc{X}}(T(x))) \, d\mu(x) \leq \eps.$$
\end{proposition}
\begin{proof}
We again proceed by induction on the number of nodes $n$ of the graph $G$. If $n=1$, this follows from \cite[Lemma 3.9]{Beiglbock2022Biadapted}.

For $n > 1$ assume the result holds for all graphs with at most $n-1$ nodes. Assuming without loss of generality via a relabelling that $G$ is sorted, let $\tilde{G} \coloneqq G_{1:n-1}$. In particular, the result holds for $\tilde{G}$. Fix $\eps > 0$.

Let $f: \mc{X}_{\parents{n}} \to \probabilities{\mc{X}_n}$ be the measurable function given by the disintegration $\xparents{n} \mapsto \mu_n^{\xparents{n}}$. We endow $\probabilities{\mc{X}_n}$ with the $1$-Wasserstein metric $W_1$ induced by $d_n$, under which it becomes a Polish space. We can thus apply Lusin's Theorem \cite[Box 1.6]{Santambrogio} to obtain a compact set $K \subseteq \mc{X}_{\parents{n}}$ such that $\mu_{\parents{n}}(K^c) \leq \eps$ and $f|_K$ is uniformly continuous. Thus, there exists a $\delta > 0$ such that $W_1(f(\xparents{n}), f(\yparents{n})) \leq \eps$ whenever $\xparents{n}, \yparents{n} \in K$ and $\sum_{i\in \parents{n}}d_i(x_i, y_i) \leq \delta$. We may choose $\delta \leq 1$.

By the induction hypothesis, there exists $\Tilde{T}:\mc{X}_{1:n-1} \to \hat{\mc{X}}_{1:n-1}$ a $\Tilde{G}$-biadapted map such that $\pushforward{\Tilde{T}}{\mu_{1:n-1}} = \widehat{\mu_{1:n-1}}$ and
$$\int \sum_{i=1}^{n-1} d_i(x_i, \proj{\mc{X}_{i}}(\Tilde{T}_i(x_i, \xparents{i}))) \, d\mu_{1:n-1}(x_{1:n-1}) \leq \delta\eps \leq \eps.$$ 
In particular, this implies that $\mu_{\parents{n}}(A) \leq \eps$ for the set $$A \coloneqq \setdef{\xparents{n} \in \mc{X}_{\parents{n}}}{\sum_{i\in \parents{n}} d_i(x_i, \proj{\mc{X}_{i}}(\Tilde{T}_i(x_i, \xparents{i}))) > \delta}.$$

Now, as $G$ is transitively closed, we may define the map $\Tilde{T}_{\parents{n}}:\mc{X}_{\parents{n}} \to \hat{\mc{X}}_{\parents{n}}$ given by $(x_i)_{i\in \parents{n}} \mapsto (\Tilde{T}_i(x_i, \xparents{i}))_{i \in \parents{n}}$. Since $\pushforward{\Tilde{T}}{\mu_{1:n-1}} = \widehat{\mu_{1:n-1}}$ transitive closure implies $\pushforward{{\Tilde{T}_{\parents{n}}}}{\mu_{\parents{n}}} = \widehat{\mu_{\parents{n}}}$. Thus, for $B \coloneqq \tilde{T}_{\parents{n}}^{-1}(K^c \times [0,1]^{\abs{\parents{n}}})$ we have
\begin{align*}
    \mu_{\parents{n}}(B) = \mu_{\parents{n}}(\tilde{T}_{\parents{n}}^{-1}(K^c \times [0,1]^{\abs{\parents{n}}})) = \widehat{\mu_{\parents{n}}}(K^c \times [0,1]^{\abs{\parents{n}}}) = \mu_{\parents{n}}(K^c) \leq \eps.
\end{align*}

Using a measurable selection argument \cite[Proposition 7.50]{BerstekasShreve1978}, since we know Monge couplings are dense in the set of couplings $\couplings{\mu_n^{\xparents{n}}}{\widehat{\mu_n}^{\Tilde{T}_{\parents{n}}(\xparents{n})}}$, we may find a measurable map $T_n : \mc{X}_n \times \mc{X}_{\parents{n}} \to \hat{\mc{X}}_n$ such that $T_n^{\xparents{n}}$ is bijective, $\pushforward{T_n^{\xparents{n}}}{\mu_n^{\xparents{n}}} = \widehat{\mu_n}^{\Tilde{T}_{\parents{n}}(\xparents{n})}$, and $$\int d_n(x_n, \proj{\mc{X}_n}(T_n^{\xparents{n}}(x_n))) \, d\mu_n^{\xparents{n}}(x_n) \leq W_1\brackets{\mu_n^{\xparents{n}}, \mu_n^{\Tilde{T}_{\parents{n}}(\xparents{n})}} + \eps.$$

The key to our construction so far is that for $\xparents{n} \in K \cap A^c \cap B^c$, $$W_1\brackets{\mu_n^{\xparents{n}}, \mu_n^{\Tilde{T}_{\parents{n}}(\xparents{n})}} \leq \eps.$$

Thus, if we define $T : \mc{X} \to \hat{\mc{X}}$ via $T(x) = (\Tilde{T}(x_{1:n-1}), T_n(x_n, \xparents{n}))$, $T$ is $G$-biadapted by Lemma \ref{recursive_Gbiadapted}, $\pushforward{T}{\mu} = \widehat{\mu}$, and 
\begin{align*}
    &\int \sum_{i=1}^{n} d_i(x_i, \proj{\mc{X}_{i}}({T}_i^{\xparents{i}}(x_i))) \, d\mu(x) = \int \sum_{i=1}^{n-1} d_i(x_i, \proj{\mc{X}_{i}}(\Tilde{T}_i^{\xparents{i}}(x_i))) \, d\mu_{1:n-1}(x_{1:n-1}) \\&+ \iint d_n(x_n, \proj{\mc{X}_n}(T_n^{\xparents{n}}(x_n))) \, d\mu_n^{\xparents{n}}(x_n) \, d\mu_{\parents{n}}(\xparents{n}) \\ 
    &\leq \eps + \mu_{\parents{n}}(K \cap A^c \cap B^c) \cdot 2\eps + \mu_{\parents{n}}((K \cap A^c \cap B^c)^c) \cdot 1 \leq \eps + 2\eps + 3\eps = 6\eps.
\end{align*}

As $\eps >0$ was fixed arbitrarily, the result follows.
\end{proof}

\begin{figure}[h]
\begin{center}
\begin{tikzpicture}[node distance=1cm]

  \node (X)        at (0,0)   {$\mu$};
  \node (Y)        at (6,0)  {$\nu$};
  \node (T)        at (3, 1.2) {$T_m$};

  \draw[dotted] (X)  -- (Y) node[midway, above] {$\pi$};

  \node (Xhat) at (0,2) {$\hat{\mu}$};
  \node (Yhat) at (6,2) {$\hat{\nu}$};

  \draw[-stealth] (Xhat)  -- (Yhat) node[midway, above] {$\hat{T}$};

  \draw[-stealth] (X) -- (Xhat) node[midway, left] {$\phi_m$};
  \draw[stealth-] (Y) -- (Yhat) node[midway, right] {$\psi_m^{-1}$};

  \draw[-stealth] (X) edge [bend left=30] (Y);

\end{tikzpicture}
\end{center}
\caption{A visual diagram of the proof of Theorem \ref{main_dense}.}
\label{denseness_diagram}
\end{figure}
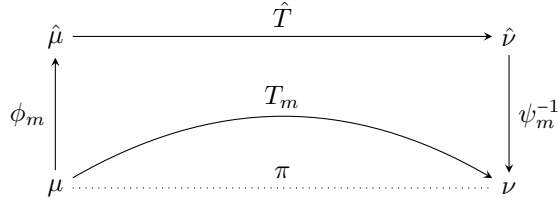

\begin{proof}[Proof of Theorem \ref{main_dense}]
    By Theorem \ref{thm:lifted}, $\Gbicouplings{\mu}{\nu} \subseteq \newGbicouplings{\mu}{\nu}$, thus we may assume more generally that $\pi \in \newGbicouplings{\mu}{\nu}$. 
    Equipped with Proposition \ref{partitions}, the proof can now follow analogously to the proof of \cite[Theorem 3.11]{Beiglbock2022Biadapted}, but using the defining property of a lifted $G$-bicausal coupling.
    Hence, we only sketch the idea of the proof, as illustrated in Figure \ref{denseness_diagram}. 
    
    Let $m \geq 1$. By the definition of lifted $G$-bicausality, there exists a $G$-biadapted map $\hat{T} : \hat{\mc{X}} \to \hat{\mc{Y}}$ such that $\pushforward{\hat{T}}{\hat{\mu}} = \hat{\nu}$ and the projection of its induced Monge coupling onto $\mc{X} \times \mc{Y}$ is $\pi$. Let $\phi_m, \psi_m$ be the $G$-biadapted perturbation maps obtained from Proposition \ref{partitions} with $\eps = \frac{1}{m}$ for $\mu$ and $\nu$ respectively. Then, if we define $T_m \coloneqq \psi_m^{-1} \circ \hat{T} \circ \phi_m$, by Lemma \ref{closed_composition} as $G$ is transitively closed, $T_m$ is $G$-biadapted. As the Monge coupling induced by $\hat{T}$ projects to $\pi$, and the perturbations given by $\phi_m, \psi_m$ can be made arbitrarily small as $m \to \infty$, one can show that the Monge maps induced by the $T_m$ approximate $\pi$ as $m \to \infty$.
\end{proof}

\begin{remark}
    Note that the proof above shows that one could have in fact replaced $\Gbicouplings{\mu}{\nu}$ by $\newGbicouplings{\mu}{\nu}$ in the statement of Theorem \ref{main_dense} to obtain that $\Gmonge{\mu}{\nu}$ is weakly dense in $\newGbicouplings{\mu}{\nu}$. 
\end{remark}
\subsection{Proof of Dynamic Programming Principle Results}
\label{proof_dpp}

\subsection*{DPP Compatible Graphs}

We begin with a further discussion of our definition of a perfect DAG in Definition \ref{def:DAGthree_structures} and by proving Propostion \ref{prop:perfection}.
In \cite[Section 2.1]{LauritzenBook}, the notion of a perfect DAG $G$ is defined via the notion of \emph{moralisation}. For a DAG $G$, we may define its \emph{moral graph} $G^m$ to be the \emph{un}directed graph with the same vertex set $V$ as $G$, which has an edge between $i \neq j \in V$ if and only if $i$ is a parent of $j$, $j$ is a parent of $i$, or there exists a $k \in V$ such that $i, j \in \parents{k}$. If the third condition adds no new edges to $G^m$ other than those added by the first two conditions, $G$ is said to be perfect. In other words a DAG $G$ is perfect if whenever $i, j \in \parents{k}$ with $i \neq j$  then either $i \in \parents{j}$ or $j \in \parents{i}$, or equivalently, if it has no merges. We now establish the equivalent characterisation introduced in Section \ref{section_dpp}.

\begin{proof}[Proof of Proposition \ref{prop:perfection}]
    To show that 2) implies 1), let $i < j < k$ be such that $i, j \in \parents{k}$. Since, $\parents{k} \subseteq \parents{\prox(k)} \cup \curlybrackets{\prox(k)}$, and $\prox(k)$ is by definition the largest (with respect to the sorted order) parent of $k$, we must have that either $j = \prox(k)$ and $i \in \parents{\prox(k)}$ in which case $i$ is a parent of $j$, or both $i$ and $j$ are parents of $\prox(k) < k$. One can then restart this argument replacing $k$ with $\prox(k)$. As every time $k$ must decrease, we cannot iterate indefinitely and thus must obtain that $i$ is a parent of $j$ after a finite number of steps.

    Conversely, assume 1) holds and let $k \in V$, and $i \in \parents{k}$. Assume that $i$ is not equal to $j \coloneq \prox(k) \in \parents{k}$. Then, by 1) we must have that $i$ is a parent of $j$ or $j$ is a parent of $i$, as otherwise $(k,i,j)$ would be a merge. As $j$ is the largest parent of $k$ in the sorted order, we cannot have that $j \in \parents{i}$. Thus, $i \in \parents{j} = \parents{\prox(k)}$. Therefore, $\parents{k} \subseteq \parents{\prox(k)} \cup \curlybrackets{\prox(k)},$ and 2) holds.
\end{proof}

To better understand the structure of a perfect DAG, we associate to $G$ its proximal graph:

\begin{definition}[Proximal graph]
    The \emph{proximal graph} of $G$, $\prox(G)$, is defined as the directed graph with vertex set $V$ and edge set $$\prox(E) \coloneqq \curlybrackets{(\prox(i), i) \, \vert \, i \in V, \parents{i} \neq \varnothing} \subseteq E.$$
\end{definition}

As $G$ is sorted, $\prox(i) \leq i$ and the equality holds if and only if $i$ has no parents. Thus, $\prox(G)$ is acyclic and each of its vertices has at most one parent by definition. In other words, $\prox(G)$ is a forest with edges directed from parents to their children. 

Hence, if $i$ is an ancestor of $j$ in $\prox(G)$, there exists a unique directed path in $\prox(G)$ from $i$ to $j$, which we call the proximal path from $i$ to $j$. Clearly, as the only possible parent in $\prox(G)$ of a vertex $v$ is $\prox(v)$, such a proximal path must take the form $i = \prox^m(j) \to \ldots \to \prox(j) \to j$ for some $m \geq 1$. 

Moreover, for $i \in V$ we define its depth, $\depth(i)$ as the length of the path from the root of the forest component of $\prox(G)$ containing $i$ to the vertex $i$. Equivalently, $\depth(i)$ is the smallest $m \geq 0$ such that $\prox^{m}(i) = \prox^{m + 1}(i)$.

Figure \ref{fig:proximal_figure} illustrates the concepts defined above.

\begin{figure}[H]
    \centering
    \definecolor{cbblue}{RGB}{0,114,178}      
\definecolor{cborange}{RGB}{230,159,0}      
\definecolor{cbvermillion}{RGB}{213,94,0} 

\begin{tikzpicture}[
    >=stealth,
    every node/.style={circle, fill=black, inner sep=2.5pt},
    blacknode/.style={circle, fill=black, inner sep=2.5pt},
    bluenode/.style={circle, fill=cbblue, inner sep=2.5pt},
    orangenode/.style={circle, fill=cborange, inner sep=2.5pt},
    rednode/.style={circle, fill=cbvermillion, inner sep=2.5pt},
    grayedge/.style={->, draw=gray!70, thick},
    blackedge/.style={->, draw=black, thick},
    orangeedge/.style={->, draw=cborange, very thick, densely dashed},
    label/.style={font=\large, fill=none, inner sep=0pt},
    dotedge/.style={-, draw=gray!30, thick, dotted}
]

\begin{scope}[local bounding box=Gbox]
    \node (r)   at (0,2.08)      {};
    \node (l1)  at (-1.2,1.04) {};
    \node (r1)  at ( 1.2,1.04) {};
    \node (l1a) at (-1.8,0) {};
    \node (l1b) at (-0.6,0) {};
    \node (r1a) at ( 0.6,0) {};
    \node (r1b) at ( 1.8,0) {};
    \node (l2a) at ( 0.6,-1.04) {};
    \node (l2b) at ( 1.8,-1.04) {};

    \draw[blackedge] (r)   -- (l1);
    \draw[blackedge] (r)   -- (r1);
    \draw[blackedge] (l1)  -- (l1a);
    \draw[blackedge] (l1)  -- (l1b);
    \draw[blackedge] (r1)  -- (r1a);
    \draw[blackedge] (r1)  -- (r1b);
    \draw[blackedge] (r1a) -- (l2a);
    \draw[blackedge] (r1b) -- (l2b);

    \draw[grayedge] (r)  to[out=-25, in=85]  (r1b);
    \draw[grayedge] (r)  to[out=-5, in=55]  (l2b);
    \draw[grayedge] (r1) to[out=-90, in=55] (l2a);
    \draw[grayedge] (r) to[left, out=-130, in=80] (l1b);

    \node[label] at (-1.5,2.08) {$G$};

    \node[label, rotate=-90] (d) at (-3.5, 0.52)  {};
\end{scope}

\begin{scope}[xshift=5.4cm, local bounding box=Pbox]
    \node[bluenode]  (R)    at (0,2.08)       {};
    \node[blacknode] (L1)   at (-1.2,1.04)  {};
    \node[orangenode](R1)   at ( 1.2,1.04)  {};
    \node[blacknode] (L1a)  at (-1.8,0)  {};
    \node[blacknode] (L1b)  at (-0.6,0)  {};
    \node[blacknode] (R1a)  at ( 0.6,0)  {};
    \node[orangenode](R1b)  at ( 1.8,0)  {};
    \node[blacknode] (L2a)  at ( 0.6,-1.04) {};
    \node[rednode]   (L2b)  at ( 1.8,-1.04) {};

    \node[label, rotate=-90] (d) at (3.5, 0.52)  {\small\textbf{Depth}};
    \node[label] (0) at (3.0, 2.08)  {\small$0$};
    \node[label] (1) at (3.0, 1.04)  {\small$1$};
    \node[label] (2) at (3.0, 0)  {\small$2$};
    \node[label] (3) at (3.0,-1.04)  {\small$3$};

    \begin{pgfonlayer}{background}
    \draw[dotedge] (R)   -- (0);
    \draw[dotedge] (L1)  -- (1);
    \draw[dotedge] (L1a)  -- (2);
    \draw[dotedge] (L2a)  -- (3);
    \end{pgfonlayer}

    \draw[blackedge] (R)   -- (L1);
    \draw[blackedge] (L1)  -- (L1a);
    \draw[blackedge] (L1)  -- (L1b);
    \draw[blackedge] (R1)  -- (R1a);
    \draw[blackedge] (R1a) -- (L2a);

    \draw[orangeedge] (R)   -- (R1);
    \draw[orangeedge] (R1)  -- (R1b);
    \draw[orangeedge] (R1b) -- (L2b);

    \node[label] at (-1.5,2.08) {$\mathrm{prox}(G)$};

    \node[star, star points=5, fill=white, draw=none, inner sep=0pt, minimum size=4pt] at (L2b) {};
    \node[regular polygon, regular polygon sides=4, fill=white, draw=none, inner sep=0pt, minimum size=4pt] at (R) {};
\end{scope}

\end{tikzpicture}
    \caption{Example of a DAG $G$ and its associated proximal graph. $\prox(G)$ is annotated with the depths of its vertices, and the proximal path from the blue vertex with a white square to the orange vertex with a white star is highlighted in dashed yellow arrows.}
    \label{fig:proximal_figure}
\end{figure}
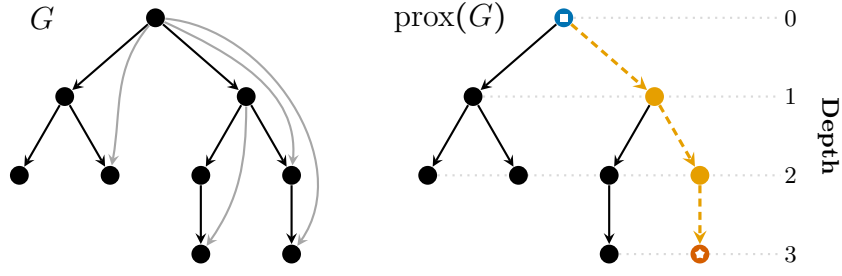

\begin{proposition}
\label{condition_explained}
    Let $G$ be a sorted DAG. Then the following are equivalent:
    \begin{enumerate}[label=\arabic*)]
        \item $G$ is perfect;
        \item For all $i, j \in V$ such that $i \in \parents{j}$, we have that $i$ is an ancestor of $j$ in $\prox(G)$ and $i \in \parents{v}$ for all vertices $v \neq i$ of the proximal path between $i$ and $j$.
    \end{enumerate}
\end{proposition}
\begin{proof}
    In this proof, we use the equivalent characterisation of a perfect DAG established in Proposition \ref{prop:perfection}.
    
    Suppose $G$ is perfect. 
    We prove 2) by induction on the depth of $j$. If $\depth(j) = 0$, $\parents{j} = \varnothing$, so there is nothing to prove. Let $n \geq 1$ and assume we have proved $2)$ for all $i, j \in V$ with $\depth(j) \leq n - 1$. 

    Fix $i, j \in V$ such that $\depth(j) = n $ and $i \in \parents{j}$. 1) then implies that $i \in \parents{j} \subseteq \parents{\prox(j)} \cup \curlybrackets{\prox(j)}$. 
    If $i = \prox(j)$, then clearly the condition in 2) holds for the pair $(i,j)$. Otherwise, we must have that $i \in \parents{\prox(j)}$. As $\depth(\prox(j)) = n - 1$, by the induction hypothesis we see then that $i$ must be an ancestor of $\prox(j)$ in $\prox(G)$ and $i \in \parents{v}$ for all vertices $v \neq i$ of the proximal path between $i$ and $\prox(j)$. But any $\prox(G)$ ancestor of $\prox(j)$ is also an ancestor of $j$ and the proximal path from $i$ to $j$ is formed by the vertices in the proximal path from $i$ to $\prox(j)$, and $j$. Thus 2) also holds for the pair $(i, j)$ and the proof follows by induction.

    For the converse, fix $j \in V$. If $i \in \parents{j}$, then as $\prox(j)$ must always be on any proximal path ending in $j$, 2) implies that either $i = \prox(j)$ or $i$ is a parent in $G$ of $\prox(j)$. Thus, $i \in \parents{\prox(j)} \cup \curlybrackets{\prox(j)}$.
\end{proof}

\subsection*{DPP for $G$-bicausal Transport}
The proof is inspired by \cite[Proposition 5.2]{Backhoff2016CausalDiscrete}.

\begin{proof}[Proof of Proposition \ref{dpp_proof}]
    We begin by showing recursively that each $V_k^{G,c}$ is well-defined and lower semi-analytic (l.s.a.). 

    Let $D_k$ denote the set $$\curlybrackets{(x_{\parents{k}}, y_{\parents{k}}, \pi) \in \mc{X}_{\parents{k}} \times \mc{Y}_{\parents{k}} \times \probabilities{\mc{X}_{k} \times \mc{Y}_{k}} \, \middle| \, \pi \in \couplings{\mu_k^{\xparents{k}}}{\nu_k^{\yparents{k}}}},$$
    and the function $g_k : D_k \to \mbb{R}$ be defined as
    $$g_k(x_{\parents{k}}, y_{\parents{k}}, \pi) \coloneqq \int f_k(x_{\parents{k}}, y_{\parents{k}}, x_k, y_k )\, d\pi(x_k,y_k),$$
    where $f_k : \mc{X}_{\parents{k}} \times \mc{Y}_{\parents{k}} \times \mc{X}_{k} \times \mc{Y}_{k} \to \mbb{R}$ is given by
    $$f_k(x_{\parents{k}}, y_{\parents{k}}, x_k, y_k) \coloneqq c_k(x_k, \xparents{k}, y_k, \yparents{k}) + \sum_{i \in A_k} V_i^{G,c}(\xparents{i}, \yparents{i}).$$
    By $D_k(x_{\parents{k}}, y_{\parents{k}})$ we denote the fibre of the set $D_k$ at $(x_{\parents{k}}, y_{\parents{k}})$, i.e., the set $\couplings{\mu_k^{\xparents{k}}}{\nu_k^{\yparents{k}}} \subseteq \probabilities{\mc{X}_{k} \times \mc{Y}_{k}}$.

    We first remark that the set $D_k$ is Borel, thus analytic. Indeed, it is the preimage of the closed (in the weak topology) set $$\setdef{(\tilde{\mu}, \tilde{\nu}, \pi) \in \probabilities{\mc{X}_k} \times \probabilities{\mc{Y}_k} \times \probabilities{\mc{X}_{k} \times \mc{Y}_{k}}}{\pi \in \couplings{\tilde{\mu}}{\tilde{\nu}}}$$
    under the Borel map $(x_{\parents{k}}, y_{\parents{k}}, \pi) \mapsto (\mu_k^{\xparents{k}}, \nu_k^{\yparents{k}}, \pi)$.

    For notational convenience in the proof below, we define $\parents{n+1} \coloneqq \varnothing$, $A_n \coloneqq \curlybrackets{n + 1}$ and $V_{n+1}^{G,c} \coloneqq 0$ so that the recursive formula in the definition of $V_{k}^{G,c}$ extends to the case $k = n$. $V_{n+1}^{G,c}$ is trivially well-defined and l.s.a. 
    
    Let $1 \leq k \leq n$ and suppose we have shown that $V_i^{G,c}$ is well-defined and l.s.a. for all $k + 1 \leq i \leq n + 1$. Fixing, $i \in A_k \subseteq [k + 1, n+1]$, $V_i^{G,c}$ is l.s.a. as a map with domain $\mc{X}_{\parents{i}} \times \mc{Y}_{\parents{i}}$. As $i \in A_k$, and $G$ is perfect, $\parents{i} \subseteq \parents{k} \cup \curlybrackets{k}$. Thus, we may extend the domain of $V_i^{G,c}$ to the product space $\mc{X}_{\parents{k}} \times \mc{Y}_{\parents{k}} \times \mc{X}_{k} \times \mc{Y}_{k}$ by composing it with the Borel projection map $(x_{\parents{k}}, y_{\parents{k}}, x_k, y_k) \mapsto (x_{\parents{i}}, y_{\parents{i}})$. By \cite[Lemma 7.30(3)]{BerstekasShreve1978} the result remains l.s.a. Now $f_k$ is defined as a sum of such functions and the l.s.a. function $c_k$, so it is l.s.a. by \cite[Lemma 7.30(4)]{BerstekasShreve1978}. In particular, $f_k$ is universally measurable so the integral $g_k$ is well-defined.

    Applying \cite[Proposition 7.48]{BerstekasShreve1978} with $X = \mc{X}_{\parents{k}} \times \mc{Y}_{\parents{k}} \times \probabilities{\mc{X}_k \times \mc{Y}_k}$, $Y = \mc{X}_k \times \mc{Y}_k$, the l.s.a. map $(x_{\parents{k}}, y_{\parents{k}}, \pi, x_k, y_k) \mapsto f_k(x_{\parents{k}}, y_{\parents{k}}, x_k, y_k)$ and the Borel stochastic kernel $q$ on $Y$ given $X$ given by $q(dx_k, dy_k \mid  \xparents{k}, \yparents{k}, \pi) = \pi(dx_k, dy_k)$, we get that $g_k$ is moreover l.s.a.

    $V_k^{G,c}(x_{\parents{k}}, y_{\parents{k}})$ is the infimum of the function $g_k$ over the fibre $D_k(x_{\parents{k}}, y_{\parents{k}})$. By \cite[Proposition 7.47]{BerstekasShreve1978}, since $D_k$ is analytic and $g_k$ is well-defined and l.s.a., then so is $V_k^{G,c}$. 
    
    It now follows by induction that $V_k^{G,c}$ is well-defined and l.s.a. for all $1 \leq k \leq n$.

    We now prove $V^{G,c}(\mu, \nu) = \dpp^{G,c}(\mu, \nu).$ Fix $\varepsilon > 0$. For each $1 \leq k \leq n$, as $D_k$ is analytic and $g_k$ is l.s.a., we can find by \cite[Proposition 7.50(b)]{BerstekasShreve1978} a universally measurable selection of $\varepsilon$-minimisers for $V_k^{G,c}$ i.e. a universally measurable map $\varphi_{k, \varepsilon} : \mc{X}_{\parents{k}} \times \mc{Y}_{\parents{k}} \to \probabilities{\mc{X}_k \times \mc{Y}_k}$ such that for all $(\xparents{k}, \yparents{k}) \in \mc{X}_{\parents{k}} \times \mc{Y}_{\parents{k}}$, $\varphi_{k, \varepsilon}(\xparents{k}, \yparents{k}) \in \couplings{\mu_k^{\xparents{k}}}{\nu_k^{\yparents{k}}}$ and moreover
    \begin{align*}
        \int &c_k(x_k, \xparents{k}, y_k, \yparents{k}) + \sum_{i \in A_k} V_i^{G,c}(\xparents{i}, \yparents{i}) \, d\varphi_{k, \varepsilon}^{\xparents{k}, \yparents{k}}(x_k, y_k) \leq  \\ &\leq V_k^{G,c}(\xparents{k}, \yparents{k}) + \varepsilon,
    \end{align*}
    where we have written $\varphi_{k, \varepsilon}^{\xparents{k}, \yparents{k}}$ instead of $\varphi_{k, \varepsilon}(\xparents{k}, \yparents{k})$, keeping with our conventions for stochastic kernels. We do the same for $\tilde{\varphi}_{k, \varepsilon}$ introduced below.

    The maps $\varphi_{k, \varepsilon}$ can be trivially extended to stochastic kernels $\tilde{\varphi}_{k, \varepsilon} : \mc{X}_{1:k-1} \times \mc{Y}_{1:k-1} \to \probabilities{\mc{X}_k \times \mc{Y}_k}$ such that $\tilde{\varphi}_{k, \varepsilon}(x_{1:k-1}, y_{1:k-1}) \coloneqq \varphi_{k, \varepsilon}(\xparents{k}, \yparents{k})$, which inherit the property of being universally measurable by \cite[Proposition 7.44]{BerstekasShreve1978}. 
    
    By \cite[Proposition 7.45]{BerstekasShreve1978} we may now compose these stochastic kernels to obtain a probability measure on $\probabilities{\mc{X} \times \mc{Y}}$ given by 
    $$\pi_{\varepsilon}(dx, dy) \coloneqq \bigotimes_{k=1}^n \tilde{\varphi}_{k, \varepsilon}^{x_{1:k-1}, y_{1:k-1}}(dx_{k}, dy_k) =
    \bigotimes_{k=1}^n \varphi_{k, \varepsilon}^{\xparents{k}, \yparents{k}}(dx_{k}, dy_k)$$
    As it always holds that $\varphi_{k, \varepsilon}^{\xparents{k}, \yparents{k}} \in \couplings{\mu_k^{\xparents{k}}}{\nu_k^{\yparents{k}}}$, we see in fact that $\pi_{\varepsilon} \in \Gbicouplings{\mu}{\nu}$, so
    \begin{align}
    \label{inequality_1}
        V^{G,c}&(\mu, \nu) \leq \int c(x,y) \, d\pi_{\varepsilon}(x,y) \\
        &= \int c_1 + \brackets{\int c_2 + \brackets{\ldots + \brackets{\int c_n \, d\varphi_{n, \varepsilon}^{\xparents{n}, \yparents{n}}} \ldots} \, d\varphi_{2, \varepsilon}^{\xparents{2}, \yparents{2}}} \, d\varphi_{1, \varepsilon}. \notag
    \end{align}
    
    Now define $I_n^c(x_{1:n-1}, y_{1:n-1}) \coloneqq \int c_n \, d\varphi_{n, \varepsilon}^{\xparents{n}, \yparents{n}}$ and recursively for $1 \leq k < n$, $$I_k^c(x_{1:k-1}, y_{1:k-1}) \coloneqq \int c_k + I_{k+1}^c(x_{1:k}, y_{1:k}) \, d\varphi_{k, \varepsilon}^{\xparents{k}, \yparents{k}}.$$
    We aim to show inductively that $I_k^c(x_{1:k-1}, y_{1:k-1}) \leq \sum_{i \in B_k} V_i^{G,c}(\xparents{i}, \yparents{i}) + (n - k + 1) \varepsilon$, where we defined the sets $B_k \coloneqq \setdef{i \geq k}{\prox(i) < k \, \mathrm{or}\, \prox(i) = i}$. Note that $B_{k+1} = A_k \cup (B_k \setminus \curlybrackets{k})$.
    
    The case for $k=n$ is immediate from the definition of  $\varphi_{n, \varepsilon}$ as an $\varepsilon$-optimiser. Let $1 \leq k < n$ and assume $I_{k + 1}^c(x_{1:k}, y_{1:k}) \leq \sum_{i \in B_{k+1}} V_i^{G,c}(\xparents{i}, \yparents{i}) + (n - k) \varepsilon$. Then
    \begin{align*}
        I_k^c(x_{1:k-1}, y_{1:k-1}) &= \int c_k + I_{k+1}^c(x_{1:k}, y_{1:k}) \, d\varphi_{k, \varepsilon}^{\xparents{k}, \yparents{k}} \\
        &\leq \int c_k + \sum_{i \in B_{k+1}} V_i^{G,c}(\xparents{i}, \yparents{i}) \, d\varphi_{k, \varepsilon}^{\xparents{k}, \yparents{k}} + (n - k) \varepsilon \\
        &= \int c_k + \sum_{i \in A_k} V_i^{G,c}(\xparents{i}, \yparents{i}) \, d\varphi_{k, \varepsilon}^{\xparents{k}, \yparents{k}} \\
        &+ \sum_{i \in B_k \setminus \curlybrackets{k}} V_i^{G,c}(\xparents{i}, \yparents{i}) + (n - k) \varepsilon \\ 
        &\leq V_k^{G,c}(\xparents{k}, \yparents{k}) + \varepsilon + \sum_{i \in B_k \setminus \curlybrackets{k}} V_i^{G,c}(\xparents{i}, \yparents{i}) + (n - k) \varepsilon \\
        &= \sum_{i \in B_k} V_i^{G,c}(\xparents{i}, \yparents{i}) + (n - k + 1) \varepsilon,
    \end{align*}
    where the first inequality follows by the inductive hypothesis, the second equality since $k \notin \parents{i}$ when $i \in B_k \setminus \curlybrackets{k}$, and the final equality by the choice of  $\varphi_{k, \varepsilon}$ as an $\varepsilon$-optimiser.
    
    By induction then, using equation (\ref{inequality_1}), we see
    \begin{align*}
        V^{G,c}(\mu, \nu) \leq I_1^c \leq \sum_{i \in B_1} V_i^{G,c}(\xparents{i}, \yparents{i}) + n \varepsilon = \dpp^{G,c}(\mu, \nu)  + n \varepsilon,
    \end{align*}
    since $i \in B_1$ if and only if $\prox(i) = i$, i.e., $i \in A_0$. 
    
    Hence, as $\varepsilon$ was fixed arbitrarily, $V^{G,c}(\mu, \nu) \leq \dpp^{G,c}(\mu, \nu)$.

    Conversely, for $\varepsilon > 0$, there exists $\pi_{\varepsilon} \in \Gbicouplings{\mu}{\nu}$ such that $\int c \, d\pi_{\varepsilon} \leq V^{G,c}(\mu, \nu) + \varepsilon$. $\pi_{\varepsilon}$ may then be disintegrated as $\pi_{\varepsilon} = 
    \bigotimes_{k=1}^n \varphi_{k, \varepsilon}^{\xparents{k}, \yparents{k}}(dx_{k}, dy_k)$ for some stochastic kernels with $\varphi_{k, \varepsilon}^{\xparents{k}, \yparents{k}} \in \couplings{\mu_k^{\xparents{k}}}{\nu_k^{\yparents{k}}}$. Using the same notation as above, one may similarly show by induction that $I_k^c(x_{1:k-1}, y_{1:k-1}) \geq \sum_{i \in B_k} V_i^{G,c}(\xparents{i}, \yparents{i})$. 
    
    Indeed, the case $k = n$ is trivial by optimality of $V_n^{G,c}$. Assuming the result has been proved for $k+1$, we see similarly as above
    \begin{align*}
         I_k^c&(x_{1:k-1}, y_{1:k-1}) = \int c_k + I_{k+1}^c(x_{1:k}, y_{1:k}) \, d\varphi_{k, \varepsilon}^{\xparents{k}, \yparents{k}} \\ &\geq \int c_k + \sum_{i \in B_{k+1}} V_i^{G,c}(\xparents{i}, \yparents{i}) \, d\varphi_{k, \varepsilon}^{\xparents{k}, \yparents{k}} \\
          &= \int c_k + \sum_{i \in A_k} V_i^{G,c}(\xparents{i}, \yparents{i}) \, d\varphi_{k, \varepsilon}^{\xparents{k}, \yparents{k}} + \sum_{i \in B_k \setminus \curlybrackets{k}} V_i^{G,c}(\xparents{i}, \yparents{i}) \\ 
         &\geq  V_k^{G,c}(\xparents{k}, \yparents{k}) + \sum_{i \in B_k \setminus \curlybrackets{k}} V_i^{G,c}(\xparents{i}, \yparents{i}) = \sum_{i \in B_k} V_i^{G,c}(\xparents{i}, \yparents{i}), 
    \end{align*}
    where the last inequality follows by optimality of $V_k^{G,c}$.

    So, $V^{G,c}(\mu, \nu) + \varepsilon \geq \int c \, d\pi_{\varepsilon} = I_1^c \geq \sum_{i \in B_1} V_i^{G,c}(\xparents{i}, \yparents{i})  = \dpp^{G,c}(\mu, \nu)$. Take $\varepsilon \searrow 0$ to see $V^{G,c}(\mu, \nu) \geq \dpp^{G,c}(\mu, \nu)$. Combining this with the previously obtained reverse inequality, we get our desired result.
\end{proof}

\subsection*{Isometry to Nested Measures}

We now specialise our cost functions to obtain an embedding of $(\Gpprobabilities{p}{\mc{X}}, W_{G,p})$ into a certain space of nested measures, which we now introduce. In what follows, we fix $p \geq 1$ and metrics $d_k$ on $\mc{X}_k$ with respect to which the space $\Gpprobabilities{p}{\mc{X}}$ is defined as usual. Moreover, given a finite collection of metric spaces $(D_1, d_{D_1}), \ldots, (D_m, d_{D_m})$, $\Pp{p}{D_1}$ is understood as a metric space under the classical $p$-Wasserstein distance induced by $d_{D_1}$, and the product space $D \coloneqq \prod_{i =1}^m D_i$ is understood as a metric space under the metric $d_D$ given by
$$d_D((x_1, \ldots, x_m), (\tilde{x}_1, \ldots, \tilde{x}_m)) ^ p \coloneqq \sum_{i =1}^m d_{D_i}^p (x_i, \tilde{x}_i).$$

\begin{definition}[$G$-nested distribution]
\label{nested_distribution}
    Let $\mc{N}_n^{G,p} \coloneqq \Pp{p}{\mc{X}_n}$. For $1 \leq k < n$ we iteratively define the spaces $\mc{N}_k^{G,p} \coloneqq \Pp{p}{\mc{X}_k \times \prod_{i \in A_k}{\mc{N}_i^{G,p}}}$. The space of \emph{$G$-nested distributions} $\mc{N}^{G,p}$ is then given by
    $$\mc{N}^{G,p} \coloneqq \prod_{k \in A_0}  \mc{N}_k^{G,p}.$$

    For $\mu \in \Gpprobabilities{p}{\mc{X}}$ and a random variable $X \sim \mu$, let $\tilde{X}_n \coloneqq \condlaw{X_n}{X_{\parents{n}}}$ and iteratively for $1 \leq k < n$, 
    $$\tilde{X}_k \coloneqq \condlaw{X_k, \brackets{\tilde{X}_i}_{i \in A_k}}{X_{\parents{k}}},$$
    where the operation $\condlaw{\cdot}{X_{\varnothing}}$ is interpreted as $\law{\cdot}$.
    Note that each random variable $\tilde{X}_k$ takes values in $\mc{N}_k^{G,p}$. 
    
    Finally, define the $G$-nested distribution corresponding to $\mu$ as $$\mathfrak{N}^G(\mu) \coloneqq \brackets{\tilde{X}_k}_{k \in A_0} \in \mc{N}^{G,p}.$$
\end{definition}

Theorem \ref{AWisWG} is a consequence of the following result:

\begin{theorem}
\label{aw_equality}
    If $G$ is a sorted perfect DAG, then the map $\mathfrak{N}^G : \Gpprobabilities{p}{\mc{X}} \to \mc{N}^{G,p}$ defined above is an isometric embedding, i.e.,
    $$W_{G,p}(\mu, \nu) = d_{\mc{N}^{G,p}}(\mathfrak{N}^G(\mu), \mathfrak{N}^G(\nu)) = AW_p (\mu, \nu), \, \forall \, \mu, \nu \in \Gpprobabilities{p}{\mc{X}}.$$

\end{theorem}
\begin{proof}
    Fix $\mu, \nu \in \Gpprobabilities{p}{\mc{X}}$. As $W_{G,p}(\mu, \nu)^p = V^{G,c}(\mu, \nu)$ for the Borel and $G$-separable cost function given by $c(x,y) \coloneqq \sum_{k = 1}^n d_{k}(x_k, y_k)^p$, we may apply Proposition \ref{dpp_proof} to see $W_{G,p}(\mu, \nu)^p = \mathrm{DPP}^{G,c}(\mu, \nu)$.

    We define $\mathfrak{N}^G_n(\mu, \xparents{n}) \coloneqq \mu_n^{\xparents{n}} \in \Pp{p}{\mc{X}_n} = \mc{N}_n^{G,p}$ and recursively for $n-1 \geq k \geq 1$, $\mathfrak{N}^G_k(\mu, \xparents{k})$ is defined as the pushforward of the measure $\mu_k^{\xparents{k}}$ under the map $\phi_{k, \xparents{k}} : \mc{X}_k \to \mc{X}_k \times \prod_{i \in A_k}{\mc{N}_i^{G,p}}$ given by
    $$x_k \mapsto \brackets{x_k, \brackets{\mathfrak{N}^G_i(\mu, \xparents{i})}_{i \in A_k}},$$
    which is well defined since for all $i \in A_k$, $i > k$ and $\parents{i} \subseteq \parents{k} \cup \curlybrackets{k}$ as $G$ is perfect. Note that $\mathfrak{N}^G_k(\mu, \xparents{k}) \in \mc{N}_k^{G,p}$ for all $1 \leq k \leq n$. We may define $\mathfrak{N}^{G}_k(\nu, \yparents{k})$ similarly. 

    Using the notation of Proposition \ref{dpp_proof}, one easily observes by induction that $$V_k^{G,c}(\xparents{k}, \yparents{k}) = W_p(\mathfrak{N}^G_k(\mu, \xparents{k}), \mathfrak{N}^G_k(\nu, \yparents{k}))^p, \, \forall \, k = 1, \ldots, n.$$

    One also observes that, in the notation of Definition \ref{nested_distribution}, $\tilde{X}_k = \mathfrak{N}^G_k(\mu, X_{\parents{k}})$ $\mu$-a.s., thus $\mathfrak{N}^G(\mu) = \brackets{\tilde{X}_k}_{k \in A_0} = \brackets{\mathfrak{N}^G_k(\mu)}_{k \in A_0}$, as any $k \in A_0$ has no parents. By symmetry, $\mathfrak{N}^G(\nu) =  \brackets{\mathfrak{N}^G_k(\nu)}_{k \in A_0}$.

    Hence, 
    \begin{align*}
        \mathrm{DPP}^{G,c}(\mu, \nu) &= \sum_{k \in A_0} V_k^{G,c} = \sum_{k \in A_0} W_p(\mathfrak{N}^G_k(\mu), \mathfrak{N}^G_k(\nu))^p \\
        &= d_{\mc{N}^{G,p}}(\mathfrak{N}^G(\mu), \mathfrak{N}^G(\nu))^p.
    \end{align*}

    We have thus shown that $\mathfrak{N}^G$ is an isometry. We now proceed to show $W_{G,p}(\mu, \nu) =  AW_p(\mu, \nu)$ via induction on the number of nodes in $G$. 

    If $G$ has a single node, then clearly $W_{G,p}(\mu, \nu) = W_p (\mu, \nu) = AW_p(\mu, \nu)$.

    Suppose we have proved that $W_{\tilde{G},p}$ coincides with $AW_p$ on spaces of $\tilde{G}$-compatible measures for any sorted perfect DAG $\tilde{G}$ with at most $n - 1$ vertices. Let $G=(V,E)$ be (as before) a sorted DAG with $n$ vertices.

    If $m \coloneqq \abs{A_0} > 1$, $G$ may be written (up to a relabelling) as $G = G_1 \oplus G_2$ for two sorted DAGs $G_1 =(V_1, E_1)$,  $G_2 =(V_2, E_2)$, where the $V_i$ partition $V$ and the $E_i = E \cap (V_i \times V_i)$ partition $E$. Moreover, we may assume without loss of generality that $V_1 = \curlybrackets{1, \ldots, m}$ for some $1 \leq m < n$ via a relabelling. Then $\mu \in \Gprobabilities{\mc{X}}$ can be written as the product measure 
    $$\mu(dx) = \mu_{V_1}(dx_{V_1}) \otimes \mu_{V_2}(dx_{V_2}),$$
    with $\mu_{V_i} \in \specificGprobabilities{G_i}{\mc{X}_{V_i}}$. Applying the first part of this proof to $G$ and the $G_i$, which are perfect, together with the induction hypothesis, we see that 
    \begin{align*}
        W_{G,p}(\mu, \nu)^p &= d_{\mc{N}^{G,p}}(\mathfrak{N}^G(\mu), \mathfrak{N}^G(\nu))^p = \sum_{i=1}^2 d_{\mc{N}^{G_i,p}}(\mathfrak{N}^{G_i}(\mu_{V_i}), \mathfrak{N}^{G_i}(\nu_{V_i}))^p \\
        &= \sum_{i=1}^2 AW_p(\mu_{V_i}, \nu_{V_i})^p = AW_p(\mu, \nu)^p,
    \end{align*}
    where the last equality follows via Lemma \ref{aw_independent}.

    If $\abs{A_0} = 1$, then we have $A_0 = \curlybrackets{1}$. Let $\tilde{G}$ be the perfect graph obtained by removing the vertex $1$ from $G$, along with any edges emitting from it. We write $\mu^{x_1} \coloneqq \bigotimes_{k=2}^n \mu_k^{\xparents{k}}, \nu^{y_1} \coloneqq \bigotimes_{k=2}^n \nu_k^{\yparents{k}} \in \specificGprobabilities{\tilde{G}}{\mc{X}_{2:n}}$. Using the DPP structure in this case, we see that
    \begin{align*}
        W_{G,p}&(\mu, \nu)^p = \inf_{\pi \in \couplings{\mu_1}{\nu_1}} \int d_1(x_1, y_1)^p + W_{\tilde{G},p}(\mu^{x_1}, \nu^{y_1})^p \, d\pi(x_1, y_1) \\
        &= \inf_{\pi \in \couplings{\mu_1}{\nu_1}} \int d_1(x_1, y_1)^p + AW_{
p}(\mu^{x_1}, \nu^{y_1})^p \, d\pi(x_1, y_1) = AW_p(\mu, \nu)^p,
    \end{align*}
    where the second equality follows by the induction hypothesis and the third by the DPP structure of the adapted Wasserstein distance (which can be seen as a special case of Proposition \ref{dpp_proof}).
\end{proof}

\begin{example}
    We illustrate the (at times opaque) definitions above via a few examples. For different graphs $G= (V, E)$, we explicitly write out the form of the $G$-nested distribution corresponding to some general $\mu \in \Gprobabilities{\mc{X}}$ with the aid of a random variable $X \sim \mu$.

    \begin{itemize}
        \item If $V = \curlybrackets{1, 2, 3}$, $E = \setdef{(i, j) \in V \times V}{i < j}$, then $\mathfrak{N}^G(\mu)$ is the usual nested distribution \cite{Pflug2009NestedDistributions}
        $$\mathfrak{N}^G(\mu) = \law{X_1, \condlaw{X_2, \condlaw{X_3}{X_1, X_2}}{X_1}}.$$
        \item If $V = \curlybrackets{1, 2, 3}$, $E = \curlybrackets{(1,2), (2,3)}$, then
        $$\mathfrak{N}^G(\mu) = \law{X_1, \condlaw{X_2, \condlaw{X_3}{X_2}}{X_1}}.$$
        However, as $\condindep{X_3}{X_2}{X_1}$, $\condlaw{X_3}{X_1, X_2} = \condlaw{X_3}{X_2}$, and we could have simply defined $\mathfrak{N}^G(\mu)$ as in the item above, since they are equal.
        \item If $V = \curlybrackets{1, 2, 3, 4, 5}$, $E = \curlybrackets{(1,2), (1,3), (3,4)}$, then
        $$\mathfrak{N}^G(\mu) = \brackets{\law{X_1, \condlaw{X_2}{X_1}, \condlaw{X_3, \condlaw{X_4}{X_3}}{X_1}}, \law{X_5}}.$$
    \end{itemize}
\end{example}

\subsection{Proof of Statements in Section \ref{section:applications}}
\label{proof_lipschitz_optimisation}

We first prove Theorem \ref{thm:lipschitz_optimisation} on Lipschitz continuity for stochastic team problems.

\begin{proof}[Proof of Theorem \ref{thm:lipschitz_optimisation}]
   Fix $\pi \in \Gbicouplings{\mu}{\nu}$ and let $(X, Y) \sim \pi$. 
   
    Let $\varepsilon > 0$. Then there exist measurable functions $H_k : \mc{X}_k \times \mc{X}_{\parents{k}} \to \mbb{R}$ such that $\alpha = H(Y)$ is an $\eps$-optimiser for $v(\nu)$, i.e., $$v(\nu) \leq \int Q(y, H(y)) \, d\nu(y) \leq v(\nu) + \eps.$$

    By Lemma \ref{applications_lemma} below applied with $A = \curlybrackets{k} \cup \parents{k}$, using the fact that $G$ is transitively closed and hence $\parents{A} \subseteq A$ and $\Gbicouplings{\mu}{\nu} \subseteq \newGbicouplings{\mu}{\nu}$ (see Theorem \ref{thm:lifted}), we know $\condindep{(Y_k, Y_{\parents{k}})}{X_k, X_{\parents{k}}}{X}$. Thus, one may define the measurable functions $\Tilde{H}_k : \mc{X}_k \times \mc{X}_{\parents{k}} \to \mbb{R}$ via
    $$\Tilde{H}_k(x_k, \xparents{k}) \coloneqq \int H_k(y_k, \yparents{k}) \, d\pi^{x_k, x_{\parents{k}}}(y_k, \yparents{k}) = \int H_k(y_k, \yparents{k}) \, d\pi^{x}(y_k, \yparents{k}),$$
    and thus we can write coordinate-wise that $\Tilde{H}(x) = \int H(y) \, d\pi^{x}(y) \in A$.
    
    Then, by convexity properties of $Q$ and Jensen's inequality,
    \begin{align*}
        \int &Q(x, \Tilde{H}(x)) \, d\mu(x) = \int Q\brackets{x, \int H(y) \, d\pi^{x}(y)} \, d\mu(x) \\
        & \leq \iint Q\brackets{x, H(y)} \, d\pi^{x}(y) \, d\mu(x) = \int Q\brackets{x, H(y)} \, d\pi(x,y).
    \end{align*}
    Hence, using the Lipschitz property of $Q$,
    \begin{align*}
        v(\mu) - v(\nu) - \varepsilon \leq \int Q\brackets{x, H(y)} -  Q\brackets{y, H(y)}\, d\pi(x,y) 
        \leq L \int d_{\mc{X}}(x,y) d\pi(x,y).
    \end{align*}
    We can obtain a similar relation by switching the roles of $\mu$ and $\nu$. As $\pi$ and $\eps$ were fixed arbitrarily, the result then follows.
\end{proof}

\begin{lemma}
    \label{applications_lemma}
    Let $G$ be a transitively closed DAG, $\mu \in \Gprobabilities{\mc{X}}, \nu \in \Gprobabilities{\mc{Y}}$, $\pi \in \newGbicouplings{\mu}{\nu}$. Then, if $(X, Y) \sim \pi$, we have that $\condindep{Y_A}{X_A, X_{\parents{A}}}{X}$, for all $A \subseteq V$.
\end{lemma}
\begin{proof}
Let $T : \hat{\mc{X}} \to \hat{\mc{Y}}$ be a $G$-biadapted map such that $\pi = \pushforward{\proj{\mc{X} \times \mc{Y}}}{\hat{\pi}}$, where $\hat{\pi}\coloneqq \pushforward{(\mathrm{Id}, T)}{\hat{\mu}}$. Then if $(X,U,Y,V) \sim \hat{\pi}$, $(X,Y) \sim \pi$. Moreover, for all $i \in V$, $Y_i = T_i(X_i, U_i, X_{\parents{i}}, U_{\parents{i}})$, and thus $Y_A = T_A(X_A, X_{\parents{A}}, U_A, U_{\parents{A}})$, where we make a slight abuse of notation and add in inputs that might not be used. As $U$ is independent of $X$, this implies our conclusion.
\end{proof}

Finally, we end with the following simple lemma used in the proof of Corollary \ref{corollary:ate}. We work under the assumptions and notation of Proposition \ref{prop:continuity_ate} and Corollary \ref{corollary:ate}, presented in Section \ref{section:applications}.

\begin{lemma}
    \label{lemma:ate}
    Let $G = (V, E)$ and $G' = (V, E')$ be DAGs such that $E \subseteq E'$. Then $\specificGprobabilities{G, \delta}{\mc{X}} \subseteq \specificGprobabilities{G', \delta}{\mc{X}}$ and $\psi^{\mu, G} = \psi^{\mu, G'}$ for all $\mu \in \specificGprobabilities{G, \delta}{\mc{X}}$.
\end{lemma}
\begin{proof}
    We denote by $\specificGparents{G}{j}$ the parents of node $j$ in $G$ and $\specificGparents{G'}{j}$ the parents of node $j$ in $G'$. As $E \subseteq E'$, clearly $\specificGparents{G}{j} \subseteq \specificGparents{G'}{j}$, and $\specificGparents{G'}{j}$ cannot contain a descendant of $j$ in $G$, as otherwise $G'$ would be cyclic.

    Let $\mu \in  \specificGprobabilities{G, \delta}{\mc{X}}$. Then, as $\mu$ is $G$-compatible, by the remark above $\condindep{X_j}{X_{\specificGparents{G}{j}}}{X_{\specificGparents{G'}{j}}}$ holds under $\mu$. Consequently, $\mu(x_j =1 \, \vert \, x_{\specificGparents{G'}{j}}) = \mu(x_j =1 \, \vert \, x_{\specificGparents{G}{j}}) \in [\delta, 1-\delta]$, and so $\mu \in\specificGprobabilities{G', \delta}{\mc{X}}$.

    To show $\psi^{\mu, G} = \psi^{\mu, G'}$, it is enough to note that $\specificGparents{G'}{j}$ satisfies the back-door adjustment criterion \cite[Definition 3.3.1]{Pearl2009Causality} for $(X_j,X_k)$ in $G$ and apply \cite[Theorem 3.3.2]{Pearl2009Causality}. We offer a direct proof for completeness. 

    Let $A \coloneq \specificGparents{G'}{j} \setminus \specificGparents{G}{j}$. Then:
    \begin{align*}
        \int_{\mc{X}_{\specificGparents{G'}{j}}} &\int_{\mc{X}_k} x_k \, \mu(dx_k \,\vert\,x_j=1, x_{\specificGparents{G'}{j}}) \, \mu(dx_{\specificGparents{G'}{j}}) = \\
        &= \int_{\mc{X}_{\specificGparents{G}{j}}} \int_{\mc{X}_A} \int_{\mc{X}_k} x_k \, \mu(dx_k \,\vert\,x_j=1, x_{\specificGparents{G}{j}}, x_A) \, \mu(dx_A \,\vert\, x_{\specificGparents{G}{j}}) \, \mu(dx_{\specificGparents{G}{j}}) \\
        &= \int_{\mc{X}_{\specificGparents{G}{j}}} \int_{\mc{X}_A} \int_{\mc{X}_k} x_k \, \mu(dx_k \,\vert\,x_j=1, x_{\specificGparents{G}{j}}, x_A) \, \mu(dx_A \,\vert\, x_j =1, x_{\specificGparents{G}{j}}) \, \mu(dx_{\specificGparents{G}{j}}) \\
        &=  \int_{\mc{X}_{\specificGparents{G}{j}}} \int_{\mc{X}_k} x_k \, \mu(dx_k \,\vert\,x_j=1, x_{\specificGparents{G}{j}}) \, \mu(dx_{\specificGparents{G}{j}}),
    \end{align*}
where the second equality follows since $\condindep{X_A}{X_{\specificGparents{G}{j}}}{X_j}$ under $\mu$ (as above). We can obtain a symmetric result for $x_j=0$, and thus we see that the average treatment effects under $G$ and $G'$ agree.
\end{proof}
\section*{Acknowledgements}
The authors would like to thank Yifan Jiang, Fang Rui Lim, Ivan Guo, Samuel N. Cohen, Stephan Eckstein, Mathias Beiglb\"ock, Gudmund Pammer, Stefan Schrott, Daniel Lacker, and Serdar Y\"uksel for interesting conversations regarding this work and related topics. Vlad Tuchilu\textcommabelow{s}' research is supported by the EPSRC Centre for Doctoral Training in Mathematics of Random Systems: Analysis, Modelling and Simulation (EPSRC Grant EP/S023925/1).

\bibliographystyle{plainnat_less} 
\bibliography{refs}  

\appendix
\section{Technical Results}
\label{appendixA}

In this appendix we prove some measure-theoretic results used in Section \ref{proof_section}. 

\begin{proposition}
\label{marginal_of_coupling}
    Let $G = (V, E)$ be a sorted DAG. Suppose $A \subseteq V$ is such that $\parents{i} \subseteq A$ for all $i \in A$. Let $\mu \in \Gprobabilities{\mc{X}}$, $\nu \in \Gprobabilities{\mc{Y}}$, $\pi \in \Gcouplings{\mu}{\nu}$. Let $\mu_A \coloneqq \pushforward{\proj{\mc{X}_A}}{\mu}$, $\nu_A \coloneqq \pushforward{\proj{\mc{Y}_A}}{\nu}$, $\pi_A \coloneqq \pushforward{\proj{\mc{X}_A \times \mc{Y}_A}}{\pi}$. 
    
    Then, if we define the graph $G_A \coloneqq (A, E \cap (A \times A))$, obtained by deleting from $G$ all the vertices that do not belong in $A$ along with any incident edges, we have that $\pi_A \in \specificGcouplings{G_A}{\mu_A}{\nu_A}$.
\end{proposition}
\begin{proof}
    We may assume without loss of generality via a relabelling that $A = \curlybrackets{1, \ldots, m}$$ \subseteq \curlybrackets{1, \ldots, n} = V$ for some $n \geq m \geq 1$. As $\parents{i} \subseteq A$ for all $i \in A$, we may ensure that $G$ remains sorted under this relabelling. Once this has been done, the statement is a simple consequence of \cite[Theorem 3.4(iv)]{Eckstein2023CausalGraphs}, as outlined below.

    By \cite[Theorem 3.4(iv)]{Eckstein2023CausalGraphs}, $\pi$ may be disintegrated as
    \begin{align*}
        \pi(dx, dy) &= \bigotimes_{i=1}^{n} \pi(dx_i, dy_i \mid  \xparents{i}, \yparents{i}) \\
        &=  \brackets{\bigotimes_{i=1}^{m} \pi(dx_i, dy_i \mid \xparents{i}, \yparents{i})} \otimes \brackets{\bigotimes_{i=m+1}^{n} \pi(dx_i, dy_i \mid \xparents{i}, \yparents{i})},
    \end{align*}
    where $\pi(dx_i \mid \xparents{i}, \yparents{i}) = \mu(dx_i \mid \xparents{i})$ $\pi$-a.s. for all $1 \leq i \leq n$. 

    Clearly now $\pi_A(dx_A, dy_A) = \bigotimes_{i=1}^{m} \pi(dx_i, dy_i \mid \xparents{i}, \yparents{i})$ with the kernels satisfying $\pi(dx_i \mid \xparents{i}, \yparents{i}) = \mu(dx_i \mid \xparents{i}) = \mu_A(dx_i \mid \xparents{i})$ $\pi$-a.s. so also $\pi_A$-a.s. for all $1 \leq i \leq m$. Thus, $\pi_A \in \specificGcouplings{G_A}{\mu_A}{\nu_A}$.
\end{proof}

In what follows, we denote the DAG with 2 vertices connected by a directed arrow by $\Garrow \coloneqq (\curlybrackets{1, 2}, \curlybrackets{(1,2)})$. 

\begin{lemma}
\label{marginal_of_gluing_easy}
    Suppose $\mu \in \specificGprobabilities{\Garrow}{\mc{X}_{1:2}}$, $\nu \in \specificGprobabilities{\Garrow}{\mc{Y}_{1:2}}$, $\eta \in \specificGprobabilities{\Garrow}{\mc{Z}_{1:2}}$, $\pi \in \specificGcouplings{\Garrow}{\mu}{\nu}$, $\tilde{\pi} \in \specificGcouplings{\Garrow}{\nu}{\eta}$. Let $\pi_1 \coloneqq \pushforward{\proj{\mc{X}_1 \times \mc{Y}_1}}{\pi}$, and $\tilde{\pi}_1 \coloneqq \pushforward{\proj{\mc{Y}_1 \times \mc{Z}_1}}{\tilde{\pi}}$. Then 
    $$\pushforward{\proj{(\mc{X}_1 \times \mc{Y}_1 \times \mc{Z}_1)}}{\brackets{\bigglue{\pi}{\tilde{\pi}}}} = \bigglue{\pi_1}{\tilde{\pi}_1}.$$
\end{lemma}
\begin{proof} Let $(X, Y, Z) \sim \bigglue{\pi}{\tilde{\pi}}$. Clearly, $(X_1, Y_1) \sim \pi_1$, $(Y_1, Z_1) \sim \tilde{\pi}_1$. It suffices to show that $\condindep{X_1}{Y_1}{Z_1}$, since this will imply that $(X_1, Y_1, Z_1) \sim \bigglue{\pi_1}{\tilde{\pi}_1}$, which is what we needed to show.

As $\tilde{\pi} \in \specificGcouplings{\Garrow}{\nu}{\eta}$, we know $\condindep{Z_1}{Y_1}{Y}$. By properties of $\otimesdot$, $\condindep{X}{Y}{Z}$, so particularly also $\condindep{Z_1}{Y}{X_1}$. Combining these two relations via the chain rule \cite[Theorem 8.12]{Kallenberg2021FoundationsThird}, we get $\condindep{Z_1}{Y_1}{(X_1, Y)}$, which implies $\condindep{Z_1}{Y_1}{X_1}$.
\end{proof}

\begin{proposition}
\label{marginal_of_gluing}
    Let $G = (V, E)$ be a sorted DAG. Suppose $A \subseteq V$ is such that $\parents{i} \subseteq A$ for all $i \in A$. Let $\mu \in \Gprobabilities{\mc{X}}$, $\nu \in \Gprobabilities{\mc{Y}}$, $\eta \in \Gprobabilities{\mc{Z}}$, $\pi \in \Gcouplings{\mu}{\nu}$, $\tilde{\pi} \in \Gcouplings{\nu}{\eta}$ and $\pi_A \coloneqq \pushforward{\proj{\mc{X}_A \times \mc{Y}_A}}{\pi}$, $\tilde{\pi}_A \coloneqq \pushforward{\proj{\mc{Y}_A \times \mc{Z}_A}}{\tilde{\pi}}$. Then 
    $$\pushforward{\proj{(\mc{X}_A \times \mc{Y}_A \times \mc{Z}_A)}}{\brackets{\bigglue{\pi}{\tilde{\pi}}}} = \bigglue{\pi_A}{\tilde{\pi}_A}.$$
\end{proposition}
\begin{proof}
    This is a simple consequence of Lemma \ref{marginal_of_gluing_easy}. Indeed, we may consider the product space $\mc{X}_A$ as a single Polish space $\tilde{\mc{X}}_1$, the product space $\mc{X}_{V \setminus A}$ as a single Polish space $\tilde{\mc{X}}_2$, with similar definitions for $\tilde{\mc{Y}}_i$, $\tilde{\mc{Z}}_i$. Then, as $\parents{i} \subseteq A$ for all $i \in A$, $\mu \in \specificGprobabilities{\Garrow}{\tilde{\mc{X}}_{1:2}}$ when viewed as a measure on $\tilde{\mc{X}}_1 \times \tilde{\mc{X}}_2$, $\pi \in \specificGcouplings{\Garrow}{\mu}{\nu}$ when viewed as a measure on $(\tilde{\mc{X}}_1 \times \tilde{\mc{X}}_2) \times (\tilde{\mc{Y}}_1 \times \tilde{\mc{Y}}_2)$, with similar relations holding for $\nu, \eta, \tilde{\pi}$. We may now apply Lemma \ref{marginal_of_gluing_easy} to get our result.
\end{proof}

\begin{lemma}
\label{aw_independent}
    Let $1 \leq m < n$. Suppose $\mu, \nu \in \Pp{p}{\mc{X}_{1:n}}$ are such that they can be written as product measures $\mu = \mu_1 \otimes \mu_2$, $\nu = \nu_1 \otimes \nu_2$ for some $\mu_1, \nu_1 \in \Pp{p}{\mc{X}_{1:m}}$, $\mu_2, \nu_2 \in \Pp{p}{\mc{X}_{m+1:n}}$. Then,
    $$AW_p(\mu, \nu)^p = AW_p(\mu_1, \nu_1)^p + AW_p(\mu_2, \nu_2)^p.$$
\end{lemma}
\begin{proof}
    Let $X \sim \mu$, $Y \sim \nu$. The adapted Wasserstein distance between $\mu$ and $\nu$ is the distance between their associated nested distributions \cite{Pflug2009NestedDistributions}. 
    
    Let $\tilde{X}_n \coloneqq \condlaw{X_n}{X_{1:n-1}}$ and for $n > k \geq 1$, $\tilde{X}_k \coloneqq \condlaw{X_k, \tilde{X}_{k+1}}{X_{1:k-1}}$. Then $\tilde{X}_1$ is the nested distribution of $\mu$. As $X_{1:m} \indep X_{m+1:n}$, we have for $k \geq m +1$ that $\tilde{X}_k = \condlaw{X_k, \tilde{X}_{k+1}}{X_{m+1:k-1}}$. In particular, $\tilde{X}_{m+1} = \law{X_{m+1}, \tilde{X}_{m+2}}$ is a constant random variable, equal to the corresponding nested distribution of $\mu_2$. Hence, now $\tilde{X}_1$ is equal to the nested distribution corresponding to the law of the vector below, viewed as an $m$-step stochastic process
    $$\hat{X} \coloneqq (X_1, \ldots, X_{m-1}, (X_m, \tilde{X}_{m+1})).$$
    Symmetrical statements hold for $\nu$. Thus
    \begin{align*}
        AW_p(\mu, \nu)^p &= d(\tilde{X}_1, \tilde{Y}_1)^p = AW_p\brackets{\law{\hat{X}}, \law{\hat{Y}}}^p \\
        &= AW_p\brackets{\law{X_{1:m}}, \law{Y_{1:m}}}^p + d(\tilde{X}_{m+1}, \tilde{Y}_{m+1})^p \\ &= AW_p(\mu_1, \nu_1)^p + AW_p(\mu_2, \nu_2)^p,
    \end{align*}
    as the constant random variables $\tilde{X}_m, \tilde{Y}_m$ simply add a constant term to the cost function.
\end{proof}

\section{Weak Closedness of $G$-compatible Measures}
\label{section_closedness}

This section records some complementary results for the proof of Proposition \ref{prop:P_G_closed}. 

\begin{example}
\label{closedness_counterexamples}
    For completeness, we repeat the example of \cite[Proposition 3.3(iv)]{Eckstein2023CausalGraphs} for $\Gmarkov$: $$\specificGprobabilities{\Gmarkov}{\mbb{R}^3} \ni \frac{1}{2}\brackets{\delta_{(1, \frac{1}{k}, 1)} + \delta_{(-1, \frac{-1}{k}, -1)}} \xRightarrow{k \to \infty} \frac{1}{2}\brackets{\delta_{(1, 0, 1)} + \delta_{(-1, 0, -1)}} \notin \specificGprobabilities{\Gmarkov}{\mbb{R}^3}.$$

    This can easily be adapted to an example for $\Gsplit$ by changing the order of the coordinates:
    $$\specificGprobabilities{\Gsplit}{\mbb{R}^3} \ni \frac{1}{2}\brackets{\delta_{(\frac{1}{k}, 1, 1)} + \delta_{(\frac{-1}{k}, -1, -1)}} \xRightarrow{k\to \infty} \frac{1}{2}\brackets{\delta_{(0, 1,1)} + \delta_{(0, -1, -1)}} \notin \specificGprobabilities{\Gsplit}{\mbb{R}^3}.$$

    Both of the examples above can be extended to graphs $G$ which have $\Gmarkov$ as a proper subgraph and thus lack the transitive closure property, or which have $\Gsplit$ as proper subgraphs and thus contain a split. This can be done by fixing all other coordinates to be equal to zero in both the measures in the sequence and in the limit.
\end{example}

\begin{lemma}
    \label{closedness_independent}
    Suppose $G = G_A \oplus G_B$ and that both $\specificGprobabilities{G_A}{\mc{X}_A}$ and $\specificGprobabilities{G_B}{\mc{X}_B}$ are weakly closed. Then $\Gprobabilities{\mc{X}_A \times \mc{X}_B}$ is weakly closed.
\end{lemma}
\begin{proof}
    Clearly, a measure $\mu \in \Gprobabilities{\mc{X}_A \times \mc{X}_B}$ if and only if it can be written as $\mu = \mu_A \otimes \mu_B$ with its marginals $\mu_A, \mu_B$ having the property that $\mu_A \in \specificGprobabilities{G_A}{\mc{X}_A}$ and $\mu_B \in \specificGprobabilities{G_B}{\mc{X}_B}$. 

    Let then $\brackets{\mu_A^{(k)} \otimes \mu_B^{(k)}}_{k \geq 1}$ be a sequence in $\Gprobabilities{\mc{X}_A \times \mc{X}_B}$ converging weakly to some $\mu \in \probabilities{\mc{X}_A \times \mc{X}_B}$. As weak convergence implies weak convergence of marginals as well, we have $\mu_A^{(k)} \Rightarrow \mu_A$ and $\mu_B^{(k)} \Rightarrow \mu_B$. Since $\specificGprobabilities{G_A}{\mc{X}_A}$ is weakly closed, we see $\mu_A \in \specificGprobabilities{G_A}{\mc{X}_A}$, and similarly $\mu_B \in \specificGprobabilities{G_B}{\mc{X}_B}$.
    
    Then, for all $f \in \Cb{\mc{X}_A}$, $g \in \Cb{\mc{X}_B}$:
    \begin{align*}
    \int f(x_A)g(x_B)\, d\mu &= \lim_{k\to\infty} \int f(x_A)g(x_B)\, d\mu_A^{(k)} \otimes \mu_B^{(k)} \\
    &= \lim_{k\to\infty} \int f(x_A) \, d\mu_A^{(k)} \int g(x_B) \, d\mu_B^{(k)} = \int f(x_A) \, d\mu_A \int g(x_B) \, d\mu_B,
    \end{align*}
    and thus $\mu = \mu_A \otimes \mu_B$, and in particular $\mu$ is $G$-compatible.
\end{proof}

\begin{lemma}
    \label{closedness_add}
    Suppose $\Gprobabilities{\mc{X}_{1:n}}$ is weakly closed. Then $\specificGprobabilities{G^+}{\mc{X}_{1:n+1}}$ is weakly closed.
\end{lemma}
\begin{proof}
    This follows trivially by the fact that weak convergence implies weak convergence of marginals, and noting that a measure $\mu \in \probabilities{\mc{X}_{1:n+1}}$ is $G^+$-compatible if and only if its $\mc{X}_{1:n}$ marginal $\mu_{1:n}$ is $G$-compatible.
\end{proof}
\section{DAGs Without Transitive Closure Property}
\label{section_nonHP}

We provide here some technical results used in the proof of Proposition \ref{prop:nonHP_metric}.

\begin{lemma}
        \label{lem:nonHP_inclusion}
    For a DAG $G$, and $\mu \in \Gprobabilities{\mc{X}}$, $\nu \in \Gprobabilities{\mc{Y}}$, we have that $\specificGbicouplings{\Gbar}{\mu}{\nu} \subseteq \tildeGbicouplings{\mu}{\nu}$.
\end{lemma}

\begin{proof}
    Note that since $G$ is acyclic, the set of parents of node $i$ in $\Gbar$ is $\hist{i}$, the set of ancestors of node $i$ in $G$. 

    Let $\pi \in \specificGbicouplings{\Gbar}{\mu}{\nu}$, and $(X,Y) \sim \pi$. By repeated applications of Lemma \ref{independent_construction} one can construct $\mathrm{Unif}[0,1]$ random variables $U_1, \ldots, U_n, V_1, \ldots, V_n$ such that $(X,Y), U_1, \ldots, U_n$ are independent, $(X,Y), V_1, \ldots, V_n$ are independent, and $\sigma(U_i, V_i) = \sigma(U_i, U_{\parents{i}}, V_{\parents{i}}, X_{\parents{i}}, Y_{\parents{i}}) = \sigma(V_i, U_{\parents{i}}, V_{\parents{i}}, X_{\parents{i}}, Y_{\parents{i}})$ for all $i \in V$. By induction, the last statement implies that $\sigma(U_i, V_i) = \sigma(U_i, U_{\hist{i}}, X_{\hist{i}}, Y_{\hist{i}}) = \sigma(V_i, V_{\hist{i}}, X_{\hist{i}}, Y_{\hist{i}})$ for all $i \in V$.

    By construction, $\Hat{X} \coloneq (X, U) \sim \Hat{\mu}$, $\Hat{Y} \coloneq (Y, V) \sim \Hat{\nu}$. Thus, it suffices to show that $\law{\Hat{X}, \Hat{Y}} \in \Gbicouplings{\Hat{\mu}}{\Hat{\nu}}$. We focus on the causal direction. The anticausal one will then follow by symmetry. We need to show that $\condindep{\Hat{Y}_i}{\Hat{X}_i, \Hat{X}_{\parents{i}}, \Hat{Y}_{\parents{i}}}{(\Hat{X}, \Hat{Y}_{1:i-1})}$, i.e., by expanding out that
    $$\condindep{(Y_i, V_i)}{X_i, X_{\parents{i}}, Y_{\parents{i}}, U_i, U_{\parents{i}}, V_{\parents{i}}}{(X, Y_{1:i-1}, U, V_{1:{i-1}})}.$$

    But, by the properties above, $\sigma({X_i, X_{\parents{i}}, Y_{\parents{i}}, U_i, U_{\parents{i}}, V_{\parents{i}}}) = \sigma(X_i, X_{\hist{i}}, Y_{\hist{i}}, U_i, U_{\hist{i}})$, and $\sigma({X, Y_{1:i-1}, U, V_{1:{i-1}}}) = \sigma(X, Y_{1:i-1}, U)$. Thus, conditional independence relation above becomes
    $$\condindep{(Y_i, V_i)}{X_i, X_{\hist{i}}, Y_{\hist{i}}, U_i, U_{\hist{i}}}{(X, Y_{1:i-1}, U)}.$$

    Since $V_i$ is $\sigma(U_i, U_{\hist{i}}, X_{\hist{i}}, Y_{\hist{i}})$-measurable, this is equivalent to proving 
    \begin{equation}
        \label{non_history_equation}
    \condindep{Y_i}{X_i, X_{\hist{i}}, Y_{\hist{i}}, U_i, U_{\hist{i}}}{(X, Y_{1:i-1}, U)}.
    \end{equation}

    As $\pi \in \specificGbicouplings{\Gbar}{\mu}{\nu}$, we know that 
    $\condindep{Y_i}{X_i, X_{\hist{i}}, Y_{\hist{i}}}{(X, Y_{1:i-1})}.$ By construction, we also know that $U \indep (X,Y)$, and thus $\condindep{Y_i}{X, Y_{1:i-1}}{U}$. Combining these using the chain rule of conditional independence \cite[Theorem 8.12]{Kallenberg2021FoundationsThird}, we see that 
    $$\condindep{Y_i}{X_i, X_{\hist{i}}, Y_{\hist{i}}}{(X, Y_{1:i-1}, U)},$$
    which is a stronger statement than (\ref{non_history_equation}).
\end{proof}

\begin{lemma}
    \label{independent_construction}
    Given random variables $\Xi, \Upsilon$ defined on the same probability space and taking values in standard Borel spaces, we may define (after possibly enlarging the sample space) $\mathrm{Unif}[0,1]$ random variables $U, V$ such that $U$ is independent of $(\Xi,\Upsilon)$, $V$ is independent of $(\Xi,\Upsilon)$ but $\sigma(U, V) = \sigma(U, \Xi) = \sigma(V, \Xi)$.
\end{lemma}
\begin{proof}
By the classification of standard Borel spaces \citep[Corollary 7.16.1]{BerstekasShreve1978}, we may assume without loss of generality that $\Xi$ takes values in the Cantor space $\curlybrackets{0,1}^{\mbb{N}}$. We write $\Xi = (\Xi_i)_{i \geq 1}$. 

Define i.i.d. random variables $U_i \sim \frac{1}{2}(\delta_0 + \delta_1)$ independent of $(\Xi, \Upsilon)$. Defining for $i \geq 1$ $V_i = \Xi_i (1-U_i) + (1 - \Xi_i)U_i$, we see that the $V_i$ are also i.i.d. random variables with law $\frac{1}{2}(\delta_0 + \delta_1)$ and which are independent of $(\Xi, \Upsilon)$. Moreover, it is easy to see that $\sigma((U_i)_{i \geq 1}, (V_i)_{i \geq 1}) = \sigma((U_i)_{i \geq 1}, \Xi) = \sigma((V_i)_{i \geq 1}, \Xi)$.

Finally, let $U = \sum_{i=1}^{\infty}U_i 2^{-i}$, $V = \sum_{i=1}^{\infty}V_i 2^{-i}$. Then $U, V \sim \mathrm{Unif}[0,1]$, $\sigma(U) = \sigma((U_i)_{i \geq 1})$, and $\sigma(V) = \sigma((V_i)_{i \geq 1})$. The conclusion then follows from the properties of the $U_i, V_i$ given above.
\end{proof}

The example below shows that generally $\specificGbicouplings{\Gbar}{\mu}{\nu} \nsubseteq \newGbicouplings{\mu}{\nu}$, justifying the need to relax $\newGbicouplings{\mu}{\nu}$ further to the full generality of $\tildeGbicouplings{\mu}{\nu}$ for Proposition \ref{prop:nonHP_metric}.

\begin{example}
\label{example:generalised_different}
    Consider $\Gmarkov = (\curlybrackets{1,2,3}, \curlybrackets{(1,2), (2,3)})$, and $\mc{X} = \mbb{R}^3$. Let $\mu \coloneq (\frac{1}{2}\delta_0 + \frac{1}{2}\delta_1) \otimes \delta_0 \otimes (\frac{1}{2}\delta_0 + \frac{1}{2}\delta_1) \in \specificGprobabilities{\Gmarkov}{\mc{X}}$. Clearly, $\closure{\Gmarkov} = (\curlybrackets{1,2,3}, \curlybrackets{(1,2), (2,3), (1,3)})$.

    Let $X \sim \mu$ and $Y_1 \sim \mu_1$ be independent, $Y_2 = 0$, and $Y_3 = X_3 \mathbbm{1}_{X_1=Y_1} + (1-X_3)\mathbbm{1}_{X_1 \neq Y_1}$. Then $\pi \coloneq \law{(X,Y)} \in \specificGbicouplings{\closure{\Gmarkov}}{\mu}{\mu}$. 
    
    Assuming for a contradiction that $\pi \in \newspecificGbicouplings{\Gmarkov}{\mu}{\mu}$, then we can find $U_1, U_2, U_3$ i.i.d. $\mathrm{Unif}[0,1]$ random variables independent of $X$, and measurable functions $T_1 : \mbb{R}^2 \to \mbb{R}$, $T_3 : \mbb{R}^4 \to \mbb{R}$ such that $Y_1 = T_1(X_1, U_1)$, and $Y_3 = T_3(X_3, U_3, X_2, U_2)$. Since, by construction $(X_3, U_3, X_2, U_2) \indep (X_1, U_1)$, we see that in this case $(X_3, Y_3) \indep (X_1, Y_1)$. But this is clearly not true under $\pi$. Thus, $\pi \notin \newspecificGbicouplings{\Gmarkov}{\mu}{\mu}$.
\end{example}

\end{document}